\documentclass[a4paper,11pt]{article}
\usepackage[english]{babel}
\usepackage[utf8]{inputenc}
\usepackage{paralist,url,verbatim,anysize,enumitem}
\usepackage{amscd}
\usepackage{fancyhdr}
\usepackage{mathrsfs, mathtools}
\usepackage{xfrac,faktor}
\usepackage[e]{esvect}
\usepackage{amsmath,amsthm,amssymb,amsfonts}
\usepackage{graphicx, graphics}
\usepackage[bookmarksopen=true]{hyperref}
\hypersetup{colorlinks=true, citecolor=blue}
\usepackage{bbm}
\usepackage[font=small,labelfont=bf]{caption}
\usepackage{subcaption}

\usepackage{tikz}
\usetikzlibrary{tikzmark,calc,arrows.meta}

\usepackage{tikz} %for the command circled

\theoremstyle{plain}
\newtheorem{theorem}{\bf Theorem}[]

\newtheorem{corollary}[theorem]{Corollary}
\newtheorem{lemma}[theorem]{Lemma}
\newtheorem{proposition}[theorem]{Proposition}

\newtheorem{thmnonumber}{\bf Main Theorem}

\theoremstyle{definition}
\newtheorem{example}[theorem]{Example}
\newtheorem{nonexample}[theorem]{Non-Example}
\newtheorem{remark}[theorem]{Remark}

\newtheorem{definition}[theorem]{Definition}

\newcommand{\del}{\operatorname{del}}

\newcommand{\NN}{\mathbb{N}}

\newcommand\eqdef{\mathrel{\overset{\makebox[0pt]{\mbox{\normalfont\tiny\sffamily def}}}{=}}}

\newcommand{\link}{\operatorname{link} }
\newcommand{\Cl}{\operatorname{Cl}}
\newcommand{\skel}{\operatorname{skel}}

\newcommand{\str}{\operatorname{star}}
\newcommand{\Star}{\operatorname{star}}

\newcommand{\abdel}{\operatorname{abdel}}
\newcommand{\kk}{\mathbb K}

\newcommand{\pureskel}{\operatorname{pure\text{-}skel}}
\newcommand{\widetH}{\widetilde H}

\definecolor{mypink}{RGB}{215, 5, 234}
\definecolor{lemonchiffon}{RGB}{255, 250, 205}

\begin{document}

\author{Bruno Benedetti\thanks{University of Miami, USA} \and Marta Pavelka\thanks{University of Copenhagen, Denmark}}

\date{\small 2026}
\title{\vskip -1.5 cm Skeleton Chordalities}
\maketitle

\vskip -3mm
%%%%%%%%%%%%%%%%%%%%%%%%%%%%%%%%%%%%%%%%%%%%%%%%%
%%%%%%%%%%%%%%%%%%%%%%%%%%%%%%%%%%%%%%%%%%%%%%%%%
\begin{abstract}
We study new higher-dimensional analogs of graph chordality, and review the existing ones. Our main results for simplicial complexes are: 
\begin{compactenum}[\rm (1)]
\item $\Delta$ skeleton-E-chordal $\Rightarrow$ $\Delta^\vee$ vertex-decomposable $\Rightarrow$ $\Delta$ skeleton-clique-chordal.\\
Moreover, for subflag complexes, $\Delta$ skeleton-E-chordal $\Longleftrightarrow$ $\Delta^\vee$ vertex-decomposable.\\
(For $d=1$ this boils down to  ``$G$ chordal $\Longleftrightarrow$ $G^\vee$ vertex-decomposable'', a result closely related to Fr\"oberg's theorem.) 
\item For subflag complexes, $\Delta$ is skeleton-E-chordal $\Longleftrightarrow$ either its $1$-skeleton is complete, or it splits as $\Delta = \Delta_1  \cup \Delta_2$, with each $\Delta_i$ a skeleton-E-chordal induced proper subcomplex of $\Delta$, and with $\Delta_1 \cap \Delta_2$ a complex whose $1$-skeleton is a clique. 
(This generalizes ``$G$ chordal  $\Longleftrightarrow$ either $G$ is complete, or it splits as union of two chordal proper subgraphs that intersect in a  clique'').
\item $\Delta$  skeleton-E-chordal $\Longleftrightarrow$ every nonempty induced subcomplex of $\Delta$ has a skeleton-E-simplicial
vertex.\\
(Generalizes ``$G$ chordal $\Leftrightarrow$ every nonempty induced subgraph has a simplicial
vertex''.)
\item $\Delta$ underclosed $\Rightarrow$ $\Delta$ skeleton-weakly-chordal and weakly-closed.\\
(Generalizes ``$G$ interval $\Rightarrow$ $G$ chordal and co-comparability''.) 
\item All pure E-chordal complexes are vertex-chordal; all  pure mid-chordal complexes are weakly-vertex-chordal; all pure very-weakly-chordal complexes are weakly-ridge-chordal.\\
(This expands  Bigdeli, Yazdan-Pour and Zaare-Nahandi's work on  ridge-chordality.)
\end{compactenum}

\end{abstract}
%%%%%%%%%%%%%%%%%%%%%%%%%%%%%%%%%%%%%%%%%%%%%%%%%
%%%%%%%%%%%%%%%%%%%%%%%%%%%%%%%%%%%%%%%%%%%%%%%%%

\enlargethispage{3mm}

{\small
\tableofcontents 
}

\newpage

%%%%%%%%%%%%%%%%%INTRODUCTION%%%%%%%%%%%%%%%%%%%%%%%%%%
%%%%%%%%%%%%%%%%%%%%%%%%%%%%%%%%%%%%%%%%%%%%%%%%%
\section{Introduction}
%%%%%%%%%%%%%%%%%%%%%%%%%%%%%%%%%%%%%%%%%%%%%%%%%
%%%%%%%%%%%%%%%%%%%%%%%%%%%%%%%%%%%%%%%%%%%%%%%%%
A graph is called \emph{chordal} if for every simple cycle in $G$ of length $>3$, there is an edge in $G$ (called a \emph{chord}) that connects two non-adjacent vertices of the cycle. In other words, chordality is the lack of induced cycles, with the exception of  boundaries of triangles, which are allowed.

Perhaps because of the simplicity of the definition, this graph property has been very popular since its introduction, in 1958 \cite{HS58}. Hajnal and Sur\'anyi proved that all interval graphs are chordal \cite{HS58}; later Berge \cite{Ber61} showed that all chordal graphs are perfect. Both these implications are strict: All trees are chordal and all even cycles are perfect, but some trees are not interval graphs, and obviously even cycles are  not chordal.  In 1961, Dirac characterized chordal graphs as the graphs whose minimal vertex separators are cliques \cite{Dir61}. From Dirac's result stemmed at least three more characterizations of chordality: one via vertex labelings (Theorem~\ref{thm:CharVertexLabeling}), one via simplicial vertices (Theorem \ref{thm:charsimplicial}), and one via clique decompositions (Theorem \ref{thm:DiracSplits}). Dirac is also credited with the word `chord'.\footnote{For chordal graphs, Dirac kept using the expression ``rigid circuit graphs'', a calque of the German ``starren Kreise'' from Berge \cite{Ber61}. The name `chordal graphs'  is present in the 1972 work by Gavril \cite{Gav72}.}

In the second half of the twentieth century, the importance of chordality within graph theory grew further, as many theoretically difficult problems, like finding a maximum clique or computing the chromatic polynomial, turned out to be easily solvable when treated under the aegis of chordality. Thanks to the characterization via vertex labelings, chordality can be recognized in linear time \cite{RTL76}. For an introduction to the discoveries of that period, we recommend Golumbic's chapter  on what he calls ``triangulated graphs'' \cite[Chapter 4]{Gol80}.

In 1990, Fr\"oberg discovered a further characterization via commutative algebra. Given a graph $G$, it is possible to form its edge ideal, whose generators are the quadratic monomials $x_ix_j$ corresponding to the edges $ij$ of $G$. Like all ideals, it can be studied via (minimal) free resolutions \cite{MS05}. If $\overline{G}$ denotes the complement graph, Fr\"oberg proved that 

\begin{quote} \emph{$G$ is chordal if and only if the edge ideal of $\overline{G}$ has a linear resolution} \cite{Fro90}.
\end{quote}

The 21st century brought the new goal of extending the notion of chordality to hypergraphs or simplicial complexes. The most natural attempt is 
to study $d$-dimensional complexes that lack induced $d$-spheres other than the boundary of the $(d+1)$-simplex (cf.~the  `geometric chordality' notion in the present paper). Unfortunately, though, this notion is too weak. 

In 2010, Emtander \cite{Emt10} formulated a combinatorial strengthening of the previous geometric idea, now known as E-chordality.  
Namely, a simplicial complex $\Delta$ is  \emph{E-chordal} if it has a vertex labeling such that for every two facets $f, g$ in $\Delta$ with same cardinality and same maximum vertex, $\Delta$ contains every subset $h$ of $f \cup g$ of the same size as $f$. 
To understand the importance of Emtander's work, we recall two combinatorial concepts. By $\operatorname{pure-skel}_k (\Delta)$ we mean the pure simplicial complex spanned by the $k$-faces of $\Delta$. The Alexander dual $\Delta^\vee$ of a simplicial complex $\Delta$ is formed by the complements of the non-faces of $\Delta$.  Emtander's main result can be phrased as follows, where $d$ is the dimension and $n$ the number of vertices of $\Delta$ \cite{Emt10}:  

\begin{quote} {\em If $\Delta$ is pure E-chordal, then $\operatorname{pure-skel}_{n-d-2}(\Delta^\vee)$ is Cohen--Macaulay}. \end{quote}

This conclusion is equivalent to a certain ideal having a linear resolution, namely, the ideal $I_d(\Delta)$ generated by the missing $d$-faces of $\Delta$. Hence for $d=1$, Emtander's result recovered one direction of Fr\"oberg's theorem. 

In 2011, Woodroofe \cite{Woo11} introduced ``W-chordality'', an independent, more technical notion of chordality, based on simplicial vertices and minors.  
In \cite{Woo11} he proved: 
\begin{quote}
\emph{ 
If $\Delta$ is pure W-chordal, then $\operatorname{pure-skel}_{n-d-2}(\Delta^\vee)$ is vertex-decomposable}.
\end{quote}
This conclusion is much stronger than Emtander's, as vertex-decomposability is much stronger than sequential-Cohen--Macaulayness. 

In 2016,  Bigdeli, Yazdan-Pour and Zaare-Nahandi \cite{BYZ17} tried to unify the previous two approaches: they introduced the class of ridge-chordal complexes,  which (non-trivially!) contains both pure W-chordal and E-chordal complexes, and proved:
\begin{quote} 
\emph{If $\Delta$ is pure ridge-chordal, then $\operatorname{pure-skel}_{n-d-2}(\Delta^\vee)$ is Cohen--Macaulay}.
\end{quote} 
This generalizes Emtander's result, but not Woodroofe's, because of the weaker conclusion. It also implies (cf.~\cite{BF20}) that if $\Delta$ is  ``chordal'' in the sense of \cite{BF20} (that is, $\operatorname{pure-skel}_k(\Delta)$ is ridge-chordal for every $k \le d$), then $\Delta^\vee$ is sequentially Cohen--Macaulay \cite{BF20}.  
Meanwhile, Nikseresht \cite{Nik19} in 2019 claimed a converse statement, namely, that if $\operatorname{pure-skel}_{n-d-2}(\Delta^\vee)$ is vertex-decomposable, then $\Delta$ is ridge-chordal. Using this, Bigdeli--Faridi essentially showed that if $\Delta^\vee$ is vertex-decomposable, then it is skeleton-ridge-chordal \cite{BF20} (see Section \ref{sec:nik}).

We complete the picture by pursuing an opposite goal with respect to   \cite{BYZ17}. Namely, rather than \emph{weakening} the definition of chordality to accommodate both the W- and the E-chordality \emph{assumptions}, we introduce a \emph{strengthening} of E-chordality, called ``skeleton-E-chor\-dality'', to reach a much stronger claim, which implies both Emtander's and Woodroofe's \emph{conclusions}. This `skeleton-E-chordality' 
is simply  the request that all skeleta are E-chordal with respect to the same vertex labeling. 
A related, weaker notion is  ``skeleton-clique-chordality'', which is the request that the $1$-skeleton be a chordal graph.
With these two notions we obtain the following result, which also includes a converse statement with a much weaker assumption than Bigdeli--Faridi's:

\begin{thmnonumber}[Theorems \ref{thm:VD1} \& \ref{thm:VD2}, Corollary \ref{Cor:subflagSEchordalVD}] \label{mainthm:Decompositions} For any simplicial complex $\Delta$, 
\begin{compactenum}[\rm (i)]
\item if $\Delta$ is skeleton-E-chordal, then $\Delta^\vee$ is vertex-decomposable;
\item  if $\Delta^\vee$ is vertex-decomposable, then $\Delta$ is skeleton-clique-chordal.
\end{compactenum}

In particular, for any subflag simplicial complex $\Delta$,
\[ \Delta  \textrm{ is skeleton-E-chordal } \Longleftrightarrow \Delta^\vee \textrm{ is vertex-decomposable}.\]
\end{thmnonumber}

 ``Subflag'' here means ``with no missing faces of dimension $2, 3,  \ldots, \dim \Delta$''. It is a weaker notion than flag, whence the name. Note that when  $\dim \Delta=1$, any $\Delta$ is vacuously subflag. 
Hence for graphs the second part of Main Theorem \ref{mainthm:Decompositions}  boils down to ``$G$ is chordal if and only if $G^\vee$ is vertex-decomposable'', a result related to Fr\"oberg's characterization, but to the best of our knowledge, new. We leave it to the reader to decide whether Main Theorem \ref{mainthm:Decompositions} accomplishes what is indicated in the literature as the main goal for higher-dimensional chordality, cf.~e.g. \cite[page 1]{BF20}, \cite[pp. 1714-1715]{CF13}, \cite[p.~319]{Nik19}, \cite[p.~130]{BYZ17}.

The ``skeleton-E-chordality'' and the ``skeleton-clique-chordality'' notions have three technical advantages when compared to other higher-dimensional chordalities in the literature:
\begin{compactenum}[(1)]
\item They are elementary to state and verify. Also, in subflag complexes, they coincide.
\item They are inherited under taking the $k$-skeleton. 
In contrast, all notions of chordality in the literature (including those not treated here like resolution-$k$-chordality \cite{ANS16}, $k$-Diracness \cite{ANS16}, or $d$-chordedness \cite{CF13}) only depend on the list of facets. Some authors have already realized the importance of controlling all dimensional layers and have reintroduced such control a posteriori, by predicating their chordality over all skeleta, cf. e.g.~\cite[Definition 2.6]{BF20} or \cite[Definition 8.2]{CF13}. However, using the same vertex labeling for all layers paves the way for inductive proofs.
\item They are inherited under arbitrary vertex deletions, just like graph chordality. In contrast, E-, W- and ridge-chordality are not. Caveat: what Woodroofe \cite{Woo11} calls ``vertex deletions'' are clutter operations that  maintain W-chordality, but they are not the same as vertex deletions in the sense of simplicial complexes. (This also highlights how the heterogeneity of conventions can be confusing: for framing chordality, some authors used simplicial complexes, some clutters, some hypergraphs, some matroids... We will stick to simplicial complexes.)
\end{compactenum}

\noindent This third advantage plays a decisive role for another one of our results:

\begin{thmnonumber}[Corollary \ref{cor:Decompositions}] For any subflag simplicial complex $\Delta$, t.f.a.e.: 
\begin{compactitem}
\item $\Delta$ is skeleton-E-chordal; 
\item $\Delta$ is skeleton-clique-chordal;
\item either the $1$-skeleton of $\Delta$ is complete (so $\Delta = \Sigma^d_n$), or $\Delta$ splits as $\Delta = \Delta_1  \cup \Delta_2$, where $\Delta_1, \Delta_2$ are proper skeleton-E-chordal induced subcomplexes of $\Delta$, and $\Delta_1 \cap \Delta_2$ is a simplicial complex whose $1$-skeleton is a clique.
\end{compactitem}
\end{thmnonumber}

For $d=1$, this boils down to the Dirac's characterization of chordality via clique decompositions. 
Thus a natural question is whether also Dirac's characterization of chordality via simplicial vertices can be generalized to higher dimensions. The answer is positive:

\begin{thmnonumber}[Theorem \ref{thm:SimplicialVertices}] \label{mainthm:SimplicialVertices}
$\Delta$ is skeleton-E-chordal $\Longleftrightarrow$ every nonempty induced subcomplex of $\Delta$ has a skeleton-E-simplicial
vertex.
\end{thmnonumber}

A potential objection to our work is that our `skeleton chordalities' are strong properties, but rare. To deal with this, we also introduced four weakenings of the E-chordality notion. We called them \emph{mid-}, \emph{weak-}, \emph{very-weak-}, and \emph{clique-chordality}. Each one of these four properties has a stronger ``skeleton-version'', in which the property is predicated for all faces rather than only for facets; exactly like we defined skeleton-E-chordality from E-chordality. We clarify all existing implications among these properties in a table presented in Theorem \ref{thm:diagram}. Each one of these four ``skeleton properties'' yields an alternative version of Main Theorem \ref{mainthm:SimplicialVertices} above. 

 Each one of these four properties is useful, in the sense that it has a specific application later in the paper. We already encountered skeleton-clique-chordality in Main Theorem \ref{mainthm:Decompositions} above. Mid-chordality and very-weak-chordality are put to use in Section \ref{sec:DelAbove}, where we expand on the work by Bigdeli, Faridi, Yazdan-Pour and Zaare-Nahandi  by studying deletions above $j$-dimensional faces, in analogy with ridge-chordality. Specifically, we show:

\begin{thmnonumber}[Theorems \ref{thm:E-vertex-iff} \& \ref{thm:weaklyVertexChordal}]
All pure E-chordal complexes are vertex-chordal, all  pure mid-chordal complexes are weakly-vertex-chordal, and all pure very-weakly-chordal complexes are weakly-ridge-chordal.
\end{thmnonumber}

Finally, weak-chordality is useful for the following result, which connects the higher-dimensional versions of chordal graphs and of interval and unit-interval graphs (cf.~\cite{BSV22}):

\begin{thmnonumber}[Theorem \ref{thm:hierarchy2b}]
All unit-interval  complexes are skeleton-E-chordal. \\All interval (aka underclosed) complexes are skeleton-weakly-chordal.
\end{thmnonumber}

In the last section, we conclude by giving a simple  proof of the implication ``$\Delta^\vee$ sequential-Cohen--Macaulay $\Rightarrow$ $\Delta$ geochordal'', without closure or purity assumptions. 

A few important higher-dimensional chordality notions are missing from this paper; for example, the $d$-chorded complexes by Connon and Faridi \cite{CF13}, or the $k$-Dirac and resolution-$k$-chordal complexes by Adiprasito--Nevo--Samper \cite{ANS16}. The reason for this omission is simply that we plan to treat these notions in a separate forthcoming article.

\newpage
%%%%%%%%%%%%%%%%%%%% SECTION 2 %%%%%%%%%%%%%%%%%%%%%%%%%%
%%%%%%%%%%%%%%%%%%%%%%%%%%%%%%%%%%%%%%%%%%%%%%%%%
\section{Chordality via vertex labelings}
%%%%%%%%%%%%%%%%%%%%%%%%%%%%%%%%%%%%%%%%%%%%%%%%%
%%%%%%%%%%%%%%%%%%%%%%%%%%%%%%%%%%%%%%%%%%%%%%%%%
Throughout $d, n$ are positive integers, with $d < n$. We denote by $\Sigma_n$ the $(n-1)$-dimensional simplex with vertex set $[n]$, and by $\Sigma^d_n$ its $d$-skeleton.
We write the $d$-faces of  $\Sigma_n$, or of any simplicial complex $\Delta$ on the vertex set $[n]$, by listing vertices contained \emph{in increasing order}, e.g.: $ \Delta = 123, 124, 235.$ Two $d$-faces are \emph{adjacent} if their intersection has dimension $d-1$. (In particular, for us any two adjacent faces must have same size.) 
For $1 \le i \le n-d$ let us call $H^d_i$ the $d$-face of $\Sigma_n$ with vertices $i, i+1, \ldots, i+d$.  
We extend the definition of $H^d_i$ also to $i \in \{n-d+1, \ldots, n\}$ by using the ``congruence modulo $n$'' convention. In other words, by ``$n+1$'' we mean vertex $1$, by ``$n+2$'' we mean vertex $2$, and so on. 
A natural way to extend chordality to higher dimensions is  the following (cf. also Connon--Faridi \cite[Definition 4.5]{CF13}):

\begin{definition} A simplicial $d$-complex $\Delta$ with $n$ vertices, not necessarily pure, is:
\begin{compactitem}
\item \emph{Geometrically-$k$-chordal}, for some $k \in \{1, \ldots, d\}$, if for every subcomplex $S$ of $\Delta$ homeomorphic to $\partial \Sigma_{k+2}$, but not combinatorially equivalent to $\partial \Sigma_{k+2}$, there is a $k$-face $F$ in $\Delta$  such that all vertices of $F$ are in $S$, but $F$ is not in $S$.
\item \emph{Geochordal}, if it is geometrically-$k$-chordal for all $k=1,\ldots, d$.
\end{compactitem}
\end{definition}

Geometric-$k$-chordality implies the lack of induced $k$-spheres with $\ge k+3$ vertices. To explore stronger generalizations of chordality, we need to look at its various characterizations. 

%%%%%%%%%%%%%%%%%% SUBSECTION 2.1 %%%%%%%%%%%%%%%%%%%%%%%
\subsection{A hierarchy of the chordalities via vertex labelings}
%%%%%%%%%%%%%%%%%%%%%%%%%%%%%%%%%%%%%%%%%%%%%%%%%
A famous characterization of chordality is via a total order on the vertices, also known as ``perfect elimination ordering'', with the property that the neighbors of the largest vertex form a clique\footnote{Most graph theorists prefer the opposite convention, in which the neighbors of the smallest vertex form a clique. We prefer this convention because eliminating the largest-label vertex from a graph does not require  relabeling the others. There is of course no conceptual difference, their ordering is the reverse of ours.}.
This is an easy consequence of the work of Fulkerson--Gross \cite[Section 7]{FG65}, which is in turn based on the work by Dirac \cite{Dir61}:

\begin{theorem}[Fulkerson--Gross \cite{FG65}]  \label{thm:CharVertexLabeling}
$G$ is a chordal graph  if and only if it has a vertex labeling such that, for all $a<b<c$, if $ac$ and $bc$ are both edges of $G$, so is $ab$.
\end{theorem}

This definition can be generalized to higher dimensions in \emph{many} different yet natural ways.

\begin{definition} \label{def:P}
A $d$-dimensional simplicial complex $\Delta$ with $n$ vertices, not necessarily pure, is:
\begin{compactitem}
\item \emph{E-chordal}, cf.\cite{Emt10}, if it has a vertex labeling such that for every two facets $F \ne G$ in $\Delta$ with the same size and same maximum, $\Delta$ also contains every face $H$ of  $\Sigma_n$ contained in $F \cup G$ and of the same size of $F$ and $G$.
\item \emph{Mid-chordal}, if it has a vertex labeling such that for every two facets $F \ne G$ in $\Delta$ with the same size and same maximum, $\Delta$ contains every $2$-element subset $e$ of $F \cup G -\{\max G\} $, and in addition, for each such $e$, $\Delta$ contains some face $H_e$  of the same size of $F$ and $G$ such that 
 $e \subseteq H_e \subseteq F \cup G -\{\max F\}$.
\item  \emph{Weakly-chordal}, if it has a vertex labeling such that, for every two facets $F \ne G$ in $\Delta$ with the same size and same maximum, $\Delta$ also contains some face $H$  of the same size of $F$ and $G$ such that $H \subseteq F \cup G - \{ \max F\}$.
\item  \emph{Very-Weakly-chordal}, if it has a vertex labeling such that for every two adjacent facets $F \ne G$ in $\Delta$ with the same size and same maximum, $\Delta$ also contains the unique face $H$  of  $\Sigma_n$ with vertex set equal to $F \cup G - \{ \max F\}$. 
\item \emph{Clique-chordal}, if it has a vertex labeling such that for every two facets $F \ne G$ in $\Delta$ with the same size and same maximum, $\Delta$ also contains every $2$-element subset of $F \cup G$.
\end{compactitem}
\end{definition}

Of these properties, only the first one has been studied, by Emtander \cite{Emt10}. The E in `E-chordality' is in his honor. 
All these properties predicate on facets. We would like properties propagating to the lower-dimensional skeleta as well. For this reason, we introduce the following five properties, which   \emph{also} boil down to graph chordality for $d=1$: 

\begin{definition}\label{def:skelP}
A $d$-dimensional simplicial complex $\Delta$ with $n$ vertices, not necessarily pure, is:
\begin{compactitem}
\item \emph{Skeleton-E-chordal}, if it has a vertex labeling such that for every two faces $f \ne g$ in $\Delta$ with the same size and same maximum, $\Delta$ also contains every face $h$ of  $\Sigma_n$ contained in $f \cup g$ and of the same size of $f$ and $g$.
\item \emph{Skeleton-Mid-chordal}, if it has a vertex labeling such that for every two faces $f \ne g$ in $\Delta$ with the same size $\ge 2$ and same maximum, $\Delta$ contains every $2$-element subset $e$ of $f \cup g -\{\max f\} $, and in addition, for each such $e$, $\Delta$ contains some face $h_e$  of the same size of $f$ and $g$ such that 
 $e \subseteq h_e \subseteq f \cup g -\{\max f\}$.
\item  \emph{Skeleton-Weakly-chordal}, if it has a vertex labeling such that for every two faces $f \ne g$ of $\Delta$ with the same size and same maximum, $\Delta$ also contains some face $h$  of the same size of $f$ and $g$ such that $h \subseteq f \cup g - \{ \max f\}$. 
\item  \emph{Skeleton-Very-Weakly-chordal}, if it has a vertex labeling such that for every two adjacent faces $f, g$ in $\Delta$ with the same size and same maximum, $\Delta$ also contains the unique face $h$  of  $\Sigma_n$ with vertex set equal to $f \cup g - \{ \max f\}$. 
\item \emph{Skeleton-Clique-chordal}, if it has a vertex labeling such that for every two faces $f \ne g$ in $\Delta$ with the same size $\ge 2$ and same maximum, $\Delta$ contains any $2$-element subset of $f \cup g$.
\end{compactitem}
\end{definition}

\begin{remark} \label{rem:CharSkClCh} Skeleton-clique-chordality is equivalent to the request that the $1$-skeleton be a chordal graph. To see this, 
fix any vertex labeling that shows the chordality of the 1-skeleton $G$ of $\Delta$. Let $f,g\in \Delta$ be two faces of same size $\ell \ge 2$, same maximum $m$. All vertices of $f\cup g - \{m\}$ are neighbors of $m$, with labels $< m$. So they form a clique in $G$. Hence the same labeling that proves $G$ chordal, also proves $\Delta$ skeleton-clique-chordal. The other direction is obvious.
\end{remark}

\begin{figure}[t]
    \centering
    \includegraphics[height=0.2\linewidth]{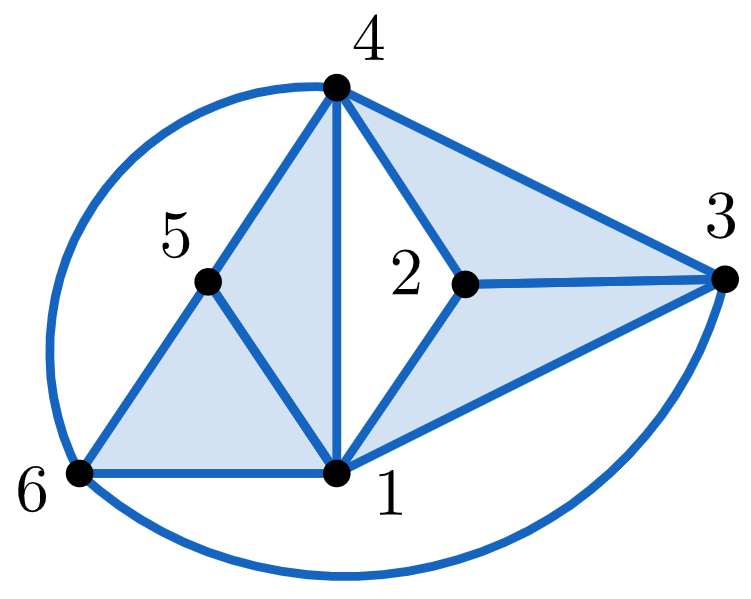}
    \hfill
    \includegraphics[height=0.24\linewidth]{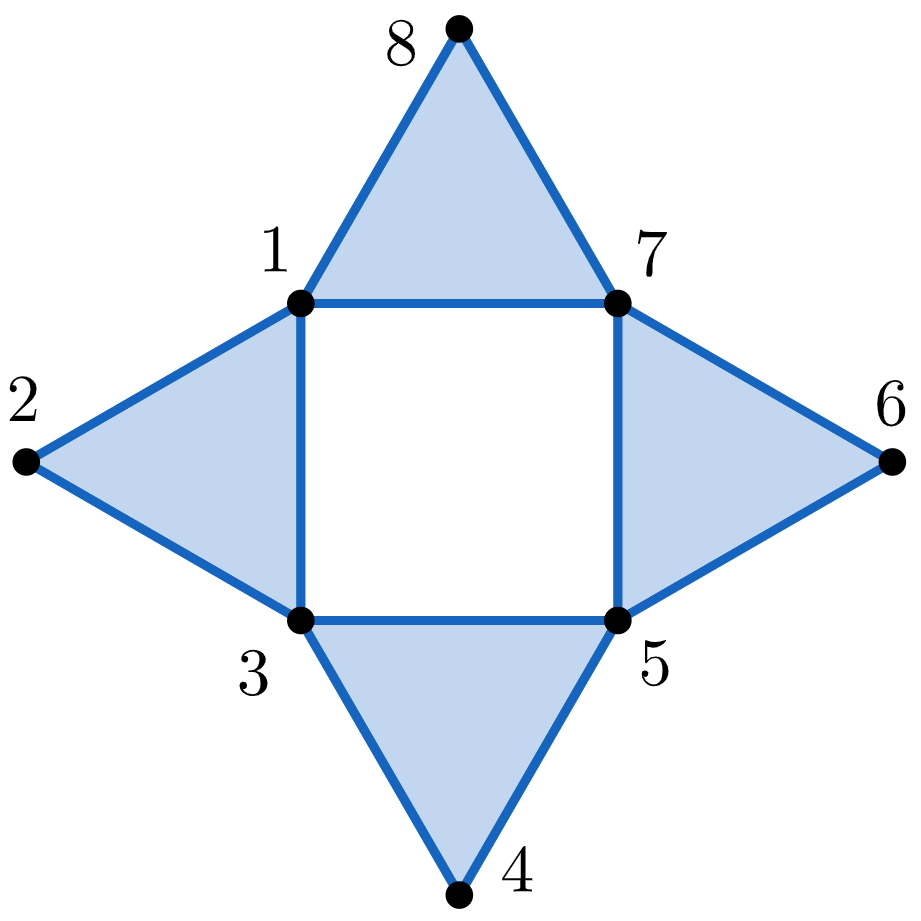}
    \hfill
    \includegraphics[height=0.18\linewidth]{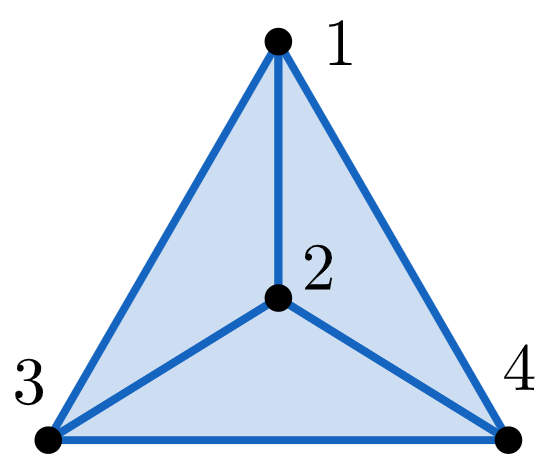}
    \caption{The three simplicial complexes $A$, $B^2$, and $C^2$ from Remark \ref{rem:Initial} (left), Remark \ref{rem:SkeletonChordalWeaklyChordal}  (center), and Lemma \ref{lem:ThreeFacets} (right). 
    The left complex $A$ is E-chordal, with 1-skeleton chordal. The central complex $B^2$ is E-chordal, though its $1$-skeleton is not. The right complex $C^2$ is not E-chordal, though its $1$-skeleton is. 
    None of the three complexes is skeleton-E-chordal.}
    \label{fig:Rem6}
\end{figure}

\begin{remark} \label{rem:Initial} 
Skeleton-E-chordality is basically the request that all skeleta of $\Delta$ are E-chordal \emph{with respect to the same labeling}. Similarly for the other properties. The parts in italics can  be omitted without consequences for skeleton-clique-chordality (cf. Remark \ref{rem:CharSkClCh}), but not for the other properties. In fact, consider the simplicial complex (Figure \ref{fig:Rem6})

\[ A = 123, 145, 156, 234, 36, 46\]
This labeling proves $A$ simultaneously  E-, mid- clique-, weakly-, and very-weakly-chordal: In fact, the only two facets with the same size and same maximum are $36$ and $46$, and indeed $34$ is a face of $A$. Also, the $1$-skeleton of the complex is $K_6$ minus the three edges $25, 26, 35$, which is chordal. Now suppose $A$ has a labeling that proves it skeleton-very-weakly-chordal. Then the vertex $v$ labeled by $n$ would satisfy the following property: for any two adjacent faces $f \ne g$ of $\Delta$ of the same size that contain $v$, $\Delta$ should also contain the unique face $h = f \cup g - \{v\}$.
However, by inspection, none of the six vertices in $A$ satisfies this property.  
This proves that $A$ is neither skeleton-E-chordal, nor skeleton-mid-chordal, nor skeleton-weakly-chordal, nor skeleton-very-weakly-chordal. It is instead skeleton-clique-chordal with the labeling 
$A=123, 146, 156, 234, 35, 45$.
\end{remark}

\begin{remark}\label{rem:SkeletonChordalWeaklyChordal}
None of the properties in Definition \ref{def:P} implies any of those in Definition \ref{def:skelP}. This is best seen generalizing \cite[Example 4.8]{Woo11}. For any $d \ge 2$, the \emph{$d$-dimensional Woodroofe complex}  $B^d$ is the simplicial complex on $4d$ vertices (cf. Figure \ref{fig:Rem6}) with facets 
\[ B^d \ = \ H^d_1, \: H^d_{d+1}, \: H^d_{2d+1}, \: H^d_{3d+1}.\]
(With our convention, the last facet contains the vertex $1$.) This $B^d$ is vacuously E-, clique-, mid-, weakly-, and very-weakly-chordal, as no two facets have same maximum. Yet its $1$-skeleton is not chordal: The induced subgraph on the vertices congruent to $1$ modulo $d$ is a $4$-cycle.  
\end{remark}

\enlargethispage{3mm}

\begin{lemma} \label{lem:annulus} Let $n \ge  2d + 2 \ge 4$ be integers. The \emph{$d$-dimensional annulus $A^d(n)$}  with facets $H^d_1, H^d_2, \ldots, H^d_n$ is a simplicial complex that is neither clique- nor very-weakly-chordal.
\end{lemma}

\begin{proof} With respect to  the labeling we provided, in $A^d(n)$ the facets containing $n$ are exactly $d+1$, namely, $H^d_{n-d}$, $H^d_{n-d+1}$, $\ldots$, $H^d_n$. In this list, consecutive facets are adjacent.  The first two, $H^d_{n-d}$ and $H^d_{n-d+1}$, already violate the very-weakly-chordal condition: the complex does not contain the facet of vertices $\{1, n-d, n-d+1, \ldots, n-1\}$. They also violate the clique-chordal condition: the edge $[n-d, 1]$ is not present. 
In fact, any pair of consecutive facets from the list $H^d_{n-d}$, $H^d_{n-d+1}$, $\ldots$, $H^d_n$, violates the very-weakly- and the clique-chordal conditions.

Now consider any other labeling of $A^d(n)$. Because  $A^d(n)$ is symmetric (all vertex links are combinatorially equivalent), we can recycle the argument above for the vertex that in this new labeling is called $n$: It will belong to exactly $d+1$ facets, which can be listed so that consecutive facets are adjacent, and it will be the maximum in all of them.
Any consecutive pair of facets in this list will violate the very-weakly- and the clique-chordal condition.
\end{proof}

\begin{lemma} \label{lem:pinchedannulus} Let $n \ge  2d + 1 \ge 4$ be integers. The $d$-dimensional pinched annulus $P^d(n)$  with facets $H^d_1, H^d_2 , \ldots, H^d_{n-d+1}$ (Figure \ref{fig:nonstandard})  is very-weakly-chordal, though not skeleton-very-weakly-chordal.  Moreover, $P^d(n)$ is weakly-chordal, clique chordal  and skeleton-clique-chordal if $n=2d+1$, but neither of the three if $n\ge 2d+2$.
\end{lemma}

\begin{proof} From $n \ge 2d+1 \ge 4$ it follows that $d \ge 2$. The facets $F=H^d_1$ and $G=H^d_{n-d+1}$ have exactly one vertex in common, which is labeled by $1$. Since we are in dimension two or higher, $F$ and $G$ are not adjacent. With this labeling, when $n=2d+1$, the $1$-skeleton of $P^d(2d+1)$ consists of vertex $1$ joined to the graph on $2, \ldots, 2d+1$ where $ij$ is an edge if and only if $|i-j| \le d$. Since this is an interval graph, hence chordal, $P^d(2d+1)$ is skeleton-clique-chordal.
Let us now go back to $n \ge 2d+1$ case. Let us relabel each vertex $i>1$ by $i-1$, and let us relabel vertex $1$ by $n$. Now  the vertices in $F$ are labeled $\{n, 1, 2, \ldots, d\}$, whereas those in $G$ are labeled $\{n-d, n-d+1,  \ldots, n\}$.  In the new labeling, the pinch point is $n$; the only facets with same maximum are $F$ and $G$, which are not adjacent.  So $P^d(n)$ is very-weakly-chordal for all $n \ge 2d+1$. 
For $n=2d+1$, moreover, the only two facets with same maxima are $F$ and $G$, and the face $H$ of vertices $1, 2, \ldots, d+1$ satisfies $H \subseteq F \cup G - \{n\}$. So when $n=2d+1$, $P^d(n)$ is also weakly-chordal. (The clique-chordality condition is not satisfied under this second labeling because the edge $1, d+2$ is missing.)
When $n \ge 2d+2$, no labeling satisfies weak-chordality. In fact, the vertex labeled by $n$ would be maximum for two or more $d$-faces that would violate the weak-chordality condition. For the same reason, for $n \ge 2d+2$, 
$P^d(n)$ is not clique-chordal. 
As for skeleta: The $1$-skeleton of a large annulus is not chordal. But even for $n=2d+1$, when the cycle generating the homology is a triangle, $P^d(n)$ is not skeleton-very-weakly-chordal with the labeling we constructed above, since the edges $[n-1, n]$ and $[1,n]$ belong to distinct $d$-faces (namely, $G$ and $F$, respectively). In fact, it is not difficult to see that no labeling works.
\end{proof}

\begin{figure}[t]
    \centering
    \includegraphics[height=0.22\linewidth]{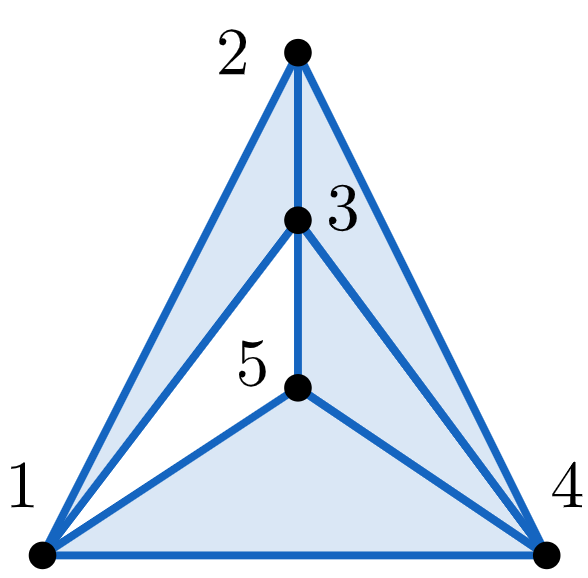}
    \hfill
    \includegraphics[height=0.22\linewidth]{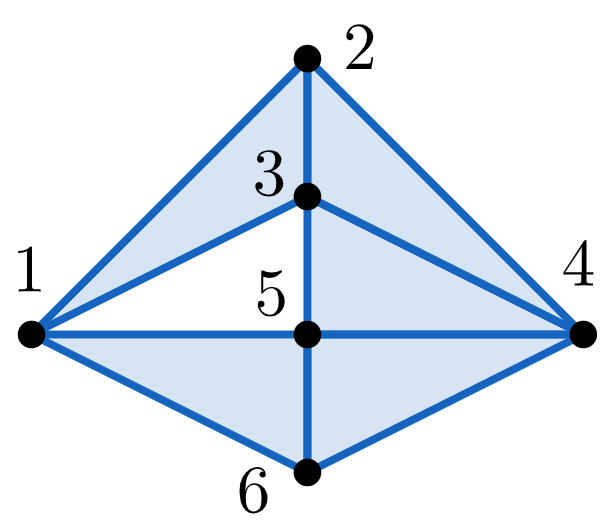}
    \hfill
    \includegraphics[height=0.22\linewidth]{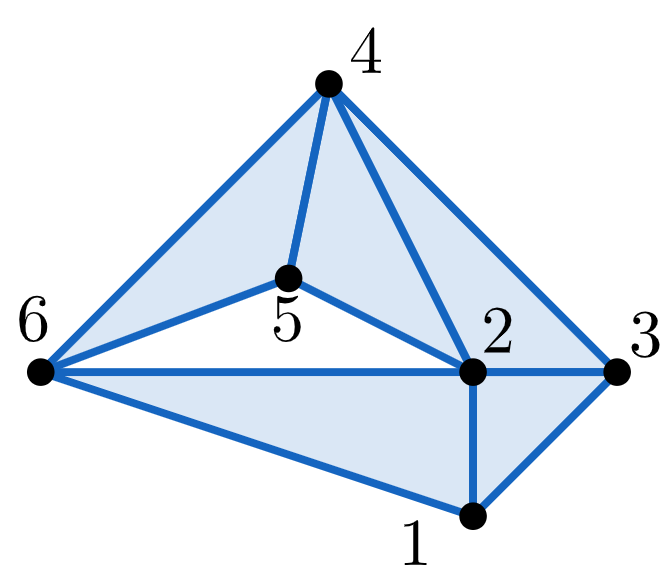}
    \caption{The simplicial complexes $P^2(5)$ (left) and $P^2(6)$ (middle) from Lemma \ref{lem:pinchedannulus}, and the simplicial complex $P'$ from Remark \ref{rem:nonstandard}.}
    \label{fig:nonstandard}
\end{figure}

\begin{remark} \label{rem:nonstandard} The above results apply to \emph{the} pinched annulus, which is a specific standard triangulation of what is known in topology as pinched annulus. Other triangulations exist, but our combinatorial results might not extend to them. For example, consider (Figure \ref{fig:nonstandard})
\[ P'= 123, 126, 234, 245, 456.\]
Topologically, this is also \emph{a} pinched annulus. However, in this triangulation, vertex $2$ has five neighbors. Hence,  $P'$ is not combinatorially equivalent to $P^2(6)$, since in $P^d(n)$ every vertex has at most $2d$ neighbors. In fact, $P'$ is weakly-chordal, whereas $P^2(6)$, by Lemma \ref{lem:pinchedannulus}, is not.
\end{remark}

\begin{lemma} \label{lem:ThreeFacets} Let $d \ge 2$. Let $C^d$ be the $d$-complex with facets $H^d_1$, $H^d_2$, and $G^d \eqdef \{1, \ldots, d, d+2\}$. Then $C^d$ is skeleton-mid-chordal, but not E-chordal (see Figure \ref{fig:Rem6}).
\end{lemma}

\begin{proof} Verifying mid-chordality is easy: $H^d_2$ and $G^d$ are the only two facets with same max, same size. Their union contains all $2$-element subsets of $\{1, 2, \ldots, d+1\}$. For each such $2$-element subset $e$, $H^d_1$ plays the role of $H_e$ in the definition of mid-chordal. Skeleton-mid-chordality is also easy: let $f,g$ be two faces with the same size and same maximum. If this maximum is $\le d+1$, then $f,g$ belong to the same simplex $H^d_1$, and the claim is clear. If this maximum is $d+2$, then the vertices in the link of $\max f$ are a subset of $\{1, \ldots, d+1\}$. So clearly any two are connected by an edge (since $H^d_1$ is in $C^d$) and any such edge in some face of $H^d_1$ of the same size of $f$ and $g$.
We are left with verifying that $C^d$ is not E-chordal. Let us start by noticing that under our labeling, as well as under any other possible labeling, each of the three facets of $C^d$ is the set $[d+2]$ minus one element. Now, a set of this type can only have two possible maxima, namely, $d+1$ (if the vertex labeled by $d+2$ is the one excluded) or $d+2$ (otherwise). So under any vertex labeling, at least two of the three facets of $C^d$ will have same maximum. However, the simplex $\Sigma_{d+2}$  has $d+2$ $d$-faces. Since $3<d+2$ when $d \ge 2$,   $C^d$ has too few facets to include the full $d$-skeleton of $\Sigma_{d+2}$. Hence, $C^d$ is not E-chordal. See Figure \ref{fig:Rem6} for the $d=2$ case.
\end{proof}

\begin{lemma} \label{lem:WeaklyNonMidChordal} Let $k>d>1$ be integers.
Let $D^d (d+k)$ be the $d$-dimensional complex on $d+k$ 
vertices obtained as follows:  join the $(d-1)$-simplex to a $0$-complex consisting of $k$ disjoint points $x_1, \ldots, x_k$, and finally add to it a $d$-face $G$ spanned by $d+1$ of the $x_i$'s.
Then $D^d(d+k)$ is skeleton-weakly-chordal and skeleton-clique-chordal, but not mid-chordal.
\end{lemma}

\begin{proof} 
Label the vertices of the $(d-1)$-simplex $S$ by $1, \ldots, d$, the ones in $G$ by $d+1, \ldots, 2d+1$, and the (possibly) remaining ones by $2d+2, \ldots, d+k$. So the only two facets with same maxima are $F = [1, 2, \ldots, d, 2d+1]$ and $G$, the maximum being $v=2d+1$. The face $H=H^d_1=[1, \ldots, d, d+1]$ is in $F \cup G - \{ \max F\}$ and proves the labeling weakly-chordal. 
Also, the vertices in the link of $v$ are  $\{1, \ldots, 2d\}$. Any two such vertices are connected by an edge, either in $G$ or in one of the cones $x_i*S$. But if $j, l$ are both larger than $d$, the unique $d$-face containing the edge $j, l$ is $G$, which also contains $v$. So this labeling does not satisfy the mid-chordality definition.
We claim no other labeling does either. In fact, let $v$ be the largest labeled vertex among those that belong to at least two facets in a given labeling. This means that $v$ belongs either to $S$ or to $G$:
\begin{compactitem}
\item If $v$ belongs to $S$, then all facets $x_i*S$ with $x_i\in G$ have maximum $v$. Let $L$ be the link of $v$ in $\Sigma_d$. Any edge contained in $x_i*L$ belongs only to $x_i*S$, which also contains $v$. So the labeling cannot prove the complex mid-chordal. 
\item If instead $v$ belongs to $G$, then both $G$ and $v*S$ have maximum $v$. Indeed, all vertices of these facets belong to at least two facets, so none has a label larger than $v$. Let $L'$ be the link of $v$ in $G$. Any edge contained in $L'$ belongs only to $G$, which also contains $v$. So also in this case, the labeling cannot prove the complex mid-chordal.
\end{compactitem}
As for the skeleton, let us  go back to our original labeling. Let $f,g$ be two faces with the same size, same maximum.  If  $\max f \le d$, then $f,g$ are both in the same simplex $S$, and we conclude (as the simplex is skeleton-mid-chordal). If $\max f >d$, there are three cases: 
(1)  either $f,g$ belong both to the $d$-simplex $G$, or (2) none of them does, or (3) exactly one of them does.
Case (1) is easy, because $f,g$ are in the same simplex. Case (2) is similar: Since $f$ and $g$ intersect, for some $i$ both $f$ and $g$ are in the same simplex $x_i*S$. So let us focus on Case (3). Say  $g$ belongs to $G$ and $f$ does not. Since they intersect, $f$ must be a face of the form $h*x$, for some $x$ in $G$ that belongs also to $g$. Since the intersection of $G$ with any other facet of $D^d(d+k)$ consists of one point only, this $x$ must coincide with the common maximum of $f$ and $g$.   
But for every $b \in G$ and $a \in \{1, \ldots, d\}$, the edge $ab$ is in $D^d(d+k)$, as a face of the simplex $S*b$. 
In particular, every edge $ab$ with $a<x$ in $f$ and $b$ in $g$ belongs to $D^d(d+k)$. Moreover, writing $f=h_0*x$ with $h_0\subseteq S$, for any $y\in g-\{x\}$ the face $h=h_0\cup\{y\}\subseteq S*y$ has the same size as $f$ and satisfies $h\subseteq f\cup g-\{x\}$. Thus $D^d(d+k)$ is skeleton-weakly-chordal. It is easy to see that its $1$-skeleton is chordal.
\end{proof}

\begin{theorem}\label{thm:hierarchy1}
 For simplicial complexes of dimension $d$, one has the following hierarchy:
\[
 \{ \textrm{E-} \}
 \subseteq
  \{ \textrm{Mid-} \}
\subseteq \{ \textrm{Weakly-} \}
\subseteq \{ \textrm{Very-Weakly-} \}
\subseteq \{ \textrm{Geometrically-$d$-chordal} \}
\subseteq \{ \textrm{all} \},
\]
and for each $d \ge 2$, all inclusions are strict.  
Moreover, one has the parallel hierarchy
\[\{ \textrm{Mid-} \}
\subseteq \{ \textrm{Clique-} \}
\subseteq \{ \textrm{all} \},\]
and for each $d \ge 2$, all inclusions are strict. The hierarchies are parallel in the sense that neither weakly- implies clique-chordal, nor clique- implies very-weakly-chordal for $d \ge 2$. 
\end{theorem}

\begin{proof} For $d=1$: E-, Mid-, Clique-, Weakly-, Very-Weakly-, and Geometrically-$d$-chordal, simply mean `chordal'. For $d\ge 2$:
\begin{compactitem} 
\item E- implies Mid- is clear. 
The $d$-dimensional complex $C^d$ with three facets $H_1^d$, $H_2^d$, and $G^d = \{1, \ldots, d, d+2\}$ of Lemma \ref{lem:ThreeFacets} is mid-chordal, but not E-chordal. 
\item Mid- implies Weakly-: Also clear. For the strictness, see Lemma \ref{lem:WeaklyNonMidChordal}.
\item Weakly- implies Very-Weakly-: The implication is trivial; for its strictness,  look at the pinched annulus $P^d(n)$, for $n$ large, and apply Lemma \ref{lem:pinchedannulus}. 

\item Very-Weakly- implies Geometrically-$d$-chordal: Let $\Delta$ be a $d$-dimensional complex, with a labeling that makes it very-weakly-chordal. Let $S$ be a subcomplex of $\Delta$ homeomorphic to, but not combinatorially equivalent to, the boundary of a $(d+1)$-simplex. Suppose by contradiction that every $d$-face of $\Delta$ with all vertices in $S$, is also in $S$. Let $v$ be the vertex of $S$ with the highest label. For any two adjacent $(d-1)$-faces $f,g$ in $\operatorname{link}(v,S)$, the facets $v *f, v*g$ of $\Delta$ are adjacent with maximum $v$, so $\Delta$ contains the $d$-face $f \cup g$. By assumption, $f \cup g$ is in $S$. There are two cases:
\begin{compactitem}[--]
\item If $\operatorname{link}(v,S)$ is the boundary of a $d$-simplex $\sigma$, then $\sigma = f \cup g$ is in $S$. So $S$ contains, and thus equals, the boundary of the $(d+1)$-simplex $v \ast \sigma$. A contradiction. 
\item If $\operatorname{link}(v,S)$ has more than $d+1$ vertices, then we claim that there are three $(d-1)$-simplices $f,g,h$ in $\operatorname{link}(v,S)$ with the property that $g$ is adjacent to both $f$ and $h$, and $f \cup g$ is different from $g \cup h$. 
To see that such facets exist, suppose otherwise. Then, for every facet $g$ of $\operatorname{link}(v,S)$, all facets adjacent to $g$ have the same union with $g$. Since every ridge belongs to exactly two facets, these facets together with $g$ form the boundary of a $d$-simplex. By strong connectivity, this is the entire link, contradicting that it has more than $d+1$ vertices.
By assumption,  the distinct $d$-faces $f\cup g$ and $g\cup h$ are contained in $\Delta$, and thus also in $S$. So  $S$ has three $d$-faces ($f\cup g$, $g\cup h$, and $g * v$) containing the $(d-1)$-face $g$. But then $S$ is not a $d$-manifold, a contradiction. 
The strictness of the implication is shown via Lemma \ref{lem:annulus}:  $A^d(n)$ is not very-weakly-chordal, but it is easily shown to be geometrically-$d$-chordal. 
\end{compactitem}
\item The boundary of the $(d+1)$-dimensional simplicial complex of facets $H^{d+1}_1$ and $ H^{d+1}_2$, is $d$-dimensional and not geometrically-$d$-chordal. This example also proves the strictness of the (easy) implication ``$\Delta$ mid-chordal $\Rightarrow$ $\Delta$ clique-chordal''.
\item  
The simplicial complexes 
\[ \Delta_1  = 123, 124, 126, 356, 456 
\qquad \textrm{ and } \qquad \Delta_2 =  124, 125, 134, 135, 234\]
are (1) weakly-chordal but not clique-chordal and (2) not very-weakly-chordal but clique-chordal, respectively. Both examples propagate to higher dimensions by taking iterated cones, and using Propositions \ref{prop:weakStability} and \ref{prop:CliqueStability}.
\qedhere
\end{compactitem}   
\end{proof}

\begin{theorem}\label{thm:hierarchy2}
 For simplicial complexes of dimension $d$, one has the following hierarchy:
\[
 \{ \textrm{Sk.-E-} \} 
\subseteq   \{ \textrm{Sk.-Mid-} \}
\subseteq \{ \textrm{Sk.-Weakly-} \}
\subseteq \{ \textrm{Sk.-Very-Weakly-} \}
\subseteq \{ \textrm{Geo-} \}
\subseteq \{ \textrm{Sk.-Clique-} \}
\subseteq \{ \textrm{all} \},
\]
and for $d \ge 2$, all inclusions are strict. 
\end{theorem}

\begin{proof}
For $d=1$ the skeleton properties  above boil down to classical graph chordality. For $d \ge 2$, the inclusions are proven analogously to the corresponding inclusions of Theorem \ref{thm:hierarchy1}. There is one additional implication, namely, that geochordality implies skeleton-clique-chordality. 
This additional implication follows immediately from Remark \ref{rem:CharSkClCh}. 

\noindent As for the strictness of the various containments, in each dimension $d \ge 2$:
\begin{compactitem} 
\item Lemma \ref{lem:ThreeFacets} yields skeleton-mid-chordal complexes that are not skeleton-E-chordal.
\item Lemma \ref{lem:WeaklyNonMidChordal} yields skeleton-weakly-chordal complexes that are not skeleton-mid-chordal.
\item In any pure $d$-dimensional complex, the (very-)weakly-chordal property depends only on the list of $d$-faces. 
Since the pinched annulus $P^d(n)$ of Lemma \ref{lem:pinchedannulus} is very-weakly- but not weakly-chordal,  
the same is true for 
$Q^d(n) = P^d(n) \cup \operatorname{skel}_{d-1}(\Sigma_n)$,
which by definition has all skeleta very-weakly-chordal. 
\item The pinched annulus $P^2(5)$ is geochordal (its $1$-skeleton is $K_5$ minus one edge, and it contains no $2$-sphere), but by Lemma \ref{lem:pinchedannulus} it is not (skeleton-)very-weakly-chordal. In dimension $d>2$, take $(d-2)$ consecutive cones over $P^2(5)$: It is geochordal by Proposition \ref{prop:conesGeo} below, but not skeleton-very-weakly-chordal, since deleting the cone apices yields back $P^2(5)$ and skeleton-very-weakly-chordal is preserved under vertex deletions (Proposition \ref{prop:SkelWeakStability} below).
\item When $d=2$, the suspension of a $3$-cycle (Figure \ref{fig:weklynotgeod}) has chordal $1$-skeleton, but being a $2$-sphere with five vertices, it is not geochordal. For $d>2$ examples of non-geochordal,  skeleton-clique-chordal simplicial complexes can be found by taking iterated cones over the suspension of the $3$-cycle (using Propositions \ref{prop:conesGeo} and \ref{prop:SkelCliqueStability} below).
\item Any non-clique-chordal complex is also not skeleton-clique-chordal. \qedhere
\end{compactitem}
\end{proof}

In conclusion, we have the following summary of implications:

\begin{theorem} \label{thm:diagram}
The logical diagram above is maximal up to transitivity: That is, all arrows not drawn and not implied by transitivity of implication, are false in any dimension $d \ge 2$. 
\begin{table}[ht]
\centering    
\begin{tabular}{||c c c c c||} 
 \hline

 Skeleton-E- & 	$\Longrightarrow$ & E- &&\\ 

$ \Downarrow$ & & $ \Downarrow$ &&\\ 

Skeleton-mid- & $\Longrightarrow$ & mid-  \tikzmark{toparrow} & &\\
$\Downarrow$ & & $\Downarrow$ &&\\
Skeleton-weakly- & $\Longrightarrow$ & weakly- &&\\
$\Downarrow$ & & $\Downarrow$ &&\\
Skeleton-very-weakly- & $\Longrightarrow$ & very-weakly- &&\\
$\Downarrow$ & & &&\\
Skeleton-clique- & $\Longrightarrow$ & clique-\tikzmark{bottomarrow} & &\\
 \hline
\end{tabular}
\begin{tikzpicture}[overlay,remember picture]
\draw[
  double,
  double distance=1.4pt,
  line width=.45pt,
  black!70,
  -{Implies[length=4.2mm,width=4.2mm]}
]
  ([xshift=.4em,yshift=.4em]pic cs:toparrow)
  .. controls +(1.1,0) and +(1.1,0) ..
  ([xshift=.4em,yshift=.1em]pic cs:bottomarrow);
\end{tikzpicture}
\caption{Chordalities from vertex labelings: a hierarchy}
\label{table:2}
\end{table}
\end{theorem}

\begin{proof}
The horizontal implications are trivial, as facets are faces. They are strict by Remark~\ref{rem:SkeletonChordalWeaklyChordal}.  
The vertical implications are proved (strict!) in Theorems \ref{thm:hierarchy1} and \ref{thm:hierarchy2}. The downward diagonal implications ($\searrow$) are true by transitivity.
As for the non-implications: Lemma \ref{lem:pinchedannulus} shows that ``very-weakly- implies clique-'' is  false. That ``weakly- implies clique'' is false, is shown in Theorem \ref{thm:hierarchy1}.
The  `slash diagonals upwards' ($\nwarrow$) are false, since already the leftward horizontals are false. As for `backslash diagonals upwards' ($\nearrow$): Skeleton-mid- does not imply E-chordal by Lemma \ref{lem:ThreeFacets}. Skeleton-Weakly does not imply mid-chordal by Lemma \ref{lem:WeaklyNonMidChordal}.
Lemma \ref{lem:pinchedannulus} yields a $d$-dimensional complex $P_d(n)$ that is not weakly-chordal; since adding lower-dimensional faces does not affect weak-chordality. $Q^d_n := P^d_n \cup \operatorname{(d-1)-skel}(\Sigma_n)$ is skeleton-very-weakly-chordal, but  not weakly-chordal. Finally, Remark~\ref{rem:SkeletonChordalWeaklyChordal} shows that all `backslash diagonals downwards' ($\swarrow$) are false.
\end{proof}

\begin{lemma}[``Alexander Dual of a point''] \label{lem:SubdividedTriangle} Let $n -2 \ge d\ge 1$.
Let $AD^d(n)$ be the $d$-dimensional simplicial complex obtained by removing a single $d$-face from the $d$-skeleton of the simplex $\Sigma_n$. Then $AD^d(n)$ is skeleton-mid-chordal. It is E-chordal if and only if $d=1$.
\end{lemma}

\begin{proof}
When $d=1$, $AD^1(n)$ is the complete graph minus one edge, which is well known to be chordal. 
When $d\ge 2$, by contradiction, suppose $AD^d(n)$ has a vertex-labeling that shows it is E-chordal. Let $M=x_0x_1  \ldots x_d$ be the missing $d$-simplex, written with the usual convention $x_0 < x_1 < \ldots <x_d$. There are two cases: 
\begin{compactitem}[$\bullet$]
\item if $x_d=d+1$, then the missing simplex is $H_1^d$. Since $n \ge d+2$, $AD^d(n)$ contains the two faces $F=[1,2, \ldots, d, d+2]$ and $G=[2,3, \ldots, d+1, d+2]$, which have same size and same maximum. So by E-chordality $AD^d(n)$ contain $H^d_1$, a contradiction.
\item if $x_d > d+1$, there is an integer $i$ in $[x_d] - M$. Since $d \ge 2$, set $F' := M - \{x_0\} \cup \{i\}$ and $G' = M - \{x_1\} \cup \{i\}$. Since $i \le x_d$, $F'$ and $G'$ are $d$-faces with the same size and same maximum, namely, $x_d$. So by E-chordality, $AD^d(n)$ contains $M$, a contradiction.
\end{compactitem}

It remains to show that $AD^d(n)$ is skeleton-mid-chordal for all $d \ge 2$. This is easily shown by picking a labeling for which the missing face is the lexicographically last one; that is, $H^d_{n-d}$. 
\end{proof}

Since one might have the impression that skeleton-E-chordality is a rare property, we conclude with two ways of generating many skeleton-E-chordal complexes. 
The second is a generalization of the so-called ``sun graph'', a graph that is chordal but not interval. The original sun graph is in fact the $1$-skeleton of the $\operatorname{Sun}^2$ of Lemma \ref{lem:dsun} below.

\begin{lemma}[Wedges] \label{lem:wedges} Let $d, k, \ell$ be integers, with $d \ge k \ge 0$ and  $\ell >0$. Let $\operatorname{Wed}^d(k,\ell)$ be the join of a $k$-simplex $\Sigma$ with the disjoint union of $\ell$ distinct $(d-k-1)$-simplices $\tau_1, \ldots, \tau_\ell$. This $d$-dimensional simplicial complex on $(k+1)+\ell(d-k)$ vertices is skeleton-E-chordal under any labeling for which the first $k+1$ vertices  are those of $\Sigma$, the next $d-k$ vertices are those of $\tau_1$, the next $d-k$ are those of $\tau_2$, and so on, until the final $d-k$ vertices are those of $\tau_{\ell}$.
\end{lemma}

\begin{proof}
Let $F, G$ be any two distinct faces of same size and same maximum in the skeleton of $\operatorname{Wed}^d(k,\ell)$. If $\max F = \max G \le k+1$, then $F$ and $G$ are faces of the $k$-simplex $\Sigma$, and belong to all $d$-simplices in $\operatorname{Wed}^d(k,\ell)$. If instead $\max F = \max G > k+1$, then both $F$ and $G$ are faces of exactly one $d$-simplex in $\operatorname{Wed}^d(k,\ell)$, the same for both, which will be of the form $\Sigma \ast \tau_i$ for some $i \in \{1, \ldots, \ell\}$. 
Either way, $F$ and $G$ are in the same simplex of $\operatorname{Wed}^d(k,\ell)$. Since simplices are skeleton-E-chordal, that simplex (and thus also $\operatorname{Wed}^d(k,\ell)$) will contain any third face $H$ contained in $F \cup G$, and of the same size of $F$ and $G$. 
\end{proof}

\begin{lemma}[$d$-dimensional sun] \label{lem:dsun} Let $d \ge 1$.
Let $\operatorname{Sun}^d$ be the $d$-dimensional simplicial complex on $2d+2$ vertices and $d+2$ facets, obtained by `stacking' (i.e. coning off) each facet of the $d$-simplex $\Sigma_{d+1}$. Then $\operatorname{Sun}^d$ is skeleton-E-chordal with any labeling in which the labels from $1$ to $d+1$ are assigned to the vertices in $\Sigma_{d+1}$. 

\end{lemma}

\begin{proof}
Consider  $\operatorname{Sun}^d$ with a fixed labeling in which the lowest labels are reserved for the vertices of the original stacked simplex, which we can thus identify with $[d+1]$. Let $F,G$ be two faces of size $k$ and with same maximum $m$. There are two cases: 
\begin{compactenum}[(i)]
\item if $m \le d+1$, then $F$ and $G$ are both in the simplex $[d+1]$. So $[d+1]$, and thus $\operatorname{Sun}^d$, contains all size-$k$ subsets of $F \cup G$.
\item if $m > d+1$, by the labeling we chose $F$ and $G$ are both contained in the same ``stacking simplex'' $m \ast \sigma$, where $\sigma$ is some size-$d$ subset of $[d+1]$. So write $F=m \ast f$ and $G=m \ast g$, with $f,g$ in $\sigma$. Let $H$ be any size-$k$ subset of  $F \cup G = f \cup g \cup \{m\}$. If $H$ does not contain $m$, then $H \subseteq f \cup g$ is a face of $\sigma$, hence $H$ is in $\operatorname{Sun}^d$. If instead $H$ contains $m$, then it is of the form $H= m \ast h$, for some $h$ in $f \cup g$. Since $f \cup g \subseteq \sigma$, clearly $h$ is a face of $\sigma$, hence in $\operatorname{Sun}^d$; so $H=m \ast h$ is  in $\operatorname{Sun}^d$ as well. Either way, $H$ is in $\operatorname{Sun}^d$. \qedhere
\end{compactenum}
\end{proof}

%%%%%%%%%%%%%%%% SUBSECTION 2.2 %%%%%%%%%%%%%%%%%%%%%%%
\subsection{Stability under cones, links, deletions}
%%%%%%%%%%%%%%%%%%%%%%%%%%%%%%%%%%%%%%%%%%%%%%%%%

\label{sec:stability}
Here we investigate how the various chordality notions seen so far behave with respect to cones, deletions, and links. We start with three toy examples, all of dimension two:

\begin{figure}[t]
    \includegraphics[height=0.25\linewidth]{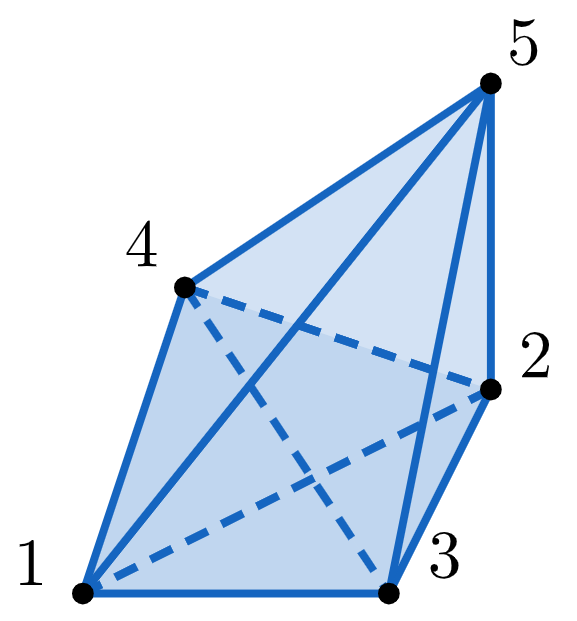}
    \hfill
    \includegraphics[height=0.25\linewidth]{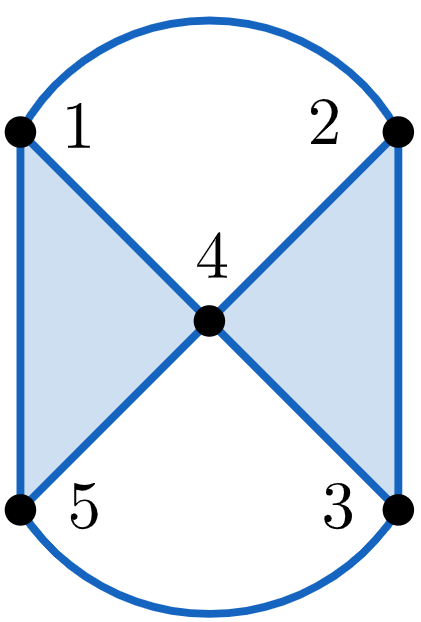}
    \hfill
    \includegraphics[height=0.24\linewidth]{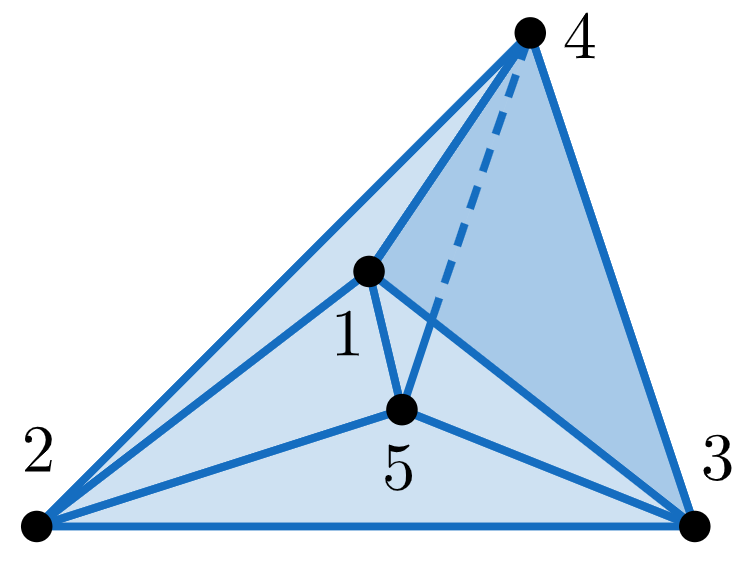}
    
    \caption{The simplicial complexes $I, J, K$ from Examples \ref{ex:WrongLinks}, \ref{ex:TwoTriangles}, and \ref{ex:SkeletonCliqueChordal}. The left complex $I$ consists of the boundary of the tetrahedron 1234, plus the cone (with apex $5$) over the edges 13, 23, 24, 14.} 
    \label{fig:ex19}
\end{figure}

\begin{example}\label{ex:WrongLinks}
The simplicial complex (in Figure \ref{fig:ex19}, left)
\[ I = 123, 124, 134, 135, 145, 234, 235, 245\]
is skeleton-mid-chordal (though not E-chordal), yet the link of $5$ is a $4$-cycle.  
\end{example}

\begin{example} \label{ex:TwoTriangles}
The $2$-dimensional E-chordal simplicial complex (in Figure \ref{fig:ex19}, center)
\[ J = 145, 234, 12, 35\]
has the property that if we delete vertex $4$, we are left with a $4$-cycle. 
\end{example}

\begin{example}\label{ex:SkeletonCliqueChordal}
The non-weakly-chordal simplicial complex (in Figure \ref{fig:ex19}, right)
\[ K = 124, 125, 134, 145, 235, 345 \]
is skeleton-clique-chordal with this labeling. Note that the link of $5$ is a $4$-cycle. 
\end{example}

\begin{proposition}
Let $\Delta$ be a $d$-dimensional simplicial complex that is geometrically-$d$-chordal. Then $v \ast \Delta$ is geometrically-$(d+1)$-chordal. However, if $v$ is an arbitrary vertex of $\Delta$, $\link(v,\Delta)$ need not be geometrically-$(d-1)$-chordal, and $\Delta':=\del(v,\Delta)$ need not be geometrically-$(\dim \Delta')$-chordal.
\end{proposition}

\begin{proof} Let $\Delta$ be a geometrically-$d$-chordal simplicial complex.
\begin{compactitem}
\item Every cone  is vacuously geometrically-$(d+1)$-chordal, as it contains no $(d+1)$-sphere.
\item Links: See Example \ref{ex:WrongLinks}.
\item Deletions: See Example \ref{ex:TwoTriangles}. \qedhere
    \end{compactitem}
\end{proof}

\begin{proposition} \label{prop:conesGeo}
Geochordality is preserved under cones and deletions, but not links.
\end{proposition}

\begin{proof} Let $\Delta$ be a geochordal $d$-dimensional complex.

\begin{compactitem}
\item Cones: Let $S$ be a $k$-sphere in $v \ast \Delta$ with the property that every $k$-face of $v \ast \Delta$ all of whose vertices lie in $S$ is also a face of $S$. If $S$ is disjoint from $v$, then $S$ is in $\Delta$; since $\Delta$ is geochordal, we conclude $S = \partial \Sigma_{k+2}$. If instead $S$ contains $v$, then some $k$-face $F$ of $S$ avoids $v$, since otherwise $S$ would be a cone. For every $(k-1)$-face $\tau$ of $F$, the $k$-face $v \ast \tau$ belongs to $v \ast \Delta$, and hence to $S$. Thus $S$ contains the boundary of the $(k+1)$-simplex on $F \cup \{v\}$. Since a triangulated sphere is a strongly connected pseudomanifold, it cannot properly contain another sphere of the same dimension. Hence $S = \partial \Sigma_{k+2}$.

\item Links: See Example \ref{ex:WrongLinks}.
\item Deletions: Any $S$  in  $\operatorname{del}(v,\Delta)$ is also  in $\Delta$. Since $S$ is disjoint from $v$, so is its chord.    \qedhere
    \end{compactitem}
\end{proof}

\begin{remark} Geochordality is also maintained under taking the $t$-skeleton. In fact, any $k$-sphere $S$ in the $t$-skeleton of $\Delta$, with $k\le t$, is also in $\Delta$.
Since both complexes have the same $k$-faces,
any chord of $S$ in $\Delta$ is also in its
$t$-skeleton.
\end{remark}

\newpage

\begin{proposition}\label{prop:EStability}
E-chordality is preserved under links, but not cones or deletions. \\However, the deletion of a vertex from a \emph{pure} E-chordal complex, is E-chordal.
\end{proposition}

\begin{proof} Fix a labeling that proves $\Delta$  E-chordal. 
\begin{compactitem}
\item Cones: See Lemma \ref{lem:SubdividedTriangle}, applied to $d=2$, $n=4$.
\item Links:  
Let $f, g$ be facets with the same size, same maximum in $\operatorname{link}(v,\Delta)$. Then $F = v \ast f$ and $G= v \ast g$ are facets of $\Delta$ with same maximum. By assumption $\Delta$ contains every face $H \subseteq F \cup G$ and of the same size of $F$ and $G$. In particular, $\Delta$ contains every face $h$ of size one less, contained in $f \cup g = F \cup G - \{v\}$. For any such $h$, since $v*h \subseteq F \cup G$ has the size of $F$, one has $v*h \in \Delta$ and thus $h \in \operatorname{link}(v,\Delta)$. 
\item Deletions: See Example \ref{ex:TwoTriangles}. 
If we know a priori that $\Delta$ is pure, and $f$ and $g$ are facets  of $\operatorname{del}(v,\Delta)$, then either $f$ and $g$ are both facets of $\Delta$, or $v \ast f$ and $v \ast g$ are both facets of $\Delta$. In the first case, the conclusion follows. In the second case, since $v \ast f$ and $v \ast g$ have also same size and same maximum, we conclude that any $H \subseteq f \cup g \cup \{v\}$ of the same size of $v \ast f$ is in $\Delta$. So in particular, any
$h \subseteq f \cup g$ of the same size of $f$ must be in $\Delta$.
\qedhere
\end{compactitem}
\end{proof}

\begin{proposition}\label{prop:SuperCLD}
Skeleton-E-chordality is preserved under links, deletions, but not cones.
\end{proposition}

\begin{proof} Fix a labeling that proves $\Delta$  skeleton-E-chordal. 
\begin{compactitem}
\item Cones: See Lemma \ref{lem:SubdividedTriangle}, applied to $d=2$, $n=4$.
\item Links:  Let $f, g$ be faces with the same size, same maximum in $\operatorname{link}(v,\Delta)$. Let $h \subseteq f \cup g$ be of the same size of $f$ and $g$. Since $\Delta$ is skeleton-E-chordal, $h \in \Delta$. Since $h \subseteq f \cup g$, and $f$ and $g$ are disjoint from $v$, also $h$ is. To prove $h \in \operatorname{link}(v, \Delta)$, it remains to exhibit a face $H$ of $\Delta$ containing $h$ and $v$. Indeed, $F = v \ast f$ and $G= v \ast g$ are faces of $\Delta$ with the same size, same maximum, so by skeleton-E-chordality $\Delta$ contains every face contained in $F \cup G$ and of the same size of $F$ and $G$. But one such face is $v \ast h$, since $F \cup G = f \cup g \cup \{v\}$. So $v \ast h$ is the face $H$ we desired. So $h$ is  in $\operatorname{link}(v,\Delta)$.
\item Deletions: Let $f, g$ be faces of same size, same maximum in $\operatorname{del}(v,\Delta)$. Let $h$ be any subset of $f \cup g$ of the same size of $f$ and $g$. Since $\operatorname{del}(v,\Delta) \subseteq \Delta$, $f$ and $g$ are also faces of same size, same maximum, of $\Delta$. By assumption, $\Delta$ contains every subset of $f \cup g$. So in particular, $h \in \Delta$. But since $f \cup g$ is disjoint from $v$, so is $h$. So $h$  is in  $\operatorname{del}(v,\Delta)$. 
\qedhere
\end{compactitem}
 
\end{proof}

\begin{proposition}\label{prop:MidStability}
Mid-chordality is preserved under cones, but not links or deletions. 
\end{proposition}

\begin{proof} Fix a labeling that proves $\Delta$  mid-chordal. 
\begin{compactitem}
\item Cones: In $\Delta \ast v$, relabel each vertex $i$ of $\Delta$ by $i+1$, and save the label $1$ for $v$. Now let $F, G$ be two facets  in $v \ast \Delta$ of same size and same maximum $M$. Being facets of a cone, $F, G$ are of the form $F= v \ast f$, $G= v\ast g$, with $f$, $g$ facets of $\Delta$.  Note that the maximum of $F$ cannot be $v$, since $v$ was assigned the lowest label; the same holds for $g$. Hence, $\max f = \max F = M = \max G = \max g$. Since $f,g$ are facets in $\Delta$ of same size and same maximum, any $2$-element subset $e$ of $f \cup g$ is in $\Delta$.  
Moreover, any such $e$ that avoids $M$, is contained in some face $h$ of $\Delta$ that also avoids $M$. Then $v * h$ is a face of $v \ast \Delta$ still containing $e$, but not $M$. If $h$ has the same size of $f$ and $g$, then $v \ast h$ has the same size of $F$ and $G$. This completes the proof for all $2$-element subsets $e$ of $f \cup g \cup \{v\} = F \cup G$ \emph{that do not contain $v$}. To conclude, we need to check the 2-element sets of the form $e=\{a, v\}$, with $a \in f \cup g$. The assumption that $f\ne g$ have the same size and the same maximum implies that the set $f\cup g-M$ has size at least two. Pick $b\ne a$ in this set. By the mid-chordality of $\Delta$ applied to  $e'=\{a,b\}$, there is a face $h'\in\Delta$ of the same size of $f$ with with $e'\subseteq h'\subseteq f\cup g-\{M\}$. Then $H_e=v*h'$ is the desired face of $v*\Delta$.
\item Links: See Example \ref{ex:WrongLinks}.
\item Deletions: See Example \ref{ex:TwoTriangles}. \qedhere

\end{compactitem}
\end{proof}

\begin{proposition}\label{prop:SkelMidStability}
Skeleton-Mid-chordality is preserved under cones and deletions, but not  links.
\end{proposition}

\begin{proof} Fix a labeling that proves $\Delta$  skeleton-mid-chordal. 
\begin{compactitem}
\item Cones: In $\Delta \ast v$, relabel each vertex $i$ of $\Delta$ by $i+1$, and save the label $1$ for $v$. Now let $F, G$ be two faces  in $v \ast \Delta$ of same dimension $k$ and same maximum. Up to swapping the labels of $F$ and $G$, there are three cases:
\begin{compactitem}
\item If $F,G$ are disjoint from $v$, then they are in $\Delta$, and the conclusion follows.
\item If $F, G$ both contain $v$, write them as $F= v \ast f$, $G= v\ast g$. Since $f,g$ are faces in $\Delta$ of same size and same maximum, any $2$-element subset of $f \cup g$ is in $v*\Delta$. It is easy to see that for all $a \in f \cup g$, the set $\{a, v\}$ is an edge of $v*\Delta$. Hence, every $2$-element subset of $f \cup g \cup \{v\} = F \cup G$ is in $v*\Delta$. Now let $e$ be any $2$-element subset of $f \cup g - \{ \max f\}$. By the assumption on $\Delta$, is contained in some face $h$ of $\Delta$ that does not contain $\max f$. Then $v * h$ is a face of $v \ast \Delta$ still containing $e$, but not $\max f$. If $h$ has the same size of $f$ and $g$, then $v \ast h$ has the same size of $F$ and $G$.
\item If $F$ contains $v$ and $G$ does not, write $F= v \ast f$. Since $v$ was assigned the lowest label, $\max f = \max F$. If $G$ has size two, the claim is clear; so, assume $G$ has size at least three. Let $g_1, \ldots, g_{k}$ be all the codimension-one faces of $G$ containing the vertex $\max G = \max F$. Clearly, the union of the $g_i$'s is $G$. Since $G$ is disjoint from $v$, so are the $g_i$'s. So $f, g_1, \ldots, g_k$ are faces in $\Delta$ with the same size, same maximum. From this it follows that any $2$-element subset of $f \cup G$ is in $\Delta$. But then it is easy to see that every $2$-element subset of $f \cup G \cup \{v\} = F \cup G$ is in $v \ast \Delta$.
Moreover, any $2$-element subset $e$ of $f \cup G -\{\max f\}$ that does not contain $\max f$, is contained in some face $h$ of $\Delta$ of the same size of $f$, and contained in $f \cup G - \{\max f\}$. But then $v \ast h$ contains $e$, is contained in $F \cup G - \{\max f\}$, and has the same size of $F$.
\end{compactitem}
\item Links: See Example \ref{ex:WrongLinks}.
\item Deletions: Let $F, G$ be faces of same size, same maximum in $\operatorname{del}(v,\Delta)$. Since $\operatorname{del}(v,\Delta) \subseteq \Delta$, $\Delta$ contains all $2$-element subsets of $F \cup G$, which are all disjoint from $v$, since $F, G$ are. 
Moreover, for each $2$-element subset $\{x,y\}$ of $F \cup G - \{\max F \}$, $\Delta$ contains some face $H_{x,y}$ of the same size of $F$ and $G$ contained in $F \cup G - \{\max F \}$ (hence still disjoint from $v$). Since all these $2$-element sets and the $H_{x,y}$ are disjoint from $v$, they belong to $\del(v, \Delta)$.
\qedhere

\end{compactitem}
\end{proof}

\begin{table}[t]
\centering    

\begin{tabular}{||c||c|c|c|c||} 
 \hline
maintained under... & Cones & Links& Deletions &Skeleta \\ [0.5ex] 
 \hline\hline
Geometrically-$d$- & 1 & 0 & 0 &0 \\
 \hline
 Geo- & 1 & 0 &  1 & 1 \\
 \hline
 E- & 0 & 1	 & $\phantom{*}1^*$ & 0\\
 \hline
  Skeleton-E- & 	0 & 1 & 1 & 1 \\ 
 \hline
 Mid- & 1 & 0	 & 0 & 0\\
 \hline
  Skeleton-Mid- & 	1 & 0 & 1 & 1 \\ 
 \hline
 (Very)-Weakly- & 1 & 0 & 0 & 0	 \\
 \hline
  Skeleton-(Very)-Weakly- & 1 & 0 & 1 & 1	 \\
 \hline
   Clique- & 1 & 0	 & $\phantom{*}1^*$ & 0\\
 \hline
  Skeleton-Clique- & 	1 & 0 & 1 & 1 \\ 
 \hline
 W- & 1 & 0 & 0 & 0 \\
 \hline
 Skeleton-W-  & 0 & 0 & 0 & 1 \\
 \hline

  Ridge- & 1 & 0 & 0 & 0 \\
 \hline
  Skeleton-Ridge- &1 &0  & 1 & 1 \\
 \hline
 Weakly-Ridge- & 1 & 0 & 0 & 0 \\
 \hline
  Skeleton-Weakly-Ridge- &1  & 0 & 1 & 1 \\
   \hline

\end{tabular}
\caption{Stability of the chordalities discussed in this paper. *: Valid for pure complexes only.\medskip}
\label{table:1}
\end{table}

\begin{proposition} \label{prop:weakStability}
Weak-chordality and very-weak-chordality are preserved under cones, but not under links or deletions. Conversely, if the cone over a simplicial complex $\Delta$ is weakly-chordal (resp. very-weakly-chordal), so is $\Delta$.\\ Moreover, the deletion of a weakly-simplicial (resp. very-weakly-simplicial) vertex from a pure simplicial complex maintains weak-chordality (resp. very-weak-chordality).
\end{proposition}

\begin{proof} Fix a labeling that proves $\Delta$ weakly-chordal (respectively, very-weakly-chordal).
\begin{compactitem}
\item Cones: Let $v$ be a new vertex. In $\Delta \ast v$, relabel each vertex $i$ of $\Delta$ by $i+1$, and assign label $1$ to $v$. Now let $F, G$ be two facets (respectively, two adjacent facets) in $v \ast \Delta$ of same size and same maximum. Write them as $F = v \ast f$ and $G=v \ast g$ for some $f,g$ in $\Delta$. Then $f$ and $g$ are same-size facets (respectively, adjacent facets) of $\Delta$. Moreover, since $v$ has label $1$, $\max f = \max F = \max G = \max g$. By assumption, $\Delta$ contains some face $h \subseteq f \cup g - \{\max f\}$. Setting $H :=v \ast h$, we are done.
\item Links: See Example \ref{ex:WrongLinks}.
\item Deletions: See Example \ref{ex:TwoTriangles}. 
\item Converse: 
Suppose $v*\Delta$ is weakly-chordal (resp. very-weakly-chordal), and let $f,g$ be facets of $\Delta$ of size $s$ and with the same maximum $m$. (In the very-weak case, assume also that $f, g$ are adjacent). Then $F=v*f$ and $G=v*g$ have same size, same maximum $M:=\max(v,m)$ (and in the very-weak case, they are also adjacent). By assumption,  $v*\Delta$ has a face $H\subseteq (f\cup g\cup\{v\})-\{M\}$ of size $s+1$ (with equality in the very-weak case). Now:
\begin{compactitem} 
\item If $v \in H$, since $M \notin H$ we have $v<m=M$. Then $h=H-\{v\}$ is the desired face of $\Delta$ contained in $f\cup g-\{m\}$. 
\item If $v \notin H$, then $H \in \Delta$.  Set $h=H-\{m\}$ if $m\in H$, and $h=H-\{x\}$ for an arbitrary $x\in H$ otherwise. Then $h$ is the desired face of $\Delta$ contained in $f\cup g-\{m\}$. (In the very-weak-case, $v \notin H$ actually leads to a contradiction: Since $H=(f\cup g\cup\{v\})-\{M\}$ does not contain $v$, it must be $M=v$ and $H=f \cup g$, which means that $f$ and $H$ are both facets of $\Delta$ with $f \subsetneq H$, a contradiction.)
\end{compactitem}
As for the last part: let  $v$ be weakly-simplicial in $\Delta$ weakly-chordal (the proof for the very-weakly- property is analogous). Let $f,g $ be facets of $\del(v, \Delta)$, of same size, same maximum. Up to swapping $f$ and $g$, there are three cases: 
\begin{compactenum}[(i)]
\item either $f, g$ are facets of $\Delta$, or  
\item $f$ and $g*v$ are facets of $\Delta$, or
\item $f * v$ and $g*v$ are facets of $\Delta$.
\end{compactenum}
Since $\Delta$ is pure, and $f$ and $g$ have same size, case (ii) is impossible. In case (i), if $f,g$ are facets of $\Delta$, there is a face $h$ in $\Delta$ of the same size of $f$ and contained in $f \cup g - \{\max f\}$, which is disjoint from $v$. So $h \in \del(v, \Delta)$. As for case (iii): the assumption on $v$ implies that there is a face $H$ in $\Delta$ disjoint from $v$, but contained in $F \cup G - \{v\} = f \cup g$. Let $h$ be $H - \{\max f\}$ if $H$ contains $\max f$; let $h$ be $H$ minus its largest vertex, otherwise. Then $h$ is disjoint from $v$, contained in $f \cup g - \{\max f\}$, and of the same size of $f$ and $g$. \qedhere
\end{compactitem}
\end{proof}

\begin{proposition}\label{prop:SkelWeakStability}
Skeleton-weak-chordality and skeleton-very-weak-chordality are preserved under cones and deletions, but not links.
\end{proposition}

\begin{proof} Fix a labeling that proves $\Delta$ skeleton-weakly- (respectively, skeleton-very-weakly) -chordal.  
\begin{compactitem}
\item Cones: Let $v$ be a new vertex. To prove that $v \ast \Delta$ is skeleton-weakly-chordal (respectively, skeleton-very-weakly-chordal), we relabel each vertex $i$ by $i+1$, and assign label $1$ to $v$. Let $F, G$ be two faces (respectively, two adjacent faces) of same size and same maximum in $v \ast \Delta$. We need to show that $v\ast \Delta$ contains a third face $H \subseteq F \cup G - \max F$ of the same size as $F$. There are three cases:
\begin{compactitem}[--]
\item if $F, G$ are in $\Delta$,  the assumption of $\Delta$ provides one such $H$.
\item if $F, G$ are not in $\Delta$, write them as $F = v \ast f$ and $G=v \ast g$ for some $f,g$ in $\Delta$. Since $v$ is assigned label $1$, $\max f = \max F = \max G = \max g$. So $\Delta$ contains a face $h$ with the same size as $f$, such that $h \subseteq f \cup g - \max f$. Setting $H :=v \ast h$, we conclude.
\item if $F$ is in $\Delta$ and $G=v \ast g$ with $v \notin F$, set $H := v \ast (F - \{\max F\})$.  We have $H \subseteq F \cup G - \{ \max F\}$, with equality in the very-weakly-chordal case. 
\end{compactitem}
\item Links: See Example \ref{ex:WrongLinks}.

\item Deletions: Let $F, G$ be faces (respectively, adjacent faces) of same size, same maximum in $\operatorname{del}(v,\Delta)$. Since $\operatorname{del}(v,\Delta) \subseteq \Delta$, $\Delta$ contains some $H \subseteq F \cup G - \{\max F\}$ of the same size as $F$. Since $F \cup G$ is disjoint from $v$, this $H$ is also disjoint from $v$. \qedhere
\end{compactitem}
\end{proof}

\begin{proposition}\label{prop:CliqueStability}
Clique-chordality is preserved under cones, but not under links or deletions. Conversely, if the cone over a simplicial complex $\Delta$ is clique-chordal, so is $\Delta$. \\Moreover, the deletion of a vertex from a \emph{pure} simplicial complex maintains clique-chordality.
\end{proposition}

\begin{proof} Fix a labeling that proves $\Delta$  clique-chordal. 
\begin{compactitem}
\item Cones: In $\Delta \ast v$, relabel each vertex $i$ of $\Delta$ by $i+1$, and save the label $1$ for $v$. Now let $F, G$ be two facets  in $v \ast \Delta$ of same size and same maximum. Write them as $F= v \ast f$, $G= v\ast g$, with $f$, $g$ facets of $\Delta$.  Since $v$ was assigned label $1$, $\max f = \max F = \max G = \max g$. Since $f,g$ are facets in $\Delta$ of same size and same maximum, any $2$-element subset of $f \cup g$ is in $\Delta$. Hence, every $2$-element subset of $f \cup g \cup \{v\} = F \cup G$ is in $v*\Delta$.
\item Links: See Example \ref{ex:SkeletonCliqueChordal}.
\item Deletions: See Example \ref{ex:TwoTriangles}. 
\item Converse: Suppose $v * \Delta$ is clique-chordal. If $f, g$ are facets of $\Delta$ with the same size and the same maximum, then $v *f$ and $v* g$ are facets of $v*\Delta$ with the same size and the same maximum. By assumption, all $2$-element subsets of $f \cup g \cup \{v\}$ are in $v*\Delta$. Since $f \cup g$ is disjoint from $v$,
all $2$-element subsets of $f \cup g$ are in $\Delta$. So $\Delta$ is clique-chordal.
\item As for the last claim: Suppose $\Delta$ is pure clique-chordal.  Let $f$ and $g$ be facets of $\operatorname{del}(v,\Delta)$ with the same size, same maximum. Then either $f$, $g$ are both facets of $\Delta$ (and the conclusion follows easily), or $v \ast f$ and $v \ast g$ are both facets of $\Delta$. In the latter case, since $v \ast f$ and $v \ast g$ have also same size and same maximum, we conclude that any $2$-element subset of $f \cup g \cup \{v\}$ is in $\Delta$. In particular, any $2$-element subset of $f \cup g$ is in $\Delta$. 
\qedhere
\end{compactitem}

\end{proof}

\begin{proposition}\label{prop:SkelCliqueStability}
Skeleton-clique-chordality is preserved under cones and deletions, but not under links.
\end{proposition}

\begin{proof} Fix a labeling that proves $\Delta$  skeleton-clique-chordal. 
\begin{compactitem}
\item Cones: In $\Delta \ast v$, relabel each vertex $i$ of $\Delta$ by $i+1$, and save the label $1$ for $v$. Now let $F, G$ be two faces  in $v \ast \Delta$ of same dimension $k$ and same maximum. There up to swapping the labels of $F$ and $G$, there are three cases:
\begin{compactitem}
\item If $F,G$ are disjoint from $v$, then they are in $\Delta$, and the conclusion follows.
\item If $F, G$ both contain $v$, write them as $F= v \ast f$, $G= v\ast g$. Since $f,g$ are faces in $\Delta$ of same size and same maximum, any $2$-element subset of $f \cup g$ is in $\Delta$. Hence, every $2$-element subset of $f \cup g \cup \{v\} = F \cup G$ is in $v*\Delta$.
\item If $F$ contains $v$ and $G$ does not, write $F= v \ast f$. Since $v$ was assigned the lowest label, $\max f = \max F$. If $|G|=2$ the claim is immediate. If $|G|\ge3$ let $g_1, \ldots, g_{k}$ be all the codimension-one faces of $G$ containing the vertex $\max G = \max F$. Clearly, the union of the $g_i$'s is $G$. Since $G$ is disjoint from $v$, so are the $g_i$'s. So $f, g_1, \ldots, g_k$ are in $\Delta$, they are all of the same size, and they have  same maximum. From this it follows that any $2$-element subset of $f \cup G$ is in $\Delta$. But then every $2$-element subset of $f \cup G \cup \{v\} = F \cup G$ is in $v\ast\Delta$.
\end{compactitem}
\item Links:  See Example \ref{ex:SkeletonCliqueChordal}.
\item Deletions: Let $f, g$ be faces of same size, same maximum in $\operatorname{del}(v,\Delta) \subseteq \Delta$. Let $h$ be any $2$-element subset of $f \cup g$.  By assumption, $\Delta$ contains $h$. But since $f \cup g$ is disjoint from $v$, so is $h$. So $h$  is in  $\operatorname{del}(v,\Delta)$. 
\qedhere
\end{compactitem}

\end{proof}

%%%%%%%%%%%%%%%% SUBSECTION 2.3 %%%%%%%%%%%%%%%%%%%%%%%
\subsection{Chordal versus Interval and Unit-Interval}
%%%%%%%%%%%%%%%%%%%%%%%%%%%%%%%%%%%%%%%%%%%%%%%%%%%%%

Chordality is not the only graph theoretical property that can be characterized in terms of vertex labelings. \emph{(Unit)-Interval graphs} are the intersection graphs of a configuration of $n$ open (length-one) intervals on one real line. \emph{Co-comparability graphs} are the intersection graphs of $n$ continuous curves connecting two parallel lines. 
These well-known notions have a connection to chordality, explained by the following theorem:

\begin{theorem}[Gilmore-Hoffman] \label{thm:GilHof}
$G$ is interval if and only if $G$ is chordal and co-comparability.
\end{theorem}

The next well known Proposition, due to Olariu and other authors (cf.~\cite{BSV22}), paves the way for extending these graph properties to higher-dimensional simplicial complexes: 

\begin{proposition} Let $G$ be a graph on $n$ vertices.
\begin{compactitem}
\item $G$ is \emph{unit-interval} if and only if it has a vertex labeling such that, for all $a<c$ such that $ac$ is an edge of $G$,
all size-$2$ subsets of $\{a, a+1, \ldots, c-1, c\}$ are edges of $G$.
\item $G$ is \emph{interval} if and only if it has a vertex labeling such that, for all $a<b<c$, if $ac$ is an edge of $G$, so is $ab$.
\item $G$ is \emph{co-comparability} if and only if it has a vertex labeling such that, for all $a<b<c$, if $ac$ is an edge of $G$, then at least one of $ab$ and $bc$ is an edge.
\end{compactitem}
\end{proposition}

So, here comes the generalization to simplicial complexes. The following definitions are basically due to Benedetti-Seccia-Varbaro \cite{BSV22}, who focused only on the pure case. Unit-interval complexes were also introduced independently by Almousa-VandeBogert \cite{AV22}.

\begin{definition}[Unit-interval, underclosed, weakly-closed complexes]
A simplicial $d$-complex $\Delta$ with $n$ vertices, not necessarily pure, is:
\begin{compactitem}
\item \emph{unit-interval}, if it has a vertex labeling such that for every facet $F$ in $\Delta$, if $s$ is the size of $F$, then $\Delta$ also contains all size-$s$ subsets of $ \{\min F, \min F +1, \ldots, \max F -1, \max F\}$;
\item \emph{underclosed} or \emph{interval}, if it has a vertex labeling such that for every facet $F$ in $\Delta$, for any face $G$ of $\Sigma_n$ of the same size of $F$, if $\min F = \min G$ and $G \le F$ componentwise, then $G$ is also in $\Delta$;
\item \emph{weakly-closed} or \emph{co-comparability}, if it has a vertex labeling such that for every facet $F$ in $\Delta$, for every integer $g \notin F$ with $\min F < g < \max F$, $\Delta$ also contains some face $G$ of the same size of $F$, adjacent to $F$, and containing $g$.
\end{compactitem} 
\end{definition}

\begin{remark}\label{rem:BSV}
None of the chordality properties discussed in this paper is strong enough to imply the underclosed property; not even if you assume the weakly-closed property.  The counterexample is the simplicial complex  $123, 256, 345, 346, 347, 356, 456$ from \cite[Proposition 37]{BSV22}, which is weakly-closed, but not underclosed. Interestingly, this simplicial complex  is skeleton-E-chordal with a different labeling of it, namely, 
\[123, 124, 134, 135, 167, 234, 246.\]
Are any of the chordality properties implied by underclosedness?
Here is a very recent result by Dochtermann--Goeckner--Pavelka \cite[Theorem 4.9]{DGP26}:
\begin{theorem}[{Dochtermann--Goeckner--Pavelka \cite[Theorem 4.9]{DGP26}}] All underclosed complexes are W-chordal.
\end{theorem}
We expand on this connection below. First of all, let us explore which of our families of examples are underclosed:
\end{remark}

\begin{lemma} \label{lem:underclosed} In each dimension $d \ge 2$,
\begin{compactenum}[\rm (i)] 
\item The non-E-chordal three-facet simplicial complex $C^d$ of Lemma {\em\ref{lem:ThreeFacets}} is
underclosed.
\item For all $k>d$, the non-mid-chordal  $D^d(d+k)$ of Lemma {\em\ref{lem:WeaklyNonMidChordal}} is underclosed.
\item The simplicial complex $AD^d(n)$ of Lemma {\em\ref{lem:SubdividedTriangle}} is underclosed for all $n \ge d+2$.
\item If  $k$ in $\{0, \ldots, d-2\}$, the simplicial complex $\operatorname{Wed}^d(k,\ell)$ of Lemma {\em\ref{lem:wedges}} is underclosed if and only if $\ell \le 2$.
\item the skeleton-E-chordal simplicial complex $\operatorname{Sun}^d$ of Lemma {\em\ref{lem:dsun}} is not underclosed.
\end{compactenum}
\end{lemma}

\begin{proof}
\begin{compactenum}[\rm (i)]  
\item With the given labeling, the only facet with non-consecutive vertices is $G$, and the only facet below it is $H^d_1$, which is in $C^d$.
\item With the given labeling, the facet $G$ has consecutive vertices, so it can be neglected; the rest of the complex is already proven underclosed in  \cite[Lemma 44]{BSV22}.
\item Choose any labeling in which the vertices of the missing $d$-face have the highest labels. 
\item  When $\ell =2$, any labeling for which the first $d-k$ vertices are those of $\tau_1$, the next $k+1$ are those of $\Sigma$, and the final $d-k$ are those of $\tau_2$, is underclosed. For arbitrary $\ell$: Since none of its facets has adjacent facets, the only chance for $\operatorname{Wed}^d(k,\ell)$ to be  weakly-closed is a labeling that uses consecutive vertices on each $d$-simplex. This is possible only if $\ell \le 2$. 
\item Let $C$ be the central $d$-simplex, and
let $S_i=(C-\{c_i\})\cup\{x_i\}$ be
the $d+1$ stacked facets of
$\operatorname{Sun}^d$.
Suppose an underclosed labeling exists.
If $F=\{f_0<\cdots<f_d\}$ is a facet
and $g\notin F$ satisfies
$f_{t-1}<g<f_t$ for some $t\ge1$,
then $(F-\{f_t\})\cup\{g\}$ is a face
adjacent to $F$, by underclosedness.
Since $C$ is the only facet adjacent
to $S_i$, no $x_j$ with $j\ne i$
can lie strictly between the smallest
and largest vertices of $S_i$.
Choose three stacked vertices
$x_i<x_j<x_\ell$.
Since $x_i\in S_i$ and $x_j$
cannot lie between its extreme
vertices, all vertices of $S_i$
are smaller than $x_j$.
Similarly, all vertices of $S_\ell$
are larger than $x_j$.
This is impossible, because
$S_i\cap S_\ell=C-\{c_i,c_\ell\}$
is nonempty for $d\ge2$.
\qedhere 
\end{compactenum}
\end{proof}

\begin{lemma}[{Benedetti-Seccia-Varbaro \cite[Lemma 41]{BSV22}}]
 Any labeling that proves a simplicial complex underclosed (resp. unit-interval), proves it also for its skeleta.
\end{lemma}

\begin{proof}
Since this was claimed without proof and only in the pure case in \cite[Lemma 41]{BSV22}, for convenience we include a proof here. 
\begin{compactenum}[(1)]
    \item Let $\Delta$ be a simplicial complex with an underclosed labeling. Let $A = [a_0, a_1, \ldots, a_{\ell}]$ and $B = [b_0, b_1, \ldots, b_{\ell}]$  be $\ell$-dimensional faces of $\Sigma_n$, with $a_0 = b_0$ and $a_i \le b_i$ for all $i$. 
    Suppose $B$ is a face of $\Delta$. 
    If $B$ is a facet, then  $A$ is in $\Delta$ by the underclosed condition. Otherwise, $B$ is contained in some facet $F = [f_0, \ldots, f_h]$, for some $h > \ell$. Let $m$ be the smallest of the positive integers $i$ for which $a_i < b_i$. 
    Let $B'$ be the face of $\Sigma_n$ obtained from $B$ by replacing $b_m$ with $a_m$. Clearly, $A \le B' \le B$ componentwise. (Possibly $A=B'$). There are two cases: 
    \begin{compactitem}
    \item If $a_m \in F$, then $B'$ is in $F$, so $B'$ is in $\Delta$;
    \item If $a_m \notin F$,  let $F':= F - \{b_m\} \cup \{a_m\}$. Since $F' \le F$, by the underclosed property $F' \in \Delta$. Since $B'$ is contained in $F'$, $B'$ is in $\Delta$ as well. 
    \end{compactitem}    
    Either way, $B' \in \Delta$. Note that the lowest $i$ for which $a_i < b'_i$ is now $m+1$; in other words, $B'$ is ``one step closer to $A$ than $B$ was''. Now repeat this reasoning with $B'$ replacing $B$. After a finite number of steps,  we obtain that $A \in \Delta$.

    \item Let $\Delta$ be a simplicial complex with a unit-interval labeling. Let $B$ be an $\ell$-dimensional face of $\Delta$, and let $A$ be any size-$(\ell+1)$ subset of
    $X=\{\min B,\min B+1,\ldots,\max B\}$.
    Choose a facet $F$ containing $B$, and let
    $Y=\{\min F,\min F+1,\ldots,\max F\}$.
    Since $A\subseteq Y$ and $|A|\le |F|$, we can extend $A$ to a size-$|F|$ subset $G$ of $Y$. By the unit-interval condition, $G\in\Delta$, and hence $A\in\Delta$. \qedhere
   
\end{compactenum}
\end{proof}

\begin{theorem} \label{thm:hierarchy2b} Let $d \ge 1$.
\begin{compactenum}[\rm (i)]
\item All unit-interval $d$-dimensional simplicial complexes are skeleton-E-chordal.
\item All underclosed $d$-dimensional simplicial complexes are skeleton-weakly-chordal. 
\end{compactenum}
In each dimension, the inclusions above are strict. 
\end{theorem}

\begin{proof} 
Let $\Delta$ be a simplicial complex. In view of the previous Lemma, it suffices to prove:
\begin{compactenum}[(i)]
\item any labeling that proves $\Delta$  unit-interval, proves it also E-chordal.
\item any labeling that proves $\Delta$ underclosed, proves it also weakly-chordal.
\end{compactenum}
Here are the proofs:
\begin{compactenum}[(i)]
\item Fix one such labeling. Let $F$ and $G$ be facets of $\Delta$ of same size $k$ and same maximum $m$. Up to swapping their names, we can assume $\min F\le \min G$. Then 
\[ F \cup G \subseteq \{ \min F, \min F + 1, \ldots, m-1, m= \max F\}.\]
Since $\Delta$ is unit-interval, and $F$ is a facet of $\Delta$, all the size-$k$ subsets of the right-hand side above are in $\Delta$. In particular, all size-$k$ subsets of $F \cup G$ belong to $\Delta$. 
\item Fix one such labeling. Let $F$ and $G$ be two $\ell$-faces of $\Delta$. Let us order the elements of $F \cup G$ increasingly. Let $H$ be the set formed by the first (i.e. lowest) $\ell +1$ elements. By construction, $H \le F$, $H \le G$, and $H$ does not contain $\max F$. So $H$ is contained in $F \cup G - \{\max F\}$. 
Let $X$ be whichever of $F,G$ has the smaller minimum. Then $\min H=\min X$ and $H\le X$ componentwise, so $H\in\Delta$ by the underclosed assumption.
\end{compactenum}
As for the strictness: In dimension $1$, interval graphs are well-known to be a proper subclass of chordal graphs (and unit-interval graphs are even fewer). In each dimension $\ge 2$, Lemma \ref{lem:underclosed}, parts (iv) and (v), yields infinitely many skeleton-weakly- and even skeleton-E-chordal complexes that are not underclosed, so in particular not unit-interval. 
\qedhere
\end{proof}

We conclude with a proposed partial generalization of Theorem \ref{thm:GilHof}:
\begin{corollary}
All underclosed complexes are weakly-closed and skeleton-weakly-chordal. \\ The converse is false.
\end{corollary}

\begin{proof}
That all pure underclosed complexes are weakly-closed is shown in \cite{BSV22}; the same proof extends to the non-pure setup. For the strictness, see  Remark \ref{rem:BSV} above.    
\end{proof}

\begin{remark}
Theorem \ref{thm:hierarchy2b} is best possible, in the sense that Lemma \ref{lem:underclosed}, part (ii),  yields examples of underclosed complexes that are not  mid-chordal. The  Corollary instead can be improved by replacing the ``weakly-closed'' conclusion  with a ``semi-closed'' one, cf.~\cite{BSV22} for the definition. 
\end{remark}

\newpage

%%%%%%%%%%%%%%%% SECTION 3 %%%%%%%%%%%%%%%%%%%%%%%
%%%%%%%%%%%%%%%% SECTION 3 %%%%%%%%%%%%%%%%%%%%%%%
\section{Chordality via simplicial vertices}
%%%%%%%%%%%%%%%% SECTION 3 %%%%%%%%%%%%%%%%%%%%%%%
%%%%%%%%%%%%%%%% SECTION 3 %%%%%%%%%%%%%%%%%%%%%%%
\label{sec:SimplicialVertices}
A crucial property of chordal graphs, first noticed by Dirac \cite[Theorem 4]{Ber61}, is the presence  of \emph{simplicial} vertices, i.e. vertices whose neighbors form a clique. 
(See Hlin\v{e}n\'{y} \cite{Hli03} for a new proof.)  Simplicial vertices may appear also in non-chordal graphs (e.g. a $4$-cycle with an extra leaf). However, they lead to another characterization of chordality as follows:

\begin{theorem}[{essentially Dirac \cite{Dir61}}] \label{thm:charsimplicial}
A graph $G$ is chordal if and only if every nonempty induced subgraph of $G$ has a simplicial vertex.
\end{theorem}

In this section, we discuss how to generalize this to higher dimensions.

%%%%%%%%%%%%%%%% SUBSECTION 3.1 %%%%%%%%%%%%%%%%%%%%%%%
\subsection{Ten types of simplicial vertices}
%%%%%%%%%%%%%%%% SUBSECTION 3.1 %%%%%%%%%%%%%%%%%%%%%%%

\begin{definition}[P-simplicial vertices] 
A vertex $v$ in a simplicial complex $\Delta$ is called:
\begin{compactitem}
\item \emph{E-simplicial}, if for any two facets $F \ne G$ of $\Delta$ of the same size that contain $v$, $\Delta$ contains all faces  with vertices from the set $F \cup G$, and of the same size of $F$ (and $G$).
\item \emph{mid-simplicial}, if for any two facets $F \ne G$ of $\Delta$ of the same size that contain $v$, $\Delta$ contains each $2$-element subset $e$ of $F \cup G - \{v\}$, and in addition, for each such $e$, some face $H_e$ of the same size of $F$ such that 
$e \subseteq H_e \subseteq F \cup G - \{v\}$.
\item \emph{weakly-simplicial}, if for any two facets $F \ne G$ of $\Delta$ of the same size that contain $v$, $\Delta$ contains some face $H$ with vertex set contained in $F \cup G - \{v\}$, and of the same size of $F$.
\item \emph{very-weakly-simplicial}, if for any two adjacent facets $F \ne G$ of $\Delta$ of the same size that contain $v$, $\Delta$ also contains the unique face $H$ with vertex set equal to $F \cup G - \{v\}$.
\item \emph{clique-simplicial}, if for any two facets $F \ne G$ of $\Delta$ of the same size that contain $v$, $\Delta$ contains all size-2 subsets of $F \cup G$.
\item \emph{skeleton-E-simplicial}, if for any two faces $f \ne g$ of $\Delta$ of the same size that contain $v$, $\Delta$ contains all faces contained in $f \cup g$, and of the same size of $f$.
 \item \emph{skeleton-mid-simplicial}, if for any two faces $f \ne g$ of $\Delta$ of the same size that contain $v$, $\Delta$ contains each $2$-element subset $e$ of $f \cup g - \{v\}$, and in addition, for each such $e$, some face $h_e$ of the same size of $f$ such that $e \subseteq h_e \subseteq f \cup g - \{v\}$.
\item \emph{skeleton-weakly-simplicial}, if for any two faces $f \ne g$ of $\Delta$ of the same size that contain $v$, $\Delta$ contains some face $h$  with vertex set contained in $f \cup g - \{v\}$, of the same size of $f$.
\item \emph{skeleton-very-weakly-simplicial}, if for any two adjacent faces $f \ne g$ of $\Delta$ of the same size that contain $v$, $\Delta$ also contains the unique face $h$ with vertex set equal to $f \cup g - \{v\}$.
\item \emph{skeleton-clique-simplicial}, if for any two faces $f \ne g$ of $\Delta$ of the same size that contain $v$, $\Delta$ contains all size-2 subsets of $f \cup g$.
\end{compactitem}
\end{definition}

Clearly, E- implies mid- implies weakly- implies very-weakly-simplicial, and the same is true with a skeleton- in front. It is a nice exercise to see that all these implications are strict.  

\begin{definition} Let $d\ge 1$. Let $\Delta$ be a  simplicial complex with vertex set $[n]$. $\Delta$ is called
\begin{compactitem}
\item \emph{flag}, if any clique $X \subset [n]$ is a face in $\Delta$. 
\item  \emph{subflag}, if any clique $X \subset [n]$ of size $\le \dim \Delta +1$ is a face in $\Delta$.
\end{compactitem}
\end{definition}

The word “subflag” is new, but we are not the first ones to study this concept: See, e.g., \cite{GKN11, Sta24}. Clearly, all flag complexes are subflag. Also, all graphs are subflag. Any graph containing a triangle  is not flag. 
The next Lemma is an easy exercise: 

\begin{lemma}\label{lem:SubflagChordal}
For subflag simplicial complexes,  clique-chordal is the same as E-chordal.\\ Similarly, skeleton-clique-chordal is the same as skeleton-E-chordal. 
\end{lemma}

Rather than discussing the nuances of the notions above, we are interested in their similarities. So we will carry  all  these notions along, and see what we can prove with any of them.

\begin{proposition}  \label{prop:SimplicialVertices} Let $\Delta$ be a $d$-dimensional simplicial complex. Let $\wp$ and $\wp'$ be the lists
\[\wp=\{ \textrm{E-, mid-, weakly-, very-weakly-, clique-chordal}\},\]
\[\wp'=\{ \textrm{ skeleton-E-, skeleton-mid-, skeleton-weakly-, skeleton-very-weakly-, skeleton-clique-chordal}\}.
\] 
For each $P$ in $\wp \cup \wp'$, if $\Delta$ is P-chordal, then it has a P-simplicial vertex.

The converse is false for each $P$. 

\end{proposition}

\begin{proof} In any labeling that proves property $P$,  vertex $n$ is always  $P$-simplicial.
As for the converses: The same proof of Lemma \ref{lem:annulus} actually shows that the annulus $A^d(n)$ does not have any very-weakly- or clique-simplicial vertices.  Now take the disjoint union of $A^d(n)$ with the $d$-dimensional simplex $\Sigma_{d+1}$. The resulting complex has exactly $d+1$ $P$-simplicial vertices (namely, all vertices of $\Sigma_{d+1}$), without being very-weakly-chordal or clique-chordal.
\end{proof}

\begin{remark} \label{rem:ChordalNotBisimplicial} A famous result  by Dirac is that every chordal graph on $n \ge 2$ vertices different than $K_n$ is \emph{bisimplicial}, i.e. it has at least \emph{two} non-adjacent simplicial vertices, cf.~\cite{Hli03}. Hence a natural curiosity is whether the conclusion of Proposition \ref{prop:SimplicialVertices} can be strengthened to ``it has at least two P-simplicial vertices". The answer, as we shall see later on, is `yes' if $P$ is skeleton-clique-chordality (cf.~Lemma \ref{lem:DiracBiSkClSimplicial}), but `no' for all other properties (cf.~Proposition \ref{prop:ChordalNotBisimplicial}).
\end{remark}

\begin{lemma} \label{lem:inducedsubcomplex} Let $\Delta$ be a simplicial complex.
Let $\wp, \wp'$ be as in Prop.~\ref{prop:SimplicialVertices}.
Let $P \in \wp \cup \wp'$.
If  the deletion of any finite number of vertices from $\Delta$ has a P-simplicial vertex, then $\Delta$ is P-chordal.
\end{lemma}

\begin{proof} We proceed by induction on the number $n$ of vertices of $\Delta$. Viewing $\Delta$ as the deletion of zero vertices from $\Delta$, by assumption  $\Delta$ has a P-simplicial vertex $v$. Let $\Delta_1 = \del(v, \Delta)$. Every complex obtainable by deleting a finite set $S$ of  vertices from $\Delta_1$ is also obtainable by deleting the set $S \cup \{v\}$ of vertices from $\Delta$, and thus by assumption has a P-simplicial vertex. Thus by inductive assumption $\Delta_1$ is P-chordal. But then so is $\Delta$, if we extend to $\Delta$ the labeling that makes $\Delta_1$ P-chordal by using the label `$n$' for the vertex $v$.
\end{proof}

So here comes our promised generalization(s) of Theorem \ref{thm:charsimplicial}: 

\begin{theorem} \label{thm:SimplicialVertices} Let $\Delta$ be a simplicial complex. Let $\wp'$ be as in Proposition \ref{prop:SimplicialVertices}. Let $P \in \wp'$.\\
$\Delta$ is P-chordal $\Longleftrightarrow$  every nonempty induced subcomplex of $\Delta$ has a P-simplicial vertex.
\end{theorem}

\begin{proof}
In Section \ref{sec:stability} we proved that the following properties are maintained under vertex deletions: skeleton-E- (cf. Prop. \ref{prop:SuperCLD}), skeleton-mid-(Prop.~\ref{prop:SkelMidStability}), skeleton-(very)-weak- (Prop.~\ref{prop:SkelWeakStability}), and skeleton-clique-chordality (Prop.~\ref{prop:SkelCliqueStability}). Thus for all $P$ in $\wp'$, every induced subcomplex of a P-chordal complex, being itself P-chordal, has a P-simplicial vertex by Proposition \ref{prop:SimplicialVertices}. The converse is established by Lemma \ref{lem:inducedsubcomplex}.
\end{proof}

\begin{nonexample} The Woodroofe complex $B^2=123, 345, 567, 178$ introduced in Remark \ref{rem:SkeletonChordalWeaklyChordal} is E-chordal, but its induced subcomplex on the odd-labeled vertices is a $4$-cycle, which has no simplicial vertex. Hence the previous theorem holds for any $P$ in $\wp'$, but not for any $P$ in $\wp$.
\end{nonexample}

%%%%%%%%%%%%%%%% SUBSECTION 3.2 %%%%%%%%%%%%%%%%%%%%%%%
\subsection{W-chordality}
%%%%%%%%%%%%%%%% SUBSECTION 3.2 %%%%%%%%%%%%%%%%%%%%%%%
\label{sec:Wchordality}

A more matroidal approach, with algebraic applications, was taken by Woodroofe \cite{Woo11}. In his honor, the resulting chordality property is usually abbreviated with a `W' in front.

\begin{definition}
Let $\Delta$ be a simplicial complex. A vertex $v$ of $\Delta$ is \emph{W-simplicial} if for every two distinct facets $F$ and $G$ of $\Delta$ that contain $v$, not necessarily of the same size, $\Delta$ also contains a third facet $H$ contained in $F \cup G - \{v\}$.
\end{definition}

\begin{proposition} \label{prop:PureW} If $\Delta$ is pure, W-simplicial is the same as weakly-simplicial.
\end{proposition}

\begin{proof}
Since $F$, $G$ and $H$ are facets of $\Delta$, purity forces them all to have the same size.
\end{proof}

\begin{figure}[t]
    \includegraphics[height=0.22\linewidth]{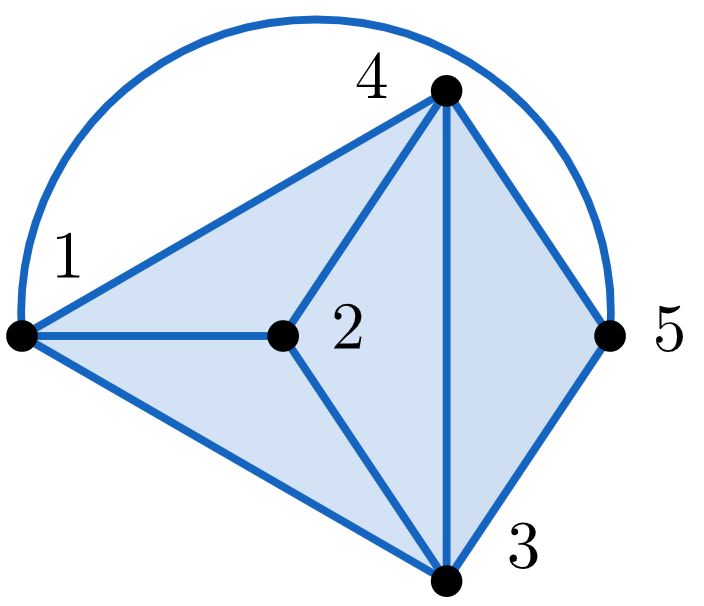}
    \hfill
    \includegraphics[height=0.22\linewidth]{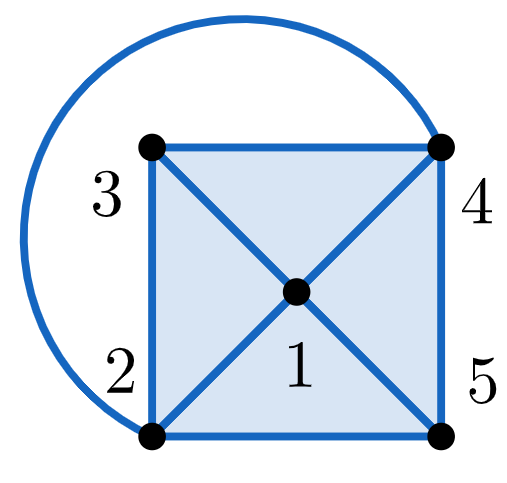}
    \hfill
    \includegraphics[height=0.22\linewidth]{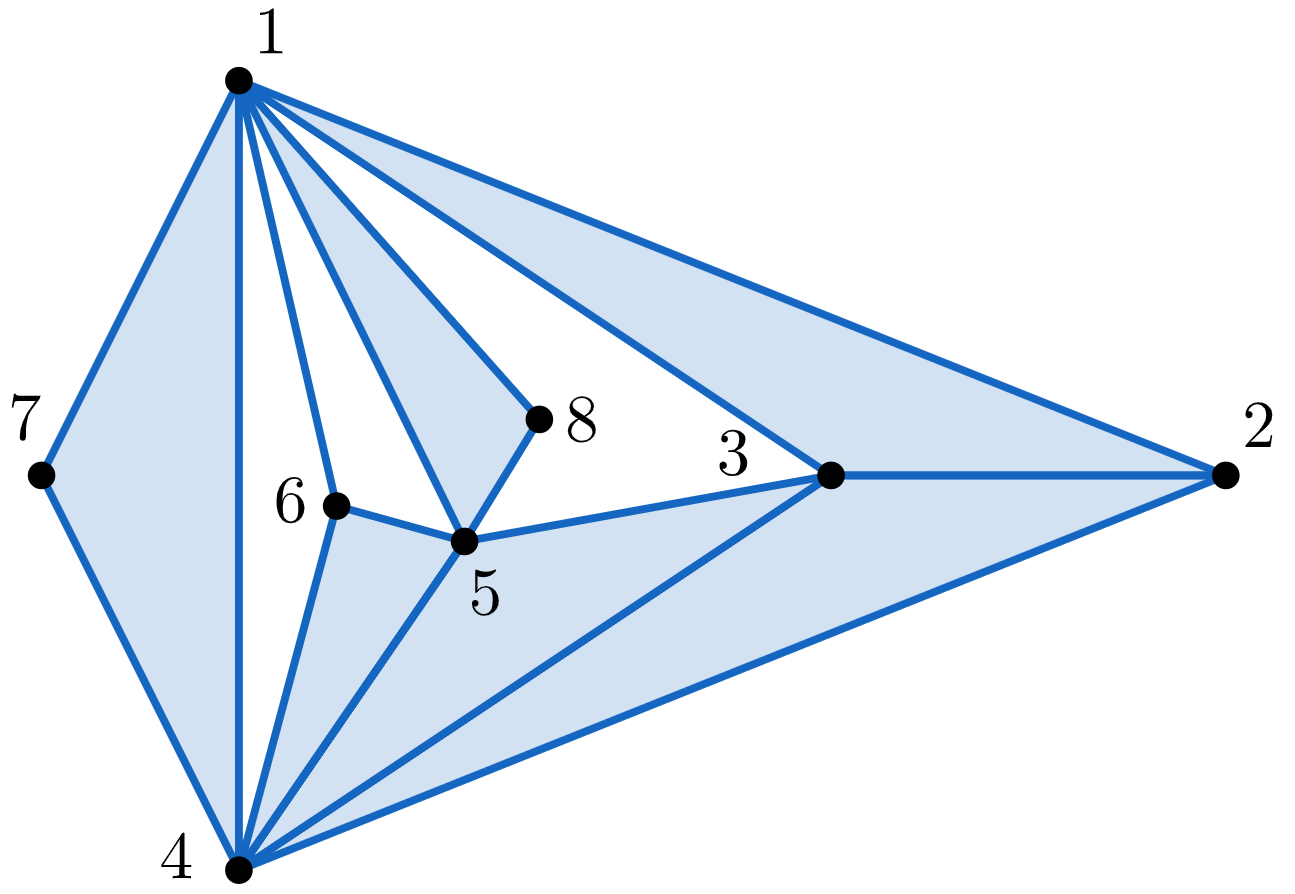}
    \caption{(left and middle) The simplicial complexes $L$  and $M$ from Remark \ref{rem:m1m2}. (right) The simplicial complex $N$ from Non-Example \ref{nex:WWeird1}.}

    \label{fig:remm1m2}
\end{figure}

\begin{remark}\label{rem:m1m2}
    Without the purity condition, weakly-simplicial and W-simplicial become incomparable vertex properties. In fact, in the non-pure simplicial complex (Figure \ref{fig:remm1m2})
    \[L = 15, 123, 124, 234, 345\]
    it is easy to see that vertex $5$ is skeleton-E-simplicial. 
    However, vertex $5$ is not W-simplicial, because $L$ has no facet contained in $\{1, 3, 4\}$. In fact, one can check that $L$ is skeleton-weakly-chordal, but  does not have \emph{any} W-simplicial vertex.
    Instead, in
    \[ M =123, 125, 134, 145, 24\]
    (also depicted in Figure \ref{fig:remm1m2}) the vertex labeled by $5$ is not  very-weakly-simplicial, because the adjacent facets $145$ and $125$ are present but $124$ is missing. However, vertex $5$ is W-simplicial, because $M$ has  a facet $24$ that is contained in $\{1,2,4\}$. 
\end{remark}

\begin{definition}[Minor]
The \emph{clutter-deletion} ${\Delta} \backslash v$ is the simplicial complex on  $V({\Delta}) - v$ with facets    $\{F : F \text{ a facet of } \Delta, v \notin F\}$. The \emph{contraction} ${\Delta}/v$ is the simplicial complex on  $V({\Delta}) - v$ whose facets are the minimal sets of  $\{F \backslash \{v\}: F \text{ facet of } \Delta\}$. In other words, $\Delta \backslash v$ removes all facets that contain $v$, while ${\Delta}/v$ removes $v$ from every facet that contains $v$ and then removes any facets that properly contain others to end up with a simplicial complex. Any simplicial complex ${\Delta'}$ obtained from ${\Delta}$ by a sequence of deletions and contractions is called a \emph{minor} of ${\Delta}$. Note that contraction does not preserve purity. (It does not preserve being underclosed either.) 
\end{definition}

\begin{definition}
A simplicial complex is \emph{W-chordal} if every minor  has a W-simplicial vertex.
\end{definition}

Since $\Delta$ is a minor of itself, every W-chordal complex has a W-simplicial vertex.

\begin{example}
Of the three simplicial complexes from Figure \ref{fig:Rem6}, the left one and the right one are W-chordal, the central one is not. So E-chordality and W-chordality are independent properties.
\end{example}

\begin{nonexample} \label{nex:WWeird1} The skeleton-E-chordal complex
\[ N = 123,  234, 345, 456, 147, 158, 16\]
is not W-chordal: In fact, the minor obtained by first deleting 6 and 7, and then by contracting 8, is the $2$-dimensional complex $123, 234, 345, 15$, which has no W-simplicial vertex. 
\end{nonexample}

\begin{example} \label{ex:WWeird} 
The  ``non-standard'' triangulation of a pinched annulus from Remark \ref{rem:nonstandard} 
is W-chordal. Note that the induced subcomplex on $\{1, 3, 4, 6\}$ is a $4$-cycle; hence, this example is not geochordal. It is however weakly-chordal.

\end{example}

\begin{definition} A $d$-dimensional simplicial complex is \emph{skeleton-W-chordal} if its $k$-skeleton is W-chordal for all $k \le d$.
\end{definition}

\begin{example} \label{ex:WoodNotCone} The pinched annulus $P_2(5)$ from Lemma \ref{lem:pinchedannulus} 
is skeleton-W-chordal. 
In contrast, the $3$-dimensional complex $v\ast P_2(5)$  is W-chordal, 
but its 2-skeleton 
\[ 123,\ 126,\ 136,\ 145,\ 146,\ 156,\ 234,\ 236,\ 246,\ 345,\ 346,\ 356,\ 456,\]
does not have W-simplicial vertices. So $v \ast P_2(5)$ is not skeleton-W-chordal.

\end{example}

\begin{example} \label{ex:WoodNotDeletion}
The simplicial complexes 
\[O = 12, 134, 135, 234, 235
\qquad \textrm{ and } \qquad P = 13, 15, 124, 146, 234, 235, 246
\] 
are both skeleton-W-chordal. In contrast, $\link(3,O)$ is the 4-cycle $14, 15, 24, 25$, which is not (W-)chordal. Also, $\operatorname{del}(3,P) = 15, 25, 124, 146, 246$ is not W-chordal because the contraction of vertex $6$ from $\operatorname{del}(3,P)$ is the same 4-cycle $14, 15, 24, 25$.  
\end{example}

\begin{proposition} \label{prop:WStable}
W-chordality is preserved under cones, 
 but not links or deletions. Also, it is not maintained under passing to the $k$-skeleton. 
\end{proposition}

\begin{proof}
Let $\Delta$ be a simplicial complex that is W-chordal.
\begin{compactitem}
\item Cones: This is proven in \cite[Proposition 4.10]{DGP26}. We repeat the proof for convenience. Let $v$ be a new vertex. Notice that $\Delta = (\Delta \ast v)/v $ and $ (\Delta \ast v) - v = \emptyset$. If $w$ is a simplicial vertex of $\Delta$, then $w$ is a simplicial vertex of $\Delta \ast v$ as $v$ is in every facet of $\Delta \ast v$. Therefore, if $\Delta$ has a simplicial vertex, then $\Delta \ast v$ also does. Moreover, the coning operation commutes with the clutter-deletion as well as with the contraction operation for every vertex $w \in \Delta$. It follows that every minor of $\Delta \ast v$ is either a minor of $\Delta$ or a cone over a minor of $\Delta$. We can conclude now that if $\Delta$ is W-chordal, then so is $\Delta\ast v$. 
\item Links and Deletions: See the simplicial complex $O$ and $P$ from Example \ref{ex:WoodNotDeletion}.
\item Skeleta: See Example \ref{ex:WoodNotCone}.
\qedhere
\end{compactitem}
\end{proof}

\begin{proposition}
Skeleton-W-chordality is maintained under taking skeleta, but not cones, links, or deletions.
\end{proposition}

\begin{proof}

    Let $\Delta$ be a simplicial complex that is skeleton-W-chordal. 
    \begin{compactitem}
        \item Cones: See Example \ref{ex:WoodNotCone}.
        \item Links: See the simplicial complex $O$ from Example \ref{ex:WoodNotDeletion}.
        \item Deletions: See the simplicial complex $P$ from Example \ref{ex:WoodNotDeletion}. \qedhere
    \end{compactitem}
\end{proof}

\begin{remark}\label{rem:WooClutterDel}
W-chordality is preserved under clutter-deletions
$\Delta\backslash v$ by definition. The latter is
the simplicial complex whose facets are the facets
of $\Delta$ disjoint from $v$, whereas the usual
deletion keeps all faces disjoint from $v$.
Clutter-deletion preserves purity, but not
skeleton-W-chordality. In fact, consider
\(
\Delta=16,26,36,125,135,456,1234.
\)
With this labeling, $\Delta$ is underclosed, and
so are all its skeleta. Hence $\Delta$ is
skeleton-W-chordal by \cite[Theorem 4.9]{DGP26}.
However,
\(
\Delta\backslash 4=16,26,36,125,135
\)
has an induced $4$-cycle $26,36,35,25$ in its
$1$-skeleton, and thus is not
skeleton-W-chordal.
An alternative notation for clutter-deletion
is introduced in Definition~\ref{def:abdel}.
\end{remark}

\newpage
%%%%%%%%%%%%%%%% SECTION 4 %%%%%%%%%%%%%%%%%%%%%%%
%%%%%%%%%%%%%%%% SECTION 4 %%%%%%%%%%%%%%%%%%%%%%%
\section{Chordality via simplicial faces}
%%%%%%%%%%%%%%%% SECTION 4 %%%%%%%%%%%%%%%%%%%%%%%
%%%%%%%%%%%%%%%% SECTION 4 %%%%%%%%%%%%%%%%%%%%%%%

The  vertices of a graph can also be viewed as its codimension-one faces. This observation led Bigdeli, Yazdan-Pour and Zaare-Nahandi to another definition of chordality \cite{BYZ17}, here called ``ridge-chordality''. 

\begin{definition}[Ridge]
A \emph{ridge} of a simplicial complex $\Delta$ is any face $R$ that is not a facet, and such  that \emph{all} facets of $\Delta$ strictly containing $R$ have dimension $\dim R +1$. 

\end{definition}

\begin{remark}  In pure simplicial complexes, ``ridges'' are the same as ``$(d-1)$-faces''. Note that our definition of ridge says ``all facets'', and not, as more common in the literature, ``some''. With our variant, a codimension-one subface $f$ of a facet $F$ need not be a ridge, since $f$ might belong to another facet $G$ with $\dim G \ge  \dim F +1 \ge \dim f + 2$. However, any face $f$ of $\Delta$ that is not a facet, is contained in some ridge. To see this, among all facets that  contain $f$, pick an $F$ of largest dimension. Then any codimension-one subface of $F$ that contains $f$ must be a ridge.
\end{remark}

\begin{definition}[$t$-cliques, $t$-simplicial]
Let $\Delta$ be a  $d$-dimensional simplicial complex. 
\begin{compactitem}[$\bullet$]

\item A \emph{$t$-clique} of $\Delta$ is any subset $X$ of the vertex set of $\Delta$, such that any $k+1\le t+1$  vertices of $X$ span a $k$-face in $\Delta$. We sometimes say `clique' instead of `$1$-clique'.

\item A face in $\Delta$ is \emph{$t$-simplicial}, if the vertices in its  star  form a $t$-clique. 
\end{compactitem}
\end{definition}

\begin{remark}
Since every $(t+1)$-clique is a $t$-clique, $(t+1)$-simplicial implies $t$-simplicial. In particular, $d$-simplicial implies $1$-simplicial. In subflag complexes, the two notions coincide. \\
Note also that 
any facet $F$ is $t$-simplicial for any $t$, since $\operatorname{star}(F, \Delta)=F$.
\end{remark}

\begin{example}\label{ex:DunceHat}
The Dunce Hat is an 8-vertex triangulation  with facets
\[ 1 2 4, 
1 2 7,
1 2 8,
1 3 4, 
1 3 5, 
1 3 6, 
1 5 6, 
1 7 8,
2 3 5, 
2 3 7, 
2 3 8, 
2 4 5, 
3 4 8,
3 6 7, 
4 5 6, 
4 6 8, 
6 7 8. \]
In this complex, no edge is $2$-simplicial. Some edges  (like $14$) are $1$-simplicial, and some  (like $12$) are not. No vertex is $1$-simplicial.
\end{example}

\begin{lemma} \label{lem:simplicialfaces}
Let $\Delta$ be a $d$-dimensional simplicial complex. 
    If a face $F$ is $t$-simplicial, all faces of dimension $<d$ containing $F$ are $t$-simplicial as well. 
The converse is false.
\end{lemma}

\begin{proof} Let $0 \le j \le \ell < d \in \mathbb{Z}$.  Let $t \in \{1, \ldots, d\}$.
Let $F \subsetneq G$ be faces of $\Delta$, with $\dim F=j$, $\dim G=\ell$. Since the vertices of $\operatorname{\Star}(G, \Delta)$ form a subset of the vertices of $\operatorname{\Star}(F, \Delta)$, if any $k+1 \le t+1$ vertices of $\operatorname{\Star}(F, \Delta)$ span a $k$-face in $\Delta$, then also any $k+1 \le t+1$ vertices of $\Star(G, \Delta)$ span a $k$-face in $\Delta$. As for the converse:  In the two-dimensional simplicial complex $123, 145$, vertex $1$ is neither $1$- nor $2$-simplicial, but all edges containing it are $1$ and $2$-simplicial. 
\end{proof}

The next technical Lemmas by Bigdeli--Yazdan-Pour--Zaare-Nahandi relate  the P-simpliciality of vertices to the $d$-simpliciality of the faces containing  them.

\begin{lemma}[{ \cite[Lemma 3.11]{BYZ17}}]
Let $\Delta$ be an E-chordal pure $d$-dimensional simplicial complex. Among all its ridges, the lexicographically-largest  one (which in particular contains the E-simplicial vertex $n$) is $d$-simplicial.
\end{lemma}

\begin{lemma}[{Bigdeli--Yazdan-Pour--Zaare-Nahandi \cite[Lemma 3.6]{BYZ17}}] Let $\Delta$ be a W-chordal pure  simplicial complex. Let $R=[x_1, \ldots, x_d]$ be a ridge such that
\begin{compactitem}
\item the vertex $x_1$ is W-simplicial in $\Delta_1 := \Delta$, and

\item for all $i \in \{1, \ldots, d-1\}$, the vertex $x_{i+1}$ is W-simplicial in $\Delta_{i+1} : = \Delta_{i}/x_i$.
\end{compactitem}
Then $R$ (which contains the W-simplicial vertex $x_1$ of $\Delta$ chosen initially) is $d$-simplicial.
\end{lemma}

In particular, \emph{pure} W- and E-chordal complexes always have $d$-simplicial ridges.
We integrate this with a Lemma that has no purity assumption:

\begin{lemma} \label{lem:FromRagsToRidges} Let $\Delta$ be a $d$-dimensional simplicial complex. Let $v$ be any vertex of $\Delta$. 
\begin{compactenum}[\rm (i)] 
\item If $v$ is clique-simplicial, then any ridge $R$ containing $v$ is 1-simplicial.\\
Moreover, for any two vertices in $\link(R, \Delta)$, the edge connecting them belongs to some facet of $\Delta$ that does not contain $R$.
\item If $v$ is skeleton-clique-simplicial, then any face containing $v$ is 1-simplicial.
\item If $v$ is very-weakly-simplicial, then any ridge $R$ containing $v$ is 1-simplicial. \\Moreover, if $r$ is the size of $R$, any two vertices  in $\link(R, \Delta)$ are in some face of dimension $r$ that does not contain $R$ or $v$.
\item If $v$ is W-simplicial, then  any ridge $R$ containing $v$ is 1-simplicial.\\
Moreover, any two vertices in $\link(R, \Delta)$ are in some facet (not necessarily $d$-dimensional) that does not contain  $R$ or $v$. 
\item If  $v$ is mid-simplicial, then any face $f$ containing $v$, and with the property that all facets containing $f$ are $d$-dimensional, is 1-simplicial. \\
Moreover, if $v$ is contained in at least two $d$-dimensional facets, then any two vertices in $\link(f,\Delta)$ are contained in some $d$-face that does not contain $v$ or $f$.
\item If $v$ is E-simplicial, then any $j$-face $f$ containing $v$, and such that all facets containing $f$ are $d$-dimensional, is $d$-simplicial. 
\end{compactenum}
\end{lemma}

\begin{proof} Note first that with our definition of ridge, any face strictly containing a ridge $R$  must be a facet. (For otherwise, a facet strictly containing such face would have dimension $\ge \dim R +2$, a contradiction.) So the facets containing a given ridge $R$ can always be written as $x_1 * R$, $\ldots$, $x_m*R$, for some integer $m$ and for some $x_1, \ldots, x_m$ vertices. Now: 
\begin{compactenum}[\rm(i)] 
\item Let $x_1 * R$, $\ldots$, $x_m*R$ be the facets containing $R$. Since they all contain the clique-simplicial vertex $v$,  $\Delta$ must contain all edges $x_ix_j$. Let $E$ be any facet of $\Delta$ containing $x_ix_j$. If $E$ contained $R$, it would also contain the two facets $x_i *R$ and $x_j *R$; so it would strictly contain both; so $\dim E \ge \dim R +2$; a contradiction with our definition of ridge. 
\item If $f$ is a facet, the claim is clear. Otherwise, let $x_1*f, \ldots, x_m * f$ be the faces of dimension $\dim f +1$ that contain $f$. Since these are $m$ faces containing the skeleton-clique-simplicial vertex $v$, the complex $\Delta$ must contain all edges $x_ix_j$.
\item Let $x_1 * R$, $\ldots$, $x_m*R$ be the facets containing $R$. Since they are pairwise-adjacent, and since they all contain the very-weakly-simplicial vertex $v$, for any $x_i$ and $x_j$ $\Delta$ contains the unique $r$-face $H_{i,j}$ (not necessarily a facet) with vertex set $\{x_i, x_j\} \cup R - \{v\}$.  In particular, $\Delta$ contains all the size-two subsets of $\{x_1, \ldots, x_m \} \cup R$. Note that $H_{i,j}$ is disjoint from $v$, while $R$ contains $v$; so $H_{i,j}$ cannot contain $R$.

\item Consider any pair $x_i*R$, $x_j*R$. The definition of W-simplicial vertex requires $\Delta$ to contain some facet $H \subseteq \{ x_i, x_j\} \cup R - \{v\}$. Since $H$ avoids $v$, it is different than $x_i*R$ and $x_j*R$. 
Since $H$ is a facet, it cannot be contained in another facet such as $x_i*R$ or $x_j*R$. Hence, $H$ must contain both $x_j$ and $x_i$.
Thus the edge $x_ix_j$ is in $H$, and so in $\Delta$. 

\item Let $F_1, \ldots, F_N$ be the $d$-faces containing $f$. Let $L$ be the set of vertices in $\link(f, \Delta)$. Let $\{x, y\}$ be any size-2 subset of $L \cup f$. If $x, y$ are in the same $F_i$, the edge $xy$ is in $\Delta$. If instead $x \in F_i$ and $y \in F_j$ for some  $i\ne j$, then $F_i$ and $F_j$ are $d$-faces containing the mid-simplicial vertex $v$. Hence, $\Delta$ contains a $d$-face $H_{x,y}$ such that $\{x,y\} \subseteq H_{x,y} \subseteq F_i \cup F_j - \{v\}$. So $xy$ is an edge of $\Delta$. Being disjoint from $v$, $H$ cannot contain $f$.  
\item If $f$ is a $d$-face, the claim is obvious. Otherwise, let $F_1, \ldots, F_N$ be the $d$-faces that properly contain $f$. Let $L=\{x_1, \ldots, x_m \}$ be the set of vertices in $\operatorname{link}(f, \Delta)$. 
Any subset $Y$ of $L \cup f$ of size  $d+1$ consists of $a$ points from $L$ and $b$ points from $f$, with $a+b = d+1$. To show that $Y$ is a face of $\Delta$, we distinguish two cases:
\begin{compactitem}
\item If $Y$ contains $f$, let $y_1 \in Y - f$. Among all the $d$-faces $F_i$ that contain $y_1$, choose one (say, $F_1$) that contains  a maximal number of elements from $Y$. 
Note that $F_1$ and $Y$ have the same size. 
If $F_1=Y$, then $Y$ is in $\Delta$ and we are done.
If $F_1 \ne Y$, there must be a vertex $y_2 \in Y$ with $y_2 \notin F_1$ and symmetrically a $z_1 \in F_1$ with $z_1  \notin Y$. Now, $y_2$ must belong to at least one of $F_2, \ldots, F_N$; up to relabeling, suppose $y_2 \in F_2$. Since $F_1$ and $F_2$ are $d$-faces containing the E-simplicial vertex $v$, all size-$(d+1)$ subsets of $F_1 \cup F_2$ are $d$-faces of $\Delta$. But one such subset is
\[ G_1 := F_1 - \{z_1\} \cup \{y_2\}.\]
A contradiction: $G_1$ is a face of $\Delta$ with one more element from $Y$ than $F_1$. 
\item If $Y$ does not contain $f$, pick a $z$ in $f$ but not in $Y$. We proceed by induction on $a$. If $a=0$, then $Y \subseteq f$, so $Y \in \Delta$. 
If $a \ge 1$, choose $i$ such that $x_i \in Y$. By assumption, $x_i$ belongs to at least one of $F_1, \ldots, F_n$; say, $F_1$. Set 
\[ 
G:= \left \{
\begin{array}{ll}
Y - \{x_i\} \cup \{z\} & \textrm{if $Y$ contains $v$,} \\
Y - \{x_i\} \cup \{v\} & \textrm{ otherwise.} 
\end{array} \right.
\]
Then $G$ has $d+1$ vertices, contains $v$, and has one fewer vertex from $L$ than $Y$. By the inductive assumption, $G$ is a $d$-face of $\Delta$.

Now $F_1$ and $G$ are two $d$-faces containing the E-simplicial vertex $v$. Therefore, every $(d+1)$-subset of $F_1 \cup G$ is a $d$-face of $\Delta$. But $x_i$ is in $F_1$ and $Y-\{x_i\}$ is contained in $G$, so $Y \subseteq F_1 \cup G$. In particular, $Y$ is a $d$-face of $\Delta$. 
\end{compactitem}
So either way, $Y \in \Delta$. Then also any subset $Y'$ of $L \cup f $ of size $\le d+1$ is in $\Delta$.
\qedhere
\end{compactenum}
\end{proof}

%%%%%%%%%%%%%%%% SUBSECTION 4.2 %%%%%%%%%%%%%%%%%%%%%%%
\subsection{Deleting above a face} \label{sec:DelAbove}
%%%%%%%%%%%%%%%% SUBSECTION 4.2 %%%%%%%%%%%%%%%%%%%%%%%

\begin{definition}[pure $k$-skeleton]
    The \emph{pure $k$-skeleton} of a simplicial complex $\Delta$, denoted by $\operatorname{pure-skel}_k(\Delta)$, is the subcomplex generated by the $k$-dimensional faces of $\Delta$. 
\end{definition}

\begin{definition}[Deletion above a face] \label{def:abdel} Let $0 \le j< d$ be integers. Let $\Delta$ be a $d$-dimensional complex.  Let $F$ be a $j$-dimensional face of $\Delta$.
\emph{Deleting above $F$} means passing from $\Delta$ to the simplicial complex $\operatorname{abdel}(F, \Delta)$ whose facets are the facets of $\Delta$ not containing that face. 

\end{definition}

Note that if $\Delta$ is pure, $\operatorname{abdel}(F, \Delta)$ is also pure, whereas $\del(F, \Delta)$ need not be. For the next definition, the empty set is by convention $(-1)$-dimensional:

\begin{definition}[$j$-face-chordal] Let $0 \le j <d$ in $\mathbb{N}$.
A $d$-dimensional simplicial complex $\Delta$ is 
\begin{compactitem}
\item \emph{$j$-face-chordal}, if it can be reduced to a simplicial complex of dimension $\le j$, by repeatedly 
deleting above a $d$-simplicial $j$-face;
\item \emph{weakly-$j$-face-chordal}, if it can be reduced to a simplicial complex of dimension $\le j$, by repeatedly  deleting above a 1-simplicial $j$-face. 
\end{compactitem}
Here $d$ denotes the dimension of the
original complex, and remains fixed
throughout the deletion sequence.
When $j=0$ or $j=d-1$, we prefer to say \emph{vertex-chordal} and \emph{ridge-chordal} instead of   ``$0$-face-chordal'' and ``$(d-1)$-face-chordal'', respectively. Similarly, we speak of \emph{weakly-vertex-chordal} and \emph{weakly-ridge-chordal}.
\end{definition}

\begin{remark} For us ridges need not be $(d-1)$-dimensional. However, we call a $d$-dimensional simplicial complex ``ridge-chordal'' if it can be reduced to a lower-dimensional simplicial complex by repeatedly deleting above a $d$-simplicial ridge \emph{of dimension $d-1$}. 
\end{remark}

\begin{example}[{cf.~\cite[Example 4.7]{BF20}}]\label{ex:DunceHat1}
The Dunce Hat of Example \ref{ex:DunceHat} is not ridge-chordal: No ridge (i.e. edge) is $2$-simplicial. But it is weakly-ridge-chordal: A sequence proving this is
\[14, 28, 78, 34, 17, 24, 45, 46, 56, 16, 13, 36, 37, 23.\]
In contrast, the barycentric subdivision of the Dunce Hat is a (flag) complex that is not weakly-ridge-chordal, because it lacks 1-simplicial ridges. 
\end{example}

Even though we have no vertex labeling to exploit, we can also create a `skeleton-version' of the ridge-chordality and the vertex-chordality properties above:

\begin{definition} 
A $d$-dimensional simplicial complex $\Delta$ is 
\begin{compactitem}[ $\bullet$]
\item \emph{skeleton-(weakly)-ridge-chordal}, if its $k$-skeleton is (weakly)-ridge-chordal for all $1 \le k \le d$;
\item \emph{skeleton-(weakly)-vertex-chordal}, if its $k$-skeleton is (weakly)-vertex-chordal, for all $1 \le k \le d$.
\end{compactitem}
\end{definition}

For graphs, all these notions boil down to chordality, since the ``deletion above a 1-simplicial vertex'' is just the deletion of a simplicial vertex in the sense of Dirac (cf. Theorem \ref{thm:charsimplicial}).

\begin{remark} The credit for the idea of chordality via deletions goes to Bigdeli, Yazdan-Pour and Zaare-Nahandi \cite{BYZ17}, although a similar notion of ``strongly-triangulable matroid'' had appeared in \cite{CLL09}. In the paper  \cite{BYZ17}, pure ridge-chordality is just called ``chordality'' and phrased  in terms of uniform clutters. The same notion is also called  ``chordality'' (of clutters) in Nikseresht \cite{Nik19},  ``$d$-chordality'' (of pure simplicial complexes) in Bigdeli--Faridi \cite{BF20}, and ``ridge-chordality''  (of pure simplicial complexes) in Benedetti--Bolognini \cite{BB21}.  
Skeleton-ridge-chordality coincides with chordality
in the sense of Bigdeli--Faridi
\cite[Definition 2.6 and Proposition 2.7]{BF20};
skeleton-weakly-ridge-chordality is new.
\end{remark}

\begin{lemma}\label{lem:one-step} Let
$0\le j \le d-2$ be integers. Let $\Delta$ be a $d$-dimensional simplicial complex. Let $F$ be a $d$-simplicial $j$-face of $\Delta$ that is not a facet. 
Then 
there is a finite sequence $\Delta_0,\Delta_1,\ldots,\Delta_r$
of subcomplexes of $\Delta$ such that:
\begin{compactenum}[\rm(1)]
\item $\Delta_0=\Delta$;
\item each $\Delta_{\ell+1}$ is obtained from $\Delta_\ell$ by deleting above some  $(j+1)$-face that contains $F$ and is $d$-simplicial in $\Delta_{\ell}$;
\item the $d$-faces of $\Delta_r$ and of $\abdel(F,\Delta)$ are the same.
\end{compactenum}
\end{lemma}

\begin{proof}
Let  $v_1 * F, \ldots, v_s *F$ be an ordered list of all the $(j+1)$-faces containing $F$, where the $v_i$ are vertices.
Ignoring the simpliciality condition, if from $\Delta$ we recursively delete above some $(j+1)$-face containing $F$, it is clear in the end we obtain $\abdel(F,\Delta)$. Now, it is possible that deleting above a single $(j+1)$-face containing $F$ makes other $(j+1)$-faces containing $F$  disappear. If this is the case, we update the order above by omitting the $(j+1)$-faces that disappear upon deleting above some previous $(j+1)$-face. Thus, there is an ordered list  $w_1 * F, \ldots, w_r *F$ of faces of $\Delta$, where the $w_i$'s are vertices, such  that:
\begin{compactitem}
\item deleting above each $w_i * F$ does not delete any of the $w_j * F$ for $j>i$;
\item deleting above all of these $w_i*F$, in their order, yields $\abdel(F,\Delta)$.
\end{compactitem}
Now inductively, let $\Delta_0 := \Delta$. For $k \ge 1$, let $\Delta_{k}$ be the simplicial complex obtained from $\Delta_{k-1}$
by deleting above $w_{k} *F$. We need to show that 
$w_{k+1}*F$ is $d$-simplicial in $\Delta_k$. For $k=0$, this is true by Lemma \ref{lem:simplicialfaces}. So, assume $k \ge 1$. Let $x_0, \ldots, x_d$ be any $d+1$ vertices in $\str(w_{k+1}*F, \Delta_k)$. Since $\Delta_k$ is a subcomplex of $\Delta$, $x_0, \ldots, x_d$ are also in 
$\str(w_{k+1}\ast F, \Delta)$. But by Lemma \ref{lem:simplicialfaces}, $w_{k+1} *F$ is $d$-simplicial in $\Delta$. Hence, the set
$
H=\{x_0,\ldots,x_d\}$
is a $d$-face of $\Delta$. We now make the crucial claim that no $w_i$ 
with $i\le k$ belongs to $H$. In fact, any face of $\Delta$ containing both $w_i$ and 
$w_{k+1} *F$ would obviously also contain $w_i *F$. Any such face is removed when deleting above $w_i *F$, and is therefore no longer  present in $\Delta_k$. Applying this to the face $w_i*w_{k+1}*F$, we conclude that no $w_i$ with $i\le k$ belongs to $\str(w_{k+1}*F, \Delta_k)$. So the set $\{w_1, \ldots, w_{k}\}$ is disjoint from the set of vertices of $\str(w_{k+1}*F, \Delta_k)$, and in particular from $\{x_0, \ldots, x_d\}$. 
So the claim is proven. But then $H$ contains none of the previously deleted faces $w_i * F$, with $i \le k$. Hence, the $d$-face $H$ survives all deletions above $w_i * F$, with $i \le k$. So $H$ is in $\Delta_k$. By the genericity of $H$, $w_{k+1} * F$ is $d$-simplicial in $\Delta_k$.
\end{proof}

\begin{theorem}\label{thm:j-to-jplusone}
Let
$0\le j \le d-2$ be integers. Let $\Delta$ be a $d$-dimensional simplicial complex.
\[ \Delta \; \textrm{$j$-face-chordal} \ \Longrightarrow  \ \Delta \; \textrm{$(j+1)$-face-chordal}. \]
Same for `weakly'. 
\end{theorem}

\begin{proof}
Let $F_1, \ldots, F_s$ be a sequence of $j$-faces proving $j$-face-chordality for $\Delta$. The $F_i$'s cannot be all facets, or else $\Delta$ would be $j$-dimensional, contradicting $j\le d-2$. If we omit from the list the $F_i$'s that are facets, and  replace each non-facet with some sequence of $(j+1)$-faces containing it, we get a  sequence of $(j+1)$-faces  that by Lemma \ref{lem:one-step} shows the $(j+1)$-face-chordality of $\Delta$.
\end{proof}

The next Lemma strengthens \cite[Corollary 3.11]{BYZ17} and \cite[Lemma 3.9]{Nik19}.

\begin{lemma} \label{lem:CdnFaceChordal}
The $d$-skeleton of $\Sigma_n$ is skeleton-vertex-chordal. \\The simplicial complex $AD^d(n)$ of Lemma \ref{lem:SubdividedTriangle} is skeleton-ridge-chordal, but for any $j<d-1$, $AD^d(d+2)$ is not $j$-face-chordal, because it has no $d$-simplicial $j$-faces. 
\end{lemma}

\begin{proof} 
Note first that $AD^2(4)$ has no 2-simplicial vertices. This generalizes to higher dimensions: $AD^d(d+2)$ is combinatorially equivalent to the stellar subdivision of a $d$-simplex. When $j<d-1$, the star  of any $j$-face of   $AD^d(d+2)$ contains all $d+2$ vertices of   $AD^d(d+2)$. Hence, since there  is a missing face,  for $j<d-1$ no $j$-face of $AD^d(d+2)$ is $d$-simplicial.

That said, the $d$-skeleton of $\Sigma_n$ is clearly vertex-chordal, by 
deleting above vertices in (reverse) lexicographic order. 
Thanks to the identity
\[ \operatorname{skel}_k(\operatorname{skel}_d (\Sigma_n)) = \operatorname{skel}_k(\Sigma_n)=\operatorname{skel}_k(AD^d(n)), \]
for $k<d$, we can conclude that
\begin{compactenum}[(1)]
\item the $d$-skeleton of $\Sigma_n$ is skeleton-vertex-chordal, and 
\item in order to prove $AD^d(n)$  skeleton-ridge-chordal, it suffices to prove it ridge-chordal. 
\end{compactenum}
So let us do it. Let  $F$ be the ``missing $d$-face'' of $AD^d(n)$. Let $r$ be any ridge of $F$. Write $F=r \cup \{z\}$. Clearly $r$ is in $AD^d(n)$. Moreover, all $d$-faces that do not contain $z$ are different from $F$ and thus present in $AD^d(n)$. Hence, the ridge $r$ is $d$-simplicial. Also,
\[\operatorname{abdel}(r, AD^d(n)) = \operatorname{abdel}(r, \operatorname{skel}_d(\Sigma_n)).\]  
It remains to argue that $\operatorname{abdel}(r, \operatorname{skel}_d(\Sigma_n))$ is ridge-chordal. But since $\Sigma_n$ is skeleton-vertex-, hence skeleton-ridge-chordal, and symmetric, we can assume that some sequence of $(d-1)$-faces proving the ridge-chordality of the $d$-skeleton of $\Sigma_n$, starts with $r$. 
\end{proof}

One of the main results of the paper \cite{BYZ17} is the following:

\begin{theorem}[{Bigdeli--Yazdan-Pour--Zaare-Nahandi \cite[Prop. 3.12 \& Cor. 3.7]{BYZ17}}] \label{thm:BYZ-EW}
Let $\Delta$ be a simplicial complex.
\begin{compactenum}[\rm (a)]
\item If $\Delta$ is pure E-chordal, then it is ridge-chordal. 
\item If $\Delta$ is pure W-chordal, then it is ridge-chordal.
\end{compactenum}
Both converses are false.
\end{theorem}

\begin{remark} In a pure W-chordal complex $\Delta$, every ridge $R$ containing a W-simplicial vertex is $1$-simplicial, by Lemma \ref{lem:FromRagsToRidges}, part (iv). However, $\operatorname{abdel}(R, \Delta)$ need not be W-chordal:
\[Q = 123,124,125,136,456\]
is W-chordal, and vertex 4 is W-simplicial in it, since $(124\cup  456) \backslash \{4\}=\{1,2,5,6\}$ contains 125. However, deleting above $R=14$ gives $123,125,136,456$, which is not W-chordal, because if we contract $4$ and then contract $1$ we get the $4$-cycle $23, 25, 36, 56$.
 
\end{remark}

\begin{remark}\label{rem:DualsAndVd0}
Theorem \ref{thm:BYZ-EW}, part (b), does not extend to non-pure complexes. The following simplicial complex is (skeleton)-W-chordal, but not ridge-chordal (cf.~also Proposition \ref{prop:DualsAndVd} below and Figure \ref{fig:notVDnotShellNotCM}):
\[
T = 124,125,134,135,234,235,45.
\]
\end{remark}

In contrast, part (a) of Theorem \ref{thm:BYZ-EW} can be considerably strengthened. First of all, it holds also in non-pure case. But more interestingly, in the pure case E-chordality turns out to be the same as vertex-chordality:

\begin{theorem}\label{thm:E-vertex-iff}
    Let $\Delta$ be a simplicial complex. 
    \begin{compactenum}[\rm I.]
\item $\Delta$ vertex-chordal $\Longrightarrow$ $\Delta$ E-chordal $\Longrightarrow$ $\Delta$ ridge-chordal. Both inclusions are strict.
\item   If $\Delta$ is pure, $\Delta$ vertex-chordal $\Longleftrightarrow$ $\Delta$ E-chordal.
    \end{compactenum}
\end{theorem}

\begin{proof} Let $d$ be the dimension of $\Delta$.
    \begin{compactenum}[I.]
    \item If $\Delta$ is vertex-chordal, let
$
\Delta=\Delta_0,\Delta_1,\ldots,\Delta_s$
be a vertex-chordal reduction sequence. Each $\Delta_{i+1}$ is obtained from
$\Delta_i$ by deleting above a $d$-simplicial vertex $v_i$, and $\dim \Delta_s\le 0$.
Let $B_i=V(\Delta_i)\setminus V(\Delta_{i+1})$
be the set of vertices that disappear when passing from $\Delta_i$ to
$\Delta_{i+1}$. Thus $v_i\in B_i$, but $B_i$ may contain other vertices as
well. Label all vertices in $B_i$ larger than all vertices in $B_\ell$ for
$\ell>i$, and larger than all vertices that remain in $\Delta_s$. Inside each
$B_i$, choose $v_i$ to be the largest vertex; the other vertices of $B_i$ may
be ordered arbitrarily. The vertices remaining in $\Delta_s$ receive the
smallest labels, in any order.

We claim that this labeling is an E-chordal labeling of $\Delta$. Let $F$ and
$G$ be facets of $\Delta$ with the same size and the same maximum
$m$.
If $m$ belongs to the final complex $\Delta_s$, then $F$ and $G$ contain no
vertex deleted in an earlier step, because all such vertices have labels larger
than $m$. Hence $F$ and $G$ are contained in $\Delta_s$. Since
$\dim \Delta_s\le 0$, this forces $F=G=m$, and the E-chordality condition
is trivial.
Otherwise, $m\in B_i$ for some $i<s$. Since all vertices disappearing before
step $i$ have labels larger than $m$, the facets $F$ and $G$ contain no such
vertices. Hence $F$ and $G$ are still facets of $\Delta_i$. If $m\ne v_i$,
then every facet of $\Delta_i$ containing $m$ must also contain $v_i$;
otherwise that facet would survive in $\abdel(v_i,\Delta_i)$ and would still
contain $m$. But $v_i$ was chosen to be the largest vertex of $B_i$, so this
would contradict the assumption that $m$ is the maximum of $F$ and $G$.
Therefore $m=v_i$.
Thus $F$ and $G$ both contain $v_i$, and
$F\cup G\subseteq V\bigl(\str(v_i,\Delta_i)\bigr)$.
Since $v_i$ is $d$-simplicial in $\Delta_i$, the vertices of
$\str(v_i,\Delta_i)$ form a $d$-clique. Therefore every subset $ H\subseteq F\cup G$
with $|H|=|F|=|G|$ is a face of $\Delta_i$, and hence a face of $\Delta$. 
Hence $\Delta$ is E-chordal. The implication is strict: Example \ref{ex:TwoTriangles} is a non-pure E-chordal complex that is not vertex-chordal.

Now suppose $\Delta$ is E-chordal. We first observe that the pure $d$-skeleton of $\Delta$ is E-chordal. Indeed, let $F$ and $G$ be $d$-facets of $\operatorname{pure-skel}_d(\Delta)$ with the same maximum. They are also facets of $\Delta$ with the same size and the same maximum. Since $\Delta$ is E-chordal, every $(d+1)$-subset $H\subseteq F\cup G$ is a $d$-face of $\Delta$, hence also of $\operatorname{pure-skel}_d(\Delta)$. Thus, $\operatorname{pure-skel}_d(\Delta)$ is E-Chordal. For pure complexes, E-chordality implies ridge-chordality by Theorem \ref{thm:BYZ-EW}. Therefore, $\operatorname{pure-skel}_d(\Delta)$ can be reduced by deleting above simplicial ridges. Performing the same ridge deletions in $\Delta$ removes all $d$-dimensional facets. 
Hence $\Delta$ is ridge-chordal. As for the strictness of this second implication: The complex $AD^d(d+2)$ from Lemma \ref{lem:SubdividedTriangle} is the cone over the boundary of a $d$-simplex. It is not E-chordal, but being a cone, it is ridge-chordal by Corollary \ref{cor:ConesRidgeChordal}. 
\item The `$\Rightarrow$' implication has already been discussed in part (I). As for `$\Leftarrow$': Assume that $\Delta$ is pure E-chordal. By Proposition \ref{prop:SimplicialVertices}, $\Delta$ has an E-simplicial vertex $v$. By Lemma \ref{lem:FromRagsToRidges}, part (vi), such vertex is $d$-simplicial. So if $F_1, \ldots, F_m$ are the $d$-faces of $\Delta$ containing $v$, any $k \le d+1$ vertices in $F_1 \cup \ldots \cup F_m$ span a $(k-1)$-face of $\Delta$.  We claim that if we delete above $v$, $\operatorname{abdel}(v, \Delta)$ is E-chordal with the induced labeling, whence the conclusion follows by recursion. Let $F, G$ be $d$-faces of $\operatorname{abdel}(v, \Delta)$, with same maximum. Since $\Delta$ is E-chordal, it contains all $d$-faces $H$ contained in $F \cup G$. Since $F$ and $G$ are disjoint from $v$, so is their union, and therefore $H$. Since $H$ is a facet of $\Delta$ not containing $v$, it is also in $\operatorname{abdel}(v, \Delta)$. Hence, $\Delta$ is vertex-chordal.  \qedhere \end{compactenum}
\end{proof}

\begin{lemma}\label{lem:rearranging1}
Let $R, S_1,\ldots,S_\ell$ be ridges in a $d$-dimensional simplicial complex $\Delta$, all containing a vertex $v$. Let $\Delta'$ be obtained from $\Delta$ by successively deleting above $S_1,\ldots,S_\ell$.
\begin{compactenum}[\rm (a)]
\item If $v$ is very-weakly-simplicial in $\Delta$, and $R$ is a $(d-1)$-dimensional ridge of $\Delta'$, then $R$ is a 1-simplicial ridge of $\Delta'$.
\item The same holds if $v$ is W-simplicial in $\Delta$.
\end{compactenum}
\end{lemma}

\begin{proof}
Let $x_1*R,\ldots,x_m*R$ be the facets of $\Delta'$ containing $R$. Since $R$ is a $(d-1)$-dimensional ridge, these are also facets of $\Delta$. If $m=1$, the claim is clear. Otherwise, for every $i,j\in\{1,\ldots,m\}$, the facets $x_i*R$ and $x_j*R$ are adjacent in $\Delta$ and contain $v$.
\begin{compactenum}[\rm (a)]
\item If $v$ is very-weakly-simplicial, then $\Delta$ contains the $d$-face
$H_{i,j}=\{x_i,x_j\}\cup R-\{v\}$.
Since $H_{i,j}$ avoids $v$, it survives all deletions, and hence $x_ix_j\in\Delta'$.
\item If $v$ is W-simplicial, then $\Delta$ contains some facet
$h_{i,j}\subseteq\{x_i,x_j\}\cup R-\{v\}$.
Since $h_{i,j}$ is a facet different from $x_i*R$ and $x_j*R$, it contains both $x_i$ and $x_j$. Moreover, it avoids $v$, so it survives all deletions. Hence $x_ix_j\in\Delta'$.
\end{compactenum}
Thus $R$ is 1-simplicial in $\Delta'$.
\end{proof}

\begin{nonexample}
Consider the $3$-dimensional simplicial complex
\[ U= 1234, 1235, 2345, 2456, 3678\]
Vertex $1$ is weakly-simplicial in $U$. Both ridges $R=123$ and $S=245$ 
are 1-simplicial in $U$. Deleting above $S$ yields the complex $U' = 1234, 1235, 3678$, in which the ridge $R$ is no longer 1-simplicial. This does not contradict Lemma \ref{lem:rearranging1} because $R$ contains  $1$, but $S$ does not.
\end{nonexample}

\begin{theorem} \label{thm:weaklyVertexChordal} Let $\Delta$ be a  simplicial complex. 
\begin{compactenum}[\rm (i)]
\item If $\Delta$ is pure very-weakly-chordal, it is weakly-ridge-chordal.
\item If $\Delta$ is pure mid-chordal, then $\Delta$ is weakly-vertex-chordal.
\item Mid-chordality and ridge-chordality are independent properties, even for pure complexes.
(In particular, pure mid-chordal complexes are not vertex-chordal in general). 
\end{compactenum}
The converses of \emph{(i)} and \emph{(ii)} are false.
\end{theorem}

\begin{proof} 
\begin{compactenum}[\rm (i)] 
\item We proceed by induction on the number of vertices. Consider the vertex $v$ labeled by $n$, in a labeling that proves $\Delta$ very-weakly-chordal. By Proposition \ref{prop:SimplicialVertices}, $v$ is very-weakly-simplicial. By Lemma \ref{lem:FromRagsToRidges}, part (iii), any ridge containing $v$ $1$-simplicial. One can obtain  $\abdel(v, \Delta)$ from $\Delta$  by repeatedly deleting above a ridge containing $v$. By Lemma \ref{lem:rearranging1}, each of these ridges is 1-simplicial not just in the original $\Delta$, but also in the subcomplex they are being deleted from. Since $\Delta$ is pure and $v$ is very-weakly-simplicial, by Proposition \ref{prop:weakStability} $\del(v, \Delta)$ is very-weakly-chordal. We claim that $\abdel(v, \Delta)$ is very-weakly-chordal. Indeed, consider two adjacent facets $F, G$ of $\abdel(v, \Delta)$ with the same size, same maximum. Then $F, G$ are also facets of $\del(v, \Delta)$. Since $\del(v, \Delta)$ is very-weakly-chordal, it contains the $d$-face $H:=F \cup G - \{\max F\}$. Since $F, G$ are disjoint from $v$, so is $H$. Since $H$ is a facet of $\del(v, \Delta) \subseteq \Delta$ not containing $v$, $H$ belongs  also to $\abdel(v, \Delta)$. So the claim is proved. By inductive assumption,  $\abdel(v, \Delta)$ is weakly-ridge-chordal. But then $\Delta$ is as well.

\item By Proposition \ref{prop:SimplicialVertices}, $\Delta$ has a mid-simplicial vertex $w$. By Lemma \ref{lem:FromRagsToRidges}, part (v), applied to $f=w$, the vertex $w$ is $1$-simplicial. We claim $\operatorname{abdel}(w, \Delta)$ is mid-chordal with the induced labeling. Let $F, G$ be $d$-faces of $\operatorname{abdel}(w, \Delta) \subseteq \Delta$, with same  maximum. Since $\Delta$ is mid-chordal, it contains each $2$-element subset $e$ of $F \cup G$. Moreover, any such $e$ contained in some face $H_e \subseteq F \cup G - \{ \max F\}$, of dimension $d$ like $F$ and $G$. Since $F$ and $G$ are disjoint from $w$, so are $F \cup G$ and $H_e$. Since it is a facet, $H_e$ survives the deletion above $w$, and so does $e$. Thus the claim is proven. By induction, $\Delta$ is weakly-vertex-chordal. 
\item The cone over a $4$-cycle is easily seen to be ridge-chordal (compare also Corollary \ref{cor:ConesRidgeChordal} below), but it is not mid-chordal. In fact, it is not even very-weakly-chordal. \\On the other hand, the simplicial complex 
\[ I = 123, 124, 134, 135, 145, 234, 235, 245\]
from Example \ref{ex:WrongLinks} and Figure \ref{fig:ex19} is skeleton-mid-chordal, but not ridge-chordal: it has only two $2$-simplicial edges, $12$ and $34$. Deleting above any of them yields a complex without $2$-simplicial edges. 

\end{compactenum}
\noindent As for the converses: Examples of simplicial complexes that are weakly-ridge-chordal, but not very-weakly-chordal, can be found via Proposition \ref{prop:PolytopesWRC}. Finally, the simplicial complex
$D^d(d+k)$ of Lemma \ref{lem:WeaklyNonMidChordal} is not mid-chordal, but it is easy to see that it is weakly-vertex-chordal, deleting above vertices in countdown order.\qedhere
\end{proof}

%%%%%%%%%%%%%%%% SUBSECTION 4.3 %%%%%%%%%%%%%%%%%%%%%%%
\subsection{Ridge-chordality vs. collapsibility}
%%%%%%%%%%%%%%%% SUBSECTION 4.3 %%%%%%%%%%%%%%%%%%%%%%%

Bigdeli--Faridi \cite{BF20} noticed a connection between ridge-chordality and  Whitehead's notion of collapsibility. Caveat: in \cite{BF20} the definition of ``free faces'' is altered to allow facets as well; also, Bigdeli--Faridi work with ``$d$-closures'', a concept introduced in \cite{CF13}. To avoid confusion, we briefly translate the Bigdeli--Faridi results into our language, recalling Whitehead's original definitions for convenience:

\begin{definition}
    Let $\Delta$ be a simplicial complex. A \emph{free face} in $\Delta$ is any face strictly contained in exactly one other face of $\Delta$. An \emph{elementary collapse} is the deletion from $\Delta$ of a free face. 
    Given a subcomplex $S$ of $\Delta$ we say that \emph{$\Delta$ collapses onto $S$}  if there exists a sequence of elementary collapses that reduces $\Delta$ to $S$. The sequence may be empty, so any complex $\Delta$ collapses onto itself. A simplicial complex is called \emph{collapsible} if it  collapses to $\emptyset$, the void complex. 
\end{definition}

Every free face of dimension $d-1$ in our context is a $d$-simplicial ridge. There are however  three key differences between elementary collapses and deletions above ridges:
\begin{compactenum}[(1)]
\item All deletions above faces maintain purity. In contrast, elementary collapses do not. 
\item Each elementary collapse removes exactly two faces, of dimension $d$ and $d-1$, thereby maintaining Euler characteristic and homotopy. In contrast, deleting above a (free) face may result in a removal of lower-dimensional faces, which might affect the Euler characteristic. For example, look back at the deletion above the ridge $78$ in Example \ref{ex:RidgeChordalNotSkel}.
\item Some $d$-simplicial ridges are not free faces. For example, in the boundary of a simplex, all of the ridges are $d$-simplicial, yet none of them are free.
\end{compactenum}

The idea for the next proposition comes from Bigdeli--Faridi \cite[Theorem 3.4]{BF20}.

\begin{proposition}[{cf.~Bigdeli--Faridi \cite[Theorem 3.4]{BF20}}]
\label{prop:BF1coll}
Let $\Delta$ be a simplicial complex.
\begin{compactitem}
    \item If $\Delta$ collapses to some $(\dim \Delta-1)$-dimensional subcomplex, then it is ridge-chordal.
    \item If $\Delta$ has no free face, the barycentric subdivision of $\Delta$ is not weakly-ridge-chordal.
\end{compactitem}
\end{proposition}

\begin{proof} Let $d = \dim \Delta$. Let $N$ be the number of $d$-faces of $\Delta$. Since the sequences of elementary collapses can always be arranged so that higher-dimensional faces are collapsed first, without loss we can assume that the elementary collapses performed on $\Delta$ are exactly $N$, each of the form $(R_i, F_i)$, with $\dim R_i=d-1$ and $\dim F_i=d$. Now let $\Delta'$ (resp.  $\Delta''$) be simplicial complexes obtained from $\Delta$ by deleting $R_1$ with an elementary collapse (resp. by deleting \emph{above} $R_1$.) Then $\Delta'$ and $\Delta''$ have the same $d$-faces. The facets and the  $(d-1)$-faces may be different,  but any $(d-1)$-face in $\Delta'$ not in $\Delta''$ is a facet of $\Delta'$, so in particular not a ridge. In conclusion, the free ridges of $\Delta''$ and of $\Delta'$ are the same. In particular, $R_2$ is also a free ridge of $\Delta''$.  Repeating this argument,  it is easy to see that the sequence $R_1, \ldots, R_N$ proves $\Delta$ ridge-chordal. 

As for the second claim: Suppose $\Delta$ has no free face. 
Let $R'$ be a ridge of $\operatorname{sd}\Delta$. Then $R'$ is a chain of
nonempty faces of $\Delta$ obtained from a maximal chain by omitting exactly
one rank. If the missing rank is $0$, then the smallest face in the chain is
an edge $\{a,b\}$, and the two vertices $\{a\}$ and $\{b\}$ both lie in
$\link(R',\operatorname{sd}\Delta)$ but are not adjacent.
If the missing rank is internal, then there are two distinct intermediate
faces $H,H'$ that complete the chain. They both lie in
$\link(R',\operatorname{sd}\Delta)$, but they are incomparable, hence not
adjacent in $\operatorname{sd}\Delta$.
Finally, if the missing rank is the top rank, then the largest face in the
chain is a face $f$ of $\Delta$ contained in some facets of dimension $\dim f +1$. Since $\Delta$ has no
free face, $f$ is contained in at least two distinct facets $F,G$
of $\Delta$. The corresponding vertices $F$ and $G$ both lie in
$\link(R',\operatorname{sd}\Delta)$, but they are incomparable, hence not
adjacent. So either way, every ridge of $\operatorname{sd}\Delta$ has two nonadjacent vertices
in its link. Hence $\operatorname{sd}\Delta$ has no $1$-simplicial ridge.
\end{proof}

\begin{corollary}
\label{cor:ConesRidgeChordal} Every cone is ridge-chordal.
\end{corollary}

\begin{proof} Every cone is collapsible, hence ridge-chordal by Proposition  \ref{prop:BF1coll}.
\end{proof}

\begin{example} \label{ex:CollEvasive} In each dimension $d \ge 2$, the paper \cite{ABL17} describes a complex $ABL^d$ that has  exactly one free face, but is collapsible, and thus ridge-chordal by Proposition \ref{prop:BF1coll}. For example,
\[ ABL^2= 125, 134, 136, 137, 145, 167, 234, 236, 256, 237, 247, 456, 467.\]
Since no vertex of $ABL^2$ is $2$-simplicial, $ABL^2$ is not vertex-chordal.
Thus the conclusion of Proposition \ref{prop:BF1coll} cannot be improved to ``$(d-2)$-face-chordal''.
\end{example}

We conclude this section with a couple of new results. The first provides many examples of weak-ridge-chordality. Recall that a vertex of a polytope $P$ is \emph{simple} if it lies in exactly $\dim P$ facets. For example, all vertices of the cube are simple.

\begin{proposition} \label{prop:PolytopesWRC} Let $d \ge 2$.
Let $P$ be any simplicial $(d+1)$-dimensional polytope different from a simplex. If $P$ has a simple vertex, then $\partial P$ is  weakly-ridge-chordal.
\end{proposition}

\begin{proof}
Let $v$ be a simple vertex of $P$. Let $B = \del(v,\partial P) = \operatorname{abdel}(v,\partial P)$. It follows from Bruggesser--Mani's theorem that $B$ is a shellable, hence collapsible, $d$-dimensional ball, cf.~\cite[Lemma 8.10 \& Corollary 8.13]{Zie95}. In particular, $B$ is ridge-chordal by Proposition \ref{prop:BF1coll}.
Since $v$ is simple, the star $S$ of $v$ in $\partial P$ is combinatorially equivalent to the star of a vertex in $\partial\Sigma_{d+2}$. We can obtain $B$ by repeatedly deleting above ridges containing $v$, each time choosing a ridge contained in a facet not yet removed.
Each such ridge $R$ is $1$-simplicial in the current complex. Indeed, its link consists of at most two vertices $x,y$. If both are present, they are connected by an edge, since $d\ge 2$ and $xy$ belongs to a facet of $\partial P$ not containing $v$, which survives all deletions. 
Indeed, $\link(xy,\partial P)$ is a
$(d-2)$-sphere, and therefore not a cone
with apex $v$.
Thus, $\partial P$ is weakly-ridge-chordal.
\end{proof}

\begin{nonexample} The boundary of the octahedron is not weakly-ridge-chordal. Note that every vertex of the octahedron is in four facets, not three.
\end{nonexample}

\begin{definition}[Pseudomanifold]
By \emph{pseudomanifold} we mean a pure simplicial complex in which any ridge is in $\le$ 2 facets. The \emph{boundary} is the subcomplex formed by those ridges that are in exactly one facet. A pseudomanifold is \emph{closed} if its boundary is empty, \emph{non-closed} otherwise.
\end{definition}

 \begin{definition}[Strongly-connected] \label{def:SConn}
The \emph{dual graph} of a pure simplicial complex with $N$ facets is the graph with vertex set $[N]$, where $i$ and $j$ are connected by an edge if and only if the corresponding facets are adjacent. 
A pure simplicial complex is \emph{strongly-connected} if its dual graph is connected. 
\end{definition}

\begin{proposition} \label{prop:PseudoManiRidge}
Every strongly-connected non-closed pseudomanifold is ridge-chordal.\\
In contrast, every strongly-connected closed pseudomanifold is not ridge-chordal, except for the boundary of the simplex.
\end{proposition}

\begin{proof} 
For the first claim: Let $M_1$ be a  strongly-connected non-closed pseudomanifold. Let $d_1=\dim M_1$. Let $T$ be a spanning tree of the dual graph of $M_1$. Let $f$ be any ridge in the (non-empty!) boundary of $M_1$. This $f$ belongs to only one facet $F$, so the deletion of $f$ (and $F$) is an elementary collapse that can be used to start a ridge-deletion sequence. We may think of $T$ as rooted at $F$. We then proceed with elementary collapses ``alongside $T$'', by deleting exactly those (internal) ridges that are crossed by $T$, as soon as they become free; this is the same technique used for spheres in \cite[Section 2.1]{BZ11}. Since $T$ is spanning, eventually all $d_1$-faces of $M_1$ are collapsed away, and $M_1$ is  collapsed to a complex of dimension $d_1 -1$. This shows ridge-chordality via Proposition \ref{prop:BF1coll}. \\
As for the second claim: Let $M_2$ be a $d_2$-dimensional strongly-connected closed pseudomanifold. Suppose $M_2$ is ridge-chordal. Let $R$ be the first ridge deleted in a sequence that proves ridge-chordality. Since $M_2$ has no boundary, $\link(R, M_2)$ consists of two points $x, y$. By the definition of ridge-chordality, $M_2$ contains the $d_2$-skeleton of the $(d_2+1)$-simplex $\Sigma = x*y * R$, which is simply the boundary of $\Sigma$. Thus, $M_2$ must coincide with the boundary of $\Sigma$.
\end{proof}

\begin{example} 
The suspension of a $3$-cycle (Figure \ref{fig:weklynotgeod}) is a $2$-sphere $V$ that is weakly-ridge-chordal, but not ridge-chordal. Compare Remark \ref{rem:weklynotgeod} below.
\end{example}

\begin{figure}[h]
    \centering
    \includegraphics[height=0.23\linewidth]{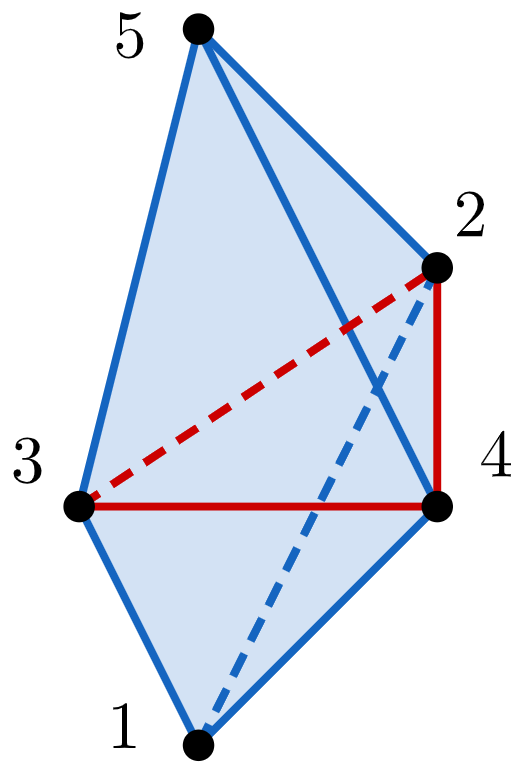}
    \caption{The simplicial complex $V=\operatorname{Susp}(C_3)$ from Remark \ref{rem:weklynotgeod}.}
    \label{fig:weklynotgeod}
\end{figure}

%%%%%%%%%%%%%%%% SUBSECTION 4.4 %%%%%%%%%%%%%%%%%%%%%%%
\subsection{Relation with geochordality and stability}
%%%%%%%%%%%%%%%% SUBSECTION 4.4 %%%%%%%%%%%%%%%%%%%%%%%

The main result of  \cite{BYZ17} is that if $\Delta$ is pure ridge-chordal, then $\Delta$ satisfies some homological condition (namely ``$\operatorname{pure-skel}_{n-d-2}(\Delta^\vee)$ 
is Cohen--Macaulay'') that implies geometric-$d$-chordality. We postpone details
to Section \ref{sec:Alexander}. Here we present a direct, elementary proof that ridge-chordality implies geometric-$d$-chordality; the novelty in this proof is that it is valid without any purity assumption.

\begin{proposition}\label{prop:FCtoGeo} Let $\Delta$ be a $d$-dimensional simplicial complex.
 If $\Delta$ is ridge-chordal, it is geometrically-$d$-chordal. In particular, all skeleton-ridge-chordal complexes are geochordal.\\
 Both implications are strict.
\end{proposition}

\begin{proof} We prove the first claim, which easily implies the second, by induction on the number $N$ of $d$-faces of $\Delta$. Any $d$-dimensional triangulated sphere has at least $d+2$ facets, a lower bound attained by the boundary of the $(d+1)$-simplex. Thus if $N< d+2$, $\Delta$  is vacuously geometrically-$d$-chordal. If $N \ge d+2$, let $R_1, \ldots, R_m$ be a sequence of $(d-1)$-dimensional ridges proving $\Delta$ ridge-chordal. Let $M$ be any subcomplex of $\Delta$ homeomorphic to a sphere, with the property that every $d$-face with all vertices in $M$, is in $M$. 
\begin{compactitem}[$\bullet$]
\item If $R_1 \notin M$, all $d$-faces of $M$ survive the deletion above $R_1$. So $M \subseteq \operatorname{abdel}(R_1, \Delta)$, which is ridge-chordal with fewer $d$-faces than $\Delta$. By induction, $M$ is the boundary of a simplex.
\item If $R_1 \in M$,  $R_1$ is contained in exactly two (adjacent) $d$-faces $G, H$ of $M$. 
Then $G \cup H$ is a size-$(d+2)$ subset of the star of $R_1$ in $\Delta$. Let $S$ be the $d$-skeleton of the $(d+1)$-simplex with vertex set $G \cup H$. Since $R_1$ is $d$-simplicial in $\Delta$, $S \subseteq \Delta$. By the assumption on $M$, each $d$-face of $S$ is in $M$, so $S \subseteq M$. But the only way a triangulated sphere can be a subcomplex of another of the same dimension, is if they coincide. Thus $S =M$. 
 \end{compactitem}
The strictness of the first implication is shown by the Dunce Hat, which is geometrically-$2$-chordal but not ridge-chordal: See Examples \ref{ex:DunceHat} and \ref{ex:DunceHat1}. For the second implication, take the $6$-vertex triangulation of the projective plane $\mathbb{RP}^2$, with facets 
\[ 124, 126, 135, 136, 145, 234, 235, 256, 346, 456.\]
Its $1$-skeleton is $K_6$ and it contains no $2$-sphere, so this triangulation is geochordal. But being a closed manifold,  it is not skeleton-ridge-chordal by Proposition \ref{prop:PseudoManiRidge}.
\end{proof}

\begin{example} \label{ex:RidgeChordalNotSkel} The Woodroofe complex $B^2 =123, 345, 567, 178$  of Remark \ref{rem:SkeletonChordalWeaklyChordal} and Figure \ref{fig:Rem6} (center), is ridge-chordal. (Deleting above $78$ yields the collapsible complex $123, 345, 567$; cf.~Proposition \ref{prop:BF1coll}.) Since the $1$-skeleton of the Woodroofe complex is not a chordal graph, ridge-chordality does not imply geochordality, or any of the properties from Definition \ref{def:skelP}.  
\end{example}

\begin{remark}\label{rem:weklynotgeod} Weakly-vertex-chordal does not imply geometrically-$d$-chordal. The suspension of a $3$-cycle, i.e. the $2$-dimensional simplicial complex (Figure \ref{fig:weklynotgeod})
\[ V = 123, 124, 134, 235, 245, 345,\]
 is weakly-vertex-chordal with the sequence $5, 4, 3$. 
Since the $1$-skeleton is $K_5$ minus the edge $15$, which is a chordal graph, this example is even skeleton-weakly-vertex-chordal. 
Weakly-vertex-chordality and ridge-chordality are in fact incomparable properties. The cone over a $4$-cycle is ridge-chordal by Corollary \ref{cor:ConesRidgeChordal}, but it does not have any $1$-simplicial vertex.
\end{remark}

\begin{proposition}
Both ridge-chordality and weak-ridge-chordality are preserved under cones, but not under links, deletions, or skeleta.
\end{proposition}

\begin{proof} By Corollary \ref{cor:ConesRidgeChordal}, cones are ridge-chordal. In particular, so is $v \ast C_4$. Yet, the link and deletion of the cone vertex are both $C_4$. Moreover, the $1$-skeleton of $v\ast C_4$ contains the base $C_4$ as an induced cycle, and hence is not a chordal graph.
\end{proof}

\newpage
\begin{proposition} \label{prop:SRCStability}
Skeleton-ridge-chordality and skeleton-weak-ridge-chordality are both  preserved under cones and deletions,  but not links.  
\end{proposition}

\begin{proof} We give the proof only  for skeleton-ridge-chordality; adapting it to skeleton-weak-ridge-chordality is easy. Let $\Delta$ be a skeleton-ridge-chordal $d$-dimensional simplicial complex. Let $S^k$ denote the $k$-skeleton of $\Delta$. By assumption, $S^k$ is ridge-chordal.
\begin{compactitem}
\item Cones: Let $T^k$ be the $k$-skeleton of $v \ast \Delta$. Note that the $k$-faces of $T^k$ are of two types: those that do not contain $v$ (i.e. the facets of $S^k$), plus those that contain $v$ (i.e. the cones over the facets of $S^{k-1}$). Let $R_1, \ldots, R_m$ be a sequence of ridges of $S^{k}$ that proves the ridge-chordality of $S^{k}$. We wish to apply the same ridge sequence  to $T^k$. To this end, we need to check that also within $T^k$, each $R_i$ is $k$-simplicial at the moment we delete above it; that is, the set $S$ of vertices in its star is a $k$-clique. Let us verify this. 

Let $C$ and $C'$ be the complexes obtained from $S^k$ and $T^k$, respectively, by deleting above $R_1,\ldots,R_{i-1}$. Let $S$ be the vertex set of the star of $R_i$ in $C$. The star of $R_i$ in $C'$ has vertex set $S\cup\{v\}$. Take any subset $A\subseteq S\cup\{v\}$ with $|A|\le k+1$.
If $v\notin A$, then $A\in C\subseteq C'$, since $S$ is a $k$-clique in $C$. Otherwise, $A=\{v\}\cup B$, where $B\subseteq S$ and $|B|\le k$. Choose a facet $F$ of $C$ containing $B$, and a face $G\subseteq F$ of size at most $k$ containing $B$. Then $v*G$ is a face of $T^k$ containing $A$. Moreover, it survives all preceding deletions, since otherwise some $R_j$, $j<i$, would be contained in $G\subseteq F$, contradicting $F\in C$. Hence $A\in C'$, and the check is complete.
Now, deleting above $R_1, \dots R_m$ removes all $k$-faces of $\Delta$ from $T^k$. All remaining $k$-faces are of the form $v\ast G$, where $G$ is a $(k-1)$-face of $\Delta$. All these faces can be removed with elementary collapses $(G, v\ast G)$. Hence, deleting above $R_1, \dots R_m$ turns $T^k$ into a ridge-chordal complex. 

\item Links: For weak-ridge-chordality, the simplicial complex of Example \ref{ex:WrongLinks},
with the property that the link of vertex $5$ is a 4-cycle, 
is skeleton-weakly-ridge-chordal. (Deleting above $45$ and $35$ yields the boundary of a tetrahedron). It is not, however, (skeleton-)ridge-chordal, as shown in the proof of Theorem \ref{thm:weaklyVertexChordal}, part (iii). 
For skeleton-ridge-chordality, consider the skeleton-ridge-chordal simplicial complex 
\[
W = 123, 125,145, 235, 345.
\]
In this simplicial complex, the link of vertex 5 is the 4-cycle $12,23,34,14$. 
\item Deletions: Let $D^k = \operatorname{skel}_k(\operatorname{del}(v,\Delta))$. Its $k$-faces are exactly the $k$-faces of $\Delta$ that do not contain $v$. Let $R_1,R_2,\ldots,R_m$ be a sequence of ridges that proves $S^k$ ridge-chordal.  From such sequence, we simply discard the ridges that contain $v$. We claim that the remaining ridges prove $D^k$ ridge-chordal. Indeed, if $v\in R_i$, then every $k$-face deleted above $R_i$ contains $v$, so that step removes no $k$-face from $\operatorname{del}(v,\Delta)$. Hence it can be ignored.
Note that  if $R_1$ contains $v$, and $R_{2}$ is $k$-simplicial in $\operatorname{abdel}(R_1, \Delta)$, then $R_{2}$ is also $k$-simplicial in $S^k$. 
If $v\notin R_i$, then deleting above $R_i$ removes exactly those remaining $k$-faces containing $R_i$. Restricting to the $k$-faces not containing $v$, this step removes exactly the remaining $k$-faces of $D^k$ that contain $R_i$. 
After all remaining ridges have been used, every $k$-face not containing $v$ has been removed, because the original sequence removed every $k$-face of $S^k$. Therefore $D^k$
is ridge-chordal. \qedhere 
\end{compactitem}
\end{proof}

\newpage
%%%%%%%%%%%%%%%% SECTION 5 %%%%%%%%%%%%%%%%%%%%%%%
%%%%%%%%%%%%%%%% SECTION 5 %%%%%%%%%%%%%%%%%%%%%%%
\section{Chordality via decompositions}
%%%%%%%%%%%%%%%% SECTION 5 %%%%%%%%%%%%%%%%%%%%%%%
%%%%%%%%%%%%%%%% SECTION 5 %%%%%%%%%%%%%%%%%%%%%%%

Another famous characterization of chordality stems out of the work of Hajnal--Sur\'anyi \cite{HS58} and Dirac \cite{Dir61}:

\begin{lemma}[Dirac \cite{Dir61}]\label{lem:DiracBisimplicial}
In any chordal graph other than $K_n$, there are at least two simplicial vertices that are not connected by an edge to one another.
\end{lemma}

\begin{theorem}[{Hajnal--Sur\'anyi \cite{HS58}, Dirac \cite{Dir61}}] \label{thm:DiracSplits}
A graph $G$ is chordal if and only if either $G$ is the complete graph, or $G$ splits as $G_1 \cup G_2$, where each $G_i$ is a proper induced chordal subgraph of $G$, and $G_1 \cap G_2$ is a clique. 
\end{theorem}

The two statements are related: The `if' part of Theorem \ref{thm:DiracSplits} is usually proven via Lemma \ref{lem:DiracBisimplicial}. In fact, suppose $G=G_1 \cup G_2$, with each $G_i$ chordal, and $G_1 \cap G_2$ a clique. By Lemma \ref{lem:DiracBisimplicial}, $G_1$ has at least two simplicial vertices that are not connected by an edge; hence, they cannot both belong to the clique. But any simplicial vertex $v$ in $G_1$ that is not in $G_1 \cap G_2$ has no neighbors in $G$ other than those already present in $G_1$, and is therefore  simplicial also in $G$. So we can label $v$ by $n$. If we delete $v$ from $G$, we obtain a new graph $G'$ on $n-1$ vertices that splits as $G'=G'_1 \cup G_2$, with $G'_1$, $G_2$ induced chordal, and with $G'_1 \cap G_2$ still a clique. By induction, we conclude.

In generalizing Lemma \ref{lem:DiracBisimplicial} and Theorem \ref{thm:DiracSplits} to higher dimensions,  there are several difficulties that we have to face, highlighted by the following Proposition:

\begin{figure}[h]
\centering
    \includegraphics[height=0.18\linewidth]{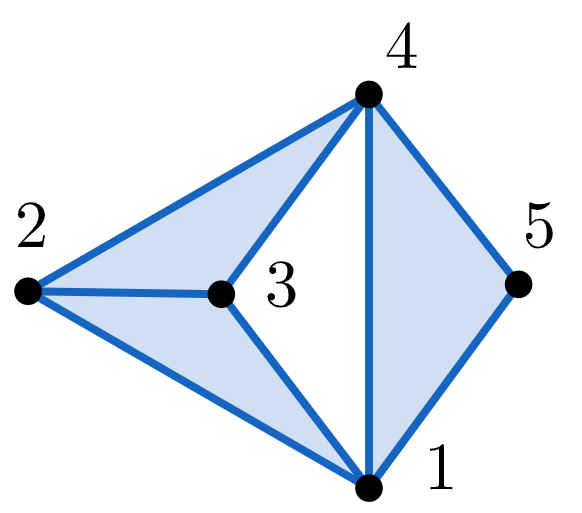}
    \hspace{17mm}
    \includegraphics[height=0.18\linewidth]{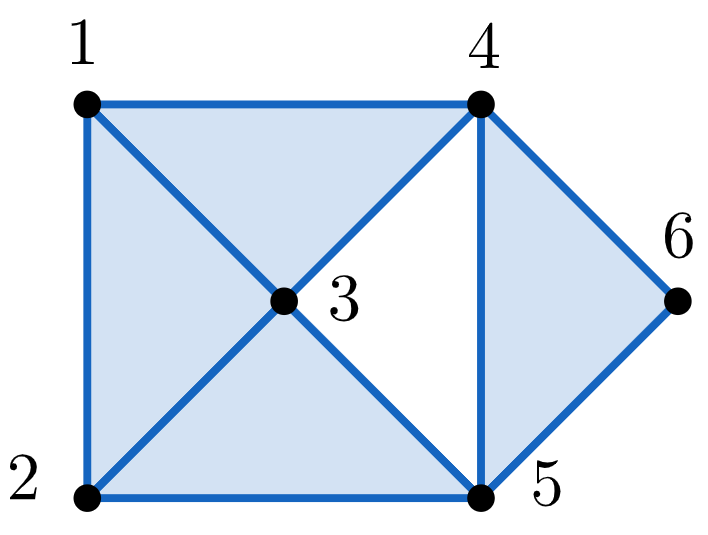}
     \caption{The simplicial complexes $X$ and $Y$ from Proposition \ref{prop:ChordalNotBisimplicial}.}
    \label{fig:ChordalNotBisimplicial}
\end{figure}

\begin{proposition}\label{prop:ChordalNotBisimplicial} Already in dimension two, 
\begin{compactenum}[\rm (1)]
\item some skeleton-E-chordal simplicial complex has only one mid-simplicial vertex, which is also the only vertex that is skeleton-very-weakly-simplicial;
\item some weakly-chordal complex has only one very-weakly-simplicial vertex;
\item some clique-chordal complex has only one clique-simplicial vertex.
\end{compactenum}
\end{proposition}

\begin{proof} 
\begin{compactenum}[(1)]
\item In the skeleton-E-chordal complex
\[X=123, 234, 145\]
there is  only one skeleton-very-weakly-simplicial vertex, namely, $5$. It is also the only mid-simplicial vertex. Note that it is not the only weakly-simplicial vertex: $1, 4$ and $5$ are all weakly-simplicial. Also, $2,3$ and $5$ are all skeleton-clique-simplicial.
\item The pinched annulus $P^2_n$ of Lemma \ref{lem:pinchedannulus} is very-weakly-chordal, and even weakly-chordal for $n=5$, but the only  very-weakly-simplicial vertex is the pinch point.
\item The complex
\[ Y = 123, 134, 235, 456\]
is clique-chordal with this given labeling. However, the only clique-simplicial vertex is $6$. 
 \qedhere
\end{compactenum}
\end{proof}

\begin{corollary}\label{cor:ChordalNotBisimplicial} Already in dimension two:
\begin{compactenum}[\rm (1)]
\item some simplicial complex $\Delta$, though neither mid-chordal nor skeleton-very-weakly-chordal, splits as $\Delta=\Delta_1 \cup \Delta_2$, where each $\Delta_i$ is skeleton-E-chordal and $\Delta_1 \cap \Delta_2$ is a vertex;
\item some simplicial complex $\Delta$, though not clique-chordal, splits as $\Delta=\Delta_1 \cup \Delta_2$, where each $\Delta_i$ is clique-chordal and $\Delta_1 \cap \Delta_2$ is a vertex. The same holds  if   `clique-' is replaced by `weakly-'.
\item some simplicial complex $\Delta$, though not very-weakly-chordal, splits as $\Delta=\Delta_1 \cup \Delta_2$, where each $\Delta_i$ is very-weakly-chordal and $\Delta_1 \cap \Delta_2$ is an edge. 
\end{compactenum}
\end{corollary}

\begin{proof} 
\begin{compactenum}[(1)]
\item Consider the weakly-chordal complex
\[ 123, 234, 149, 567, 678, 589.\]
This complex is a one-point union of two copies of the complex $X$ from Proposition \ref{prop:ChordalNotBisimplicial}, item (1). The gluing point is the only mid-simplicial vertex of the two copies, that is, the vertex eventually labeled by $9$. By construction the result of the gluing has no mid-simplicial and no skeleton-very-weakly-simplicial vertices.
\item Similarly to part (1), take two copies of the pinched annulus $P^2_5$ and glue them at the pinch point; the resulting complex will have no weakly-simplicial vertex. (Caveat: It will have very-weakly-simplicial vertices. In fact, it is easy to see that a single-vertex gluing of two very-weakly-chordal simplicial complexes is still very-weakly-chordal.) Similarly, if we take two copies of the simplicial complex $Y$ from  Proposition \ref{prop:ChordalNotBisimplicial}, item (3), and glue them at the unique clique-simplicial vertex, we get a simplicial complex without clique-simplicial vertices.
\item Let $\Delta$ be the simplicial complex
\[ \Delta = 
123, \: 145, \: 234, \:345, \:
127, \: 168, \: 267, \:678. \]
Let $\Delta_1$ (resp. $\Delta_2$) be the subcomplex of $\Delta$ formed by the first four (resp. last four) facets. Then $\Delta_1$ and $\Delta_2$ are isomorphic, and both weakly-chordal. The intersection $\Delta_1 \cap \Delta_2$ is just the edge $12$. However, $\Delta = \Delta_1 \cup \Delta_2$  lacks very-weakly-simplicial vertices. Hence, by Proposition \ref{prop:SimplicialVertices}, $\Delta$ cannot be very-weakly-chordal.
\qedhere
\end{compactenum}
\end{proof}

In view of Proposition \ref{prop:ChordalNotBisimplicial} and Corollary \ref{cor:ChordalNotBisimplicial}, the only chordality property left, for which we have hopes of  extending  Lemma \ref{lem:DiracBisimplicial} to higher dimensions, is skeleton-clique-chordality. And indeed, in contrast with Proposition \ref{prop:ChordalNotBisimplicial}, we have the following results:

\begin{lemma}[essentially Dirac] \label{lem:DiracBiSkClSimplicial}
For any skeleton-clique-chordal complex $\Delta$, either $\operatorname{skel}_1(\Delta)$ is a clique, or there are two skeleton-clique-simplicial vertices that are not connected by an edge. 
\end{lemma}

\begin{proof}
If $v$ is simplicial in the $1$-skeleton of $\Delta$, then $v$ is skeleton-clique-simplicial in $\Delta$. In fact, this is an if and only if. So the result is just a reformulation of Dirac's Lemma \ref{lem:DiracBisimplicial}.    
\end{proof}

\begin{lemma} \label{lem:DiracBiSkClSimplicial1}
Let $\Delta$ be a skeleton-clique-chordal simplicial complex. Let $K\subseteq V(\Delta)$ be a clique in $\operatorname{skel}_1 (\Delta)$. There exists a vertex labeling that proves $\Delta$ skeleton-clique-chordal in which the vertices of $K$ are labeled first.
\end{lemma}

\begin{proof}
If the 1-skeleton of $\Delta$ is complete, then any vertex labeling proves $\Delta$ skeleton-clique-chordal. Otherwise, by Lemma \ref{lem:DiracBiSkClSimplicial}, $\Delta$ has two skeleton-clique-simplicial vertices that are not connected by an edge. So these two vertices cannot both belong to the clique $K$. So there is a skeleton-clique-simplicial
vertex $v$ not in $K$. The deletion of $v$ is again skeleton-clique-chordal. Moreover, $K$ remains a clique also in $\operatorname{skel}_1(\del(v, \Delta))$. By inductive assumption, $\del(v, \Delta)$ admits a vertex labeling proving its skeleton-clique-chordality,  in which the labels used for the vertices of $K$ are the lowest. We now extend this labeling to a labeling for $\Delta$, by assigning the label $n$ to the vertex $v$.
We claim that this  labeling is the desired one. Indeed, let $f \ne g$ be faces of $\Delta$ of same size, at least $2$, and same maximum. There are two cases:
\begin{compactitem}
\item if $\max f = \max g \ne n$, then neither $f$ nor $g$ contains $v$, since $v$ is the vertex with label $n$. Thus $f,g\in \del(v, \Delta)$, and the desired conclusion follows by how the vertex labeling of $\del(v, \Delta)$ was chosen.
\item if $\max f = \max g  =n$, then both $f$ and $g$ contain $v$. Since $v$ is skeleton-clique-simplicial in $\Delta$, every 2-element subset of $f\cup g$ is an edge of $\Delta$. \qedhere
\end{compactitem}
\end{proof}

We are now half-ready to extend Theorem \ref{thm:DiracSplits}. In fact, we still have to discuss how to extend its `if' part. For graphs, the `if' part of Theorem \ref{thm:DiracSplits} is typically proven as follows: One chooses a simplicial vertex $v$, and then one splits the graph $G$ into the union of the deletion of $v$ and the so-called `neighborhood' of $v$. The neighborhood of $v$ is the induced subgraph obtained by deleting of all the vertices that are not connected by an edge to $v$. We mimic this idea closely:

\begin{definition} The \emph{neighborhood} $N(v, \Delta)$ of a vertex $v$ in a $d$-dimensional simplicial complex $\Delta$ is the simplicial complex obtained  from $\Delta$ by deleting all vertices that are neither $v$, nor connected by an edge to $v$.
\end{definition}

\begin{lemma}\label{lem:StarNeigh} Let $\Delta$ be any simplicial complex.
For any vertex $v$ of $\Delta$, $\operatorname{Star}(v, \Delta) \subseteq N(v, \Delta)$. \\The inclusion may be strict even for graphs.
\end{lemma}

\begin{proof} 
 Let $F$ be any $k$-face of $\Delta$ that contains $v$. Let $v, x_1, \ldots, x_k$ be the vertices of $F$.  Clearly, all edges $vx_i$ are in $\Delta$. Hence, none of the $x_i$'s is deleted when we pass from $\Delta$ to $N(v, \Delta)$. Thus $F$ is in  $N(v, \Delta)$. This shows the inclusion. 
 As for the strictness: if $\Delta$ is the graph $C_3$, then $N(v, C_3) = C_3$, while $\operatorname{star}(v, C_3)$ is $C_3$ minus one edge.
\end{proof}

\begin{theorem} \label{thm:NeighDec}
Let $\Delta$ be a $d$-dimensional simplicial complex that is not a vertex neighborhood. Let P be any element of the list
\[\wp'=\{ \textrm{ skeleton-E-, skeleton-mid-, skeleton-(very-)weakly-, skeleton-clique-}\}.\]
Then $\Delta$ is {P}-chordal $\Longleftrightarrow$ $\Delta$ splits as $\Delta_1 \cup \Delta_2$, where:
\begin{compactenum}[$\:$ \rm (i)]
\item  $\Delta_1$ is P-chordal; 
\item $\Delta_2$ is the neighborhood in $\Delta$ of a single vertex $v$;
\item $\Delta_1=\del(v,\Delta)$;
\item $v$ is P-simplicial in $\Delta_2$.
\end{compactenum}
\end{theorem}

\begin{proof}`$\Rightarrow$': Choose any P-simplicial vertex $v$ in $\Delta$. In Section \ref{sec:stability}, we saw that all the properties in the list $\wp'$ are maintained under arbitrary vertex deletions. Hence, if we  set $\Delta_1 =  \operatorname{del}(v, \Delta)$ and $\Delta_2=N(v, \Delta)$, both $\Delta_i$ are P-chordal with the respective induced labelings. Condition (iv) is easy to verify. In general, we have no information on the dimensions of $\Delta_1$ and $\Delta_2$.
\\
`$\Leftarrow$': Let us choose a vertex labeling on $\Delta_1$ that makes it P-chordal. The idea is to extend this labeling to $\Delta$ by assigning the label $n$ to vertex $v$, which by assumption (iii) is the only vertex missing in $\Delta_1$. A priori, we do not know if restricted to $\Delta_2$, this labeling proves $\Delta_2$ P-chordal. However, let $F,G$ be size-$k$ faces of $\Delta$ with same maximum. There are two cases:
\begin{compactitem}
\item if $\max F = \max G < n$, then $F$ and $G$ do not contain $v$, so they are faces of $\Delta_1$. For the same reason, any face $H$ in $F \cup G$ is a face of $\Delta_1$. So P-chordality in this case is implied by the P-chordality of $\Delta_1$ with respect to the original labeling. 
\item if $\max F = \max G = n$, then $F$ and $G$ are in the star of $v$. By Lemma \ref{lem:StarNeigh}, $F$ and $G$ are in $\Delta_2$. The P-chordality property is implied by the P-simpliciality of $v$ in $\Delta_2$. 
\qedhere
\end{compactitem}
\end{proof}

Iterating the above theorem, every skeleton-E-chordal complex is decomposed into skeleton-E-chordal neighborhoods that intersect in complexes whose $1$-skeleta are cliques. Typically, the purity property gets lost in the iterations, which is why we stated the theorem in the non-pure setup.
We can finally provide a proposed generalization of Theorem \ref{thm:DiracSplits}:

\begin{theorem} \label{thm:Decompositions}
Let $\Delta$ be a $d$-dimensional simplicial complex.
\begin{compactenum}[\rm (1)]
\item  If $\Delta$ is skeleton-E-chordal, then either $\Delta$ is a vertex neighborhood, or it splits as $\Delta = \Delta_1 \cup \Delta_2$, where $\Delta_1$ and $\Delta_2$ are proper skeleton-E-chordal subcomplexes of $\Delta$, and $\Delta_1 \cap \Delta_2$ is a complex whose $1$-skeleton is a clique $K$.
\item Suppose $\Delta$  splits as $\Delta = \Delta_1  \cup \Delta_2$, where each $\Delta_i$ is a proper skeleton-clique-chordal subcomplex of $\Delta$, and $\Delta_1 \cap \Delta_2$ is a complex whose $1$-skeleton is a clique $K$. 
Then $\Delta$  is skeleton-clique-chordal.
\end{compactenum} 
\end{theorem}

\begin{proof}
\begin{compactenum}[\rm (1)]
\item This is basically the $\Rightarrow$ direction of Theorem \ref{thm:NeighDec}.
\item By Lemma \ref{lem:DiracBiSkClSimplicial1}, we may choose a skeleton-clique-chordal labeling of $\Delta_i$ in which the vertices of $K$ are labeled first. So we can label the vertices of $\Delta$ as follows: first the vertices of $K$; then the vertices of $\Delta_1 - K$ in an  order that proves $\Delta_1$ skeleton-clique-chordal; and finally, the vertices of $\Delta_2 - K$, in an order that restricted to $\Delta_2$ would prove $\Delta_2$ skeleton-clique-chordal. We claim this is our desired labeling. In fact, let $f,g\in \Delta$ be two distinct faces of same size, at least 2, and  same maximum. There are three cases:
\begin{compactitem}[$\bullet$]
\item If $\max f = \max g$ is in $\Delta_1$ but not in $K$, then $\max f = \max g$ is not a vertex of $\Delta_2$. Hence neither of $f, g$  is in $\Delta_2$. So
$f,g$ are both in $\Delta_1$. Since our labeling restricted to $\Delta_1$ proves it skeleton-clique-chordal, every size-2 subset of $f\cup g$ is an edge of $\Delta_1$ and thus of $\Delta$.
\item The case `$\max f = \max g$ is in $\Delta_2$ but not in $K$' is symmetric. The conclusion follows from the fact that our labeling, restricted to $\Delta_2$, proves it skeleton-clique-chordal.
\item Finally, suppose $\max f = \max g \in K$. Since every vertex outside $K$ has larger label, this implies that both $f,g$ have all their vertices in $K$. Since $K$ is a clique, every two-element subset of $f\cup g$ is an edge of $K$ and thus of $\Delta$. \qedhere
\end{compactitem}
\end{compactenum}
\end{proof}

\begin{corollary} \label{cor:Decompositions} Let $\Delta$ be a subflag $d$-dimensional simplicial complex.  Let P be in the list \\
$\wp'=\{ \textrm{ skeleton-E-, skeleton-mid-, skeleton-weakly-, skeleton-very-weakly-, skeleton-clique-}\}$.\\
Suppose the $1$-skeleton of $\Delta$ is not a complete graph. Then 
\[ \Delta \textrm{ is P-chordal }\Longleftrightarrow
\Delta  \textrm{ splits as } \Delta = \Delta_1  \cup \Delta_2,\] where each $\Delta_i$ is an induced $P$-chordal proper subcomplex of
$\Delta$, and $\Delta_1 \cap \Delta_2$ is a complex whose $1$-skeleton is a clique.
\end{corollary}

\begin{proof}
By Lemma \ref{lem:SubflagChordal} and Theorem \ref{thm:hierarchy2}, all properties in $\wp'$ are equivalent for subflag complexes. It thus suffices to consider skeleton-E-chordality.

\noindent `$\Rightarrow$': Suppose $\Delta$ is skeleton-E-chordal and its $1$-skeleton is not complete. By Lemma \ref{lem:DiracBiSkClSimplicial}, $\Delta$ has a skeleton-clique-simplicial vertex $v$ that is not adjacent to every vertex. Since $\Delta$ is subflag, $v$ is also skeleton-E-simplicial. Set
\[
\Delta_1=\operatorname{del}(v,\Delta),
\qquad
\Delta_2=N(v,\Delta).
\]
Both are proper induced subcomplexes of $\Delta$, and are skeleton-E-chordal by Proposition \ref{prop:SuperCLD}. By Lemma \ref{lem:StarNeigh}, $\Delta=\Delta_1\cup\Delta_2$. Moreover, $\Delta_1\cap\Delta_2$ is the induced subcomplex on the neighbors of $v$, which form a clique since $v$ is skeleton-clique-simplicial.

\noindent `$\Leftarrow$': Conversely, if $\Delta=\Delta_1\cup\Delta_2$ is such a splitting, then $\Delta$ is skeleton-clique-chordal by Theorem \ref{thm:Decompositions}(2), hence skeleton-E-chordal by Lemma \ref{lem:SubflagChordal}.
\end{proof}

For $d=1$, since all graphs (and all simplicial complexes) are subflag, the previous corollary boils down precisely to Theorem 
 \ref{thm:DiracSplits}.

\newpage
%%%%%%%%%%%%%%%% SECTION 6 %%%%%%%%%%%%%%%%%%%%%%%
%%%%%%%%%%%%%%%% SECTION 6 %%%%%%%%%%%%%%%%%%%%%%%
\section{Chordality via Alexander duality} \label{sec:Alexander}
%%%%%%%%%%%%%%%% SECTION 6 %%%%%%%%%%%%%%%%%%%%%%%
%%%%%%%%%%%%%%%% SECTION 6 %%%%%%%%%%%%%%%%%%%%%%%

Perhaps the most intriguing characterization of graph chordality comes from the work of Fr\"oberg \cite{Fro90} and Eagon--Reiner \cite{ER98}. Let us recall a few combinatorial topology notions:

\begin{definition}[Alexander dual] Let $n \in \NN$. Let $\Delta$ be a simplicial complex with vertex set contained in $[n]$. The \emph{Alexander dual of $\Delta$ with respect to $n$} is the simplicial complex
\[
\Delta^\vee=\{F\subseteq [n]\colon [n]- F\notin \Delta\}.
\]
The facets of $\Delta^\vee$ are the complements in $[n]$ of the minimal non-faces of $\Delta$. Unless specified otherwise, one usually chooses $[n]$ to be \emph{equal} to the vertex set of $\Delta$.
\end{definition}

It is easy to see that $(\Delta^\vee)^\vee=\Delta$. The name ``Alexander'' comes from the facts that  $\Delta^\vee$ is a deformation retract of the complement of $\Delta$ inside the boundary of the simplex  $\Sigma_n$, cf.~\cite{BT09}. Since such boundary is a sphere, topological Alexander duality relates the homology of $\Delta$ to the cohomology of its complement, and thus of $\Delta^\vee$ \cite{BT09}. As a consequence, any simplicial complex $\Delta$ is \emph{acyclic} (in the topological sense of having trivial reduced homologies) if and only if  $\Delta^\vee$ is. A similar result holds for the following combinatorial strengthening of contractibility:

\begin{definition}[Non-evasive]
A (not necessarily pure) simplicial complex $\Delta$, which is neither $\emptyset$ nor $\{\emptyset\}$, is  \emph{non-evasive} if either (i) $\dim \Delta=0$ and $\Delta$ is a single vertex, or (ii) $\dim \Delta \ge 1$ and $\Delta$ has a vertex $v$ of $\Delta$ such that  $\operatorname{del}(v,\Delta)$ and $\operatorname{link}(v,\Delta)$ are both non-evasive.
\end{definition}

Kahn, Saks and Sturtevant proved that non-evasive $\Rightarrow$ collapsible $\Rightarrow$ contractible $\Rightarrow$  acyclic, and all converses are false  \cite{KSS84}. They also noticed that a simplicial complex $\Delta$ (different than the simplex) is non-evasive if and only if  $\Delta^\vee$ is \cite{KSS84}.  In contrast, contractibility is not preserved under Alexander duality \cite{MR14, SS03}. In fact, neither is collapsibility:  the Alexander dual of the dunce hat is collapsible, but the dunce hat itself is not  \cite{BL13}.

A similar notion to non-evasiveness was introduced in the pure case by Provan and Billera \cite{PB80}, and in the general case by Bj\"orner and Wachs \cite{BW97}: 

\begin{definition}[Vertex-decomposable]
A vertex $v$ of a simplicial complex $\Delta$ is \emph{shedding} if no facet of $\operatorname{link}(v, \Delta)$ is also a facet of $\operatorname{del}(v, \Delta)$. A (not necessarily pure) simplicial complex $\Delta$ is  \emph{vertex-decomposable} if either (i) $\Delta$ is a simplex, or (ii) $\Delta = \{ \emptyset \}$, or (iii) $\Delta =  \emptyset$, or (iv) there is a shedding vertex $v$ of $\Delta$ such that  $\operatorname{del}(v,\Delta)$ and $\operatorname{link}(v,\Delta)$ are both vertex-decomposable.
\end{definition}

The complex $123, 345$ is non-evasive, but not vertex-decomposable. In contrast, two disjoint points form a simplicial complex that is vertex-decomposable, but not non-evasive.  Thus vertex-decomposable does not imply acyclic.
However, there is a connection between the two notions:

\begin{lemma} \label{lem:VDaNE} All vertex-decomposable nonempty acyclic complexes are non-evasive.
\end{lemma}

\begin{proof}  It is known (cf.~e.g.~\cite[Corollary 2.11(i)]{KM16} for a stronger claim), and not difficult to prove directly from Wachs' \cite[Lemma 6]{Wac99}, that if $\Delta$ is a simplicial complex with a shedding vertex $v$ such that
$\del(v,\Delta)$ and $\link(v,\Delta)$ are both (pure or nonpure) shellable, then
\[
\beta_i(\Delta) \;=\; \beta_i(\del(v,\Delta)) \;+\; \beta_{i-1}(\link(v,\Delta)) \qquad \text{for every }
i.
\]
 Now let $\Delta$ be a vertex-decomposable complex that is neither $\emptyset$ nor $\{\emptyset\}$. If $\Delta$ is a single simplex, it is non-evasive. Otherwise, let $v$ be a shedding vertex such that both $\link(v, \Delta)$ and $\del(v,\Delta)$ are vertex-decomposable. In case $\Delta$ is acyclic, the ``Betti splitting'' above yields
\[
0 = \beta_i(\del(v,\Delta)) + \beta_{i-1}(\link(v,\Delta)) \qquad \text{for every }
i.
\]
This implies that  $\del(v,\Delta)$ and $\link(v,\Delta)$ are  acyclic and non-empty. By inductive assumption, they are nonevasive. Thus $\Delta$ is nonevasive as well.
\end{proof}

These notions above allowed Fr\"oberg to characterize graph chordality algebraically. 
Let $G$ be a graph on vertex set $[n]$, and let $\overline G$ be its complement
graph. We write
\[
I(\overline G)=\langle x_i x_j : ij\notin G\rangle
\]
for the edge ideal of $\overline G$ in a polynomial ring with $n$ variables. Equivalently, $I(\overline G)=I_{\operatorname{Cl}(G)}$,
where $\operatorname{Cl}(G)$ is the clique complex of $G$ -- which, as the name suggests, is the simplicial complex formed by the sets of vertices in the cliques of $G$. For resolutions of ideals, we refer the reader to \cite{MS05}.

\begin{theorem}[Fr\"oberg \cite{Fro90}, Eagon--Reiner \cite{ER98}, Woodroofe \cite{Woo11}]
\label{thm:FER}
For any graph $G$, t.f.a.e.:
\begin{compactenum}[\rm (i)]
\item $G$ is chordal;
\item $I(\overline G)$ has a linear resolution (over some, or equivalently any, field);
\item $(\operatorname{Cl}(G))^\vee$ is Cohen--Macaulay (over some, or equivalently any, field);
\item $(\operatorname{Cl}(G))^\vee$ is vertex-decomposable.
\end{compactenum}
Moreover, if $G$ is regarded as a $1$-dimensional simplicial complex on $n$ vertices, 
\[
(\operatorname{Cl}(G))^\vee=\operatorname{pure-skel}_{n-3}(G^\vee).
\]
\end{theorem}

\begin{example}
Consider the chordal graph $G=12,13,23,14$.
If we view $G$ as a $1$-dimensional complex, $G^\vee=12,13,4$.
Passing to the pure $1$-skeleton removes the extra lower-dimensional facet $4$, which corresponds to the missing triangle $123$. One has $\pureskel_{1}(G^\vee)= 12,13$, which is vertex-decomposable. Note that $\pureskel_{1}(G^\vee)=(\Cl(G))^\vee$,
since $\Cl(G)= 123,14$.
Note also that the full dual $G^\vee= 12,13,4$ is still
(nonpure) vertex-decomposable.
\end{example}

For higher-dimensional complexes the graph-theoretic analogy must be formulated with some care. Let $\Delta$ be a pure $d$-dimensional simplicial complex on $[n]$. The analog of the complement edge ideal of a graph is the ideal generated by the missing $d$-faces of $\Delta$:
\[
I_d(\Delta)=\langle x_F : |F|=d+1,\ F\notin \Delta\rangle .
\]
By Eagon--Reiner \cite{ER98}, the ideal $I_d(\Delta)$ has a linear resolution if and only if its Alexander dual complex is Cohen--Macaulay. This Alexander dual complex is precisely $\operatorname{pure-skel}_{n-d-2}(\Delta^\vee)$.
Thus, the analogue of the Fröberg--Eagon--Reiner characterization of chordal graphs is
\begin{equation}\label{eq:EagonReiner1}
I_d(\Delta)\text{ has} \text{ a linear resolution} \quad\Longleftrightarrow\quad
\operatorname{pure-skel}_{n-d-2}(\Delta^\vee)\text{ is Cohen–Macaulay}.
\end{equation}

However, for $d \ge 2$, having linear resolution (or being Cohen–Macaulay) over `every' vs. `some' field are inequivalent properties.
Moreover, the vertex-decomposability of $
\operatorname{pure-skel}_{n-d-2}(\Delta^\vee)
$ is strictly (much!) stronger than its Cohen--Macaulayness. And finally, as we shall see, some chordality notions existing in the literature imply the existence of a linear resolution for $I_d(\Delta)$, but none of these known implications can be reversed.

This $I_d(\Delta)$ should not be confused with the Stanley--Reisner ideal $I_\Delta$, which is generated by all minimal non-faces of $\Delta$, of any dimension. When $\Delta^\vee$ is pure, Eagon--Reiner's theorem  yields
\begin{equation}\label{eq:EagonReiner2}
I_\Delta \text{ has a linear resolution}
\quad\Longleftrightarrow\quad
\Delta^\vee \text{ is Cohen--Macaulay}.
\end{equation}

It frequently happens that $\Delta^\vee$ is not pure, even if $\Delta$ is. Equivalence (\ref{eq:EagonReiner2}) has been extended to the nonpure case by Herzog and Hibi \cite[Theorem 2.1(a)]{HH99}, who showed that 
\begin{equation}\label{eq:EagonReiner3}
I_\Delta \text{ is componentwise linear}
\quad\Longleftrightarrow\quad
\Delta^\vee \text{ is sequentially-Cohen--Macaulay}.
\end{equation}

The equivalences (\ref{eq:EagonReiner1}) and  (\ref{eq:EagonReiner3}) are different,  but related via the following criterion: 

\begin{theorem}[{Duval's criterion \cite[Theorem 3.3]{Duv96}}] \label{thm:duval}
If $\Delta$ is  sequentially Cohen--Macaulay, then $\operatorname{pure-skel}_{t}(\Delta)$ is Cohen--Macaulay for all $t \in \{0, \ldots, \dim \Delta\}$. 
\end{theorem}

Thus, componentwise linearity of $I_\Delta$ implies linear resolution of $I_r(\Delta)$ for every $r\ge 0$.

\medskip

\subsection{From chordality to vertex-decomposability}
The next two results are new and directly inspired by Woodroofe's \cite[Lemma 6.8]{Woo11} and \cite[Theorem 6.9]{Woo11}, respectively:

\begin{lemma} \label{lem:shedding} Let $\Delta$ be a simplicial complex with vertex set exactly $[n]$. 
Let $v$ be a skeleton-very-weakly-simplicial vertex of $\Delta$. If $v$ appears in
$\Delta^\vee$,
then $v$ is a shedding vertex of $\Delta^\vee$.
\end{lemma}

\begin{proof}
Let $g$ be any facet of $\operatorname{link}(v, \Delta^\vee)$.
By the assumption,   
$\Delta^\vee$ has a facet $G$ of the form $G = v \ast g$. 
Let $H$ be the complement in $[n]$ of $G$.  By construction, $H$ does not contain $v$.  By definition of Alexander dual,  $H$ is a minimal non-face of $\Delta$. 
In particular, $\dim H \ge 1$. So pick two distinct vertices $a, b$ in $H$, which must also be different from $v$,
and consider the faces of $\Sigma_n$
\[ A := H - \{b\} \cup \{v \} \quad \textrm{ and } 
\quad B := H - \{a\} \cup \{v\}. 
\]
Note that $A$ and $B$ have the same size of $H$ and are adjacent: in fact, their intersection is $H -\{a, b\}\cup \{v\}$, which has size one less. Were both $A$ and $B$ faces of $\Delta$, the skeleton-very-weakly-simplicial assumption on $v$ would imply $H \in \Delta$, a contradiction. So at least one of $A$, $B$ is not in $\Delta$. Up to swapping the labels of $a$ and $b$, we can assume $A \notin \Delta$. Let $C$ be the complement of $A$ in $[n]$. By definition of Alexander dual, $C$ is a face of $\Delta^\vee$. By construction, $C$ contains strictly $g$, but not  $v$. Hence, $C$ is a face of $\del(v, \Delta^\vee)$ strictly containing $g$. So $g$ cannot be a facet of $\del(v, \Delta^\vee)$. By the arbitrariness of $g$, we conclude that $v$ is a shedding vertex.
\end{proof}

\begin{figure}[ht]
    \centering
    \includegraphics[height=0.15\linewidth]{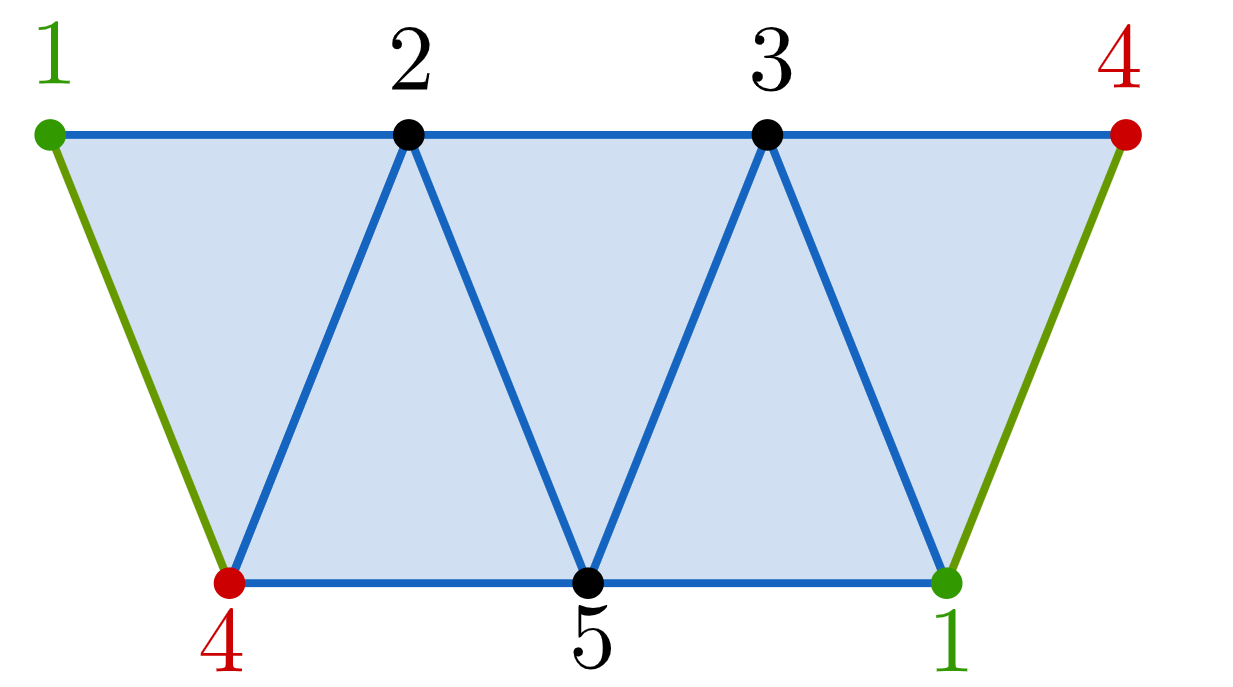}
    \caption{The Möbius band from Theorem \ref{thm:VD1} is the Alexander dual of the graph $C_5$.}
    \label{fig:DualToC5}
\end{figure}

\begin{theorem} \label{thm:VD1}
If $\Delta$ is skeleton-E-chordal, then $\Delta^\vee$ is vertex-decomposable.\\
The converse is false, already for $\dim (\Delta^\vee) =1$.
\end{theorem}

\begin{proof} We prove by induction on $n$ that if $\Delta$ is skeleton-E-chordal with vertex set $V$ contained in $[n]$, then the Alexander dual with respect to $[n]$ is vertex-decomposable. 
%In general, Alexander duals may be taken with respect to any set $[n]$ that contains the vertex set $V$ of $\Delta$. For our claim, we can always reduce ourselves to the case 
First of all, we may always reduce ourselves to the case $V=[n]$. In fact, if some $x$ in $[n]$ is not a vertex of $\Delta$, then $\Delta^\vee$ is the union of the simplex $\Sigma$ on $[n] - x$ with the cone $x*D$, where $D$ is the dual of $\Delta$ with respect to $[n] - x$. 
If $x$ is not a vertex of $\Delta^\vee$, then $\Delta^\vee=\Sigma$ is a simplex. Otherwise, $x$ is a shedding vertex of $\Delta^\vee$: $\del(x,\Delta^\vee)=\Sigma$ and $\link(x,\Delta^\vee)=D$, while $[n]-x\notin D$ since $\emptyset\in\Delta$. Since $\Sigma$ is a simplex and $D$ is vertex-decomposable by induction, so is $\Delta^\vee$.
So let us proceed with the induction: If $n<1$, both $\Delta=\{\emptyset\}$ and $\Delta = \emptyset$ have vertex-decomposable duals by definition. 
If $n\geq 1$, let $v$ be a skeleton-E-simplicial vertex in $\Delta$. 
If $v$ does not appear in any face of $\Delta^\vee$, then 
$\Delta^\vee = \operatorname{del}(v,\Delta^\vee) = [\operatorname{link}(v,\Delta)]^\vee$. But by Proposition \ref{prop:SuperCLD}, $\operatorname{link}(v,\Delta)$ is skeleton-E-chordal; also, it has fewer vertices than $\Delta$. Hence, its Alexander dual is vertex-decomposable by inductive assumption, and we are done. 
If instead $v$ does appear in $\Delta^\vee$, then by Lemma \ref{lem:shedding} $v$ is a shedding vertex of $\Delta^\vee$. Moreover: 
\begin{compactitem}
\item $\operatorname{link}(v,\Delta^\vee ) =
[\operatorname{del}(v,\Delta)]^\vee$ 
is vertex-decomposable by the inductive assumption, since $\operatorname{del}(v,\Delta)$ is skeleton-E-chordal with $n-1$ vertices;
\item $\operatorname{del}(v, \Delta^\vee ) =
[\operatorname{link}(v,\Delta)]^\vee$
is vertex-decomposable by the inductive assumption, since $\operatorname{link}(v,\Delta)$ is skeleton-E-chordal with at most $n-1$ vertices. 
\end{compactitem}
In conclusion, $\Delta^\vee$ is vertex-decomposable (though not necessarily pure). \\ As for the converse: The $5$-cycle is vertex-decomposable. Its Alexander dual
\[ (C_5)^\vee = 124, 134, 135, 235, 245, \]
depicted in Figure \ref{fig:DualToC5}, is a triangulation of the M\"obius band. Since $(C_5)^\vee$ lacks very-weakly-simplicial vertices, it is not very-weakly-chordal; hence certainly not skeleton-E-chordal.
\end{proof}

\begin{corollary} For a skeleton-E-chordal simplicial complex $\Delta$ that is neither $\emptyset$ nor $\{\emptyset\}$, the following are equivalent: 
\begin{compactenum}[\rm (a)]
\item $\Delta$ is acyclic;
\item $\Delta$ is contractible;
\item $\Delta$ is collapsible;
\item $\Delta$ is non-evasive.
\end{compactenum}
\end{corollary}

\begin{proof} The claim is clear if $\Delta$ is a simplex. Otherwise: Since  (d) $\Rightarrow$ (c) $\Rightarrow$ (b) $\Rightarrow$ (a) are true for all complexes, let us focus on (a) implies (d). Let $\Delta$ be skeleton-E-chordal acyclic. By Theorem~\ref{thm:VD1} $\Delta^\vee$ is vertex-decomposable. It is also acyclic, since $\Delta$ is. By Lemma \ref{lem:VDaNE}, $\Delta^\vee$ is non-evasive. But then so is $\Delta=(\Delta^\vee)^\vee$.
\end{proof}

\begin{corollary} If $\Delta$ is skeleton-E-chordal, then $\Delta^\vee$ is sequentially Cohen--Macaulay.
If in addition all minimal non-faces of $\Delta$ have the same dimension, then $\Delta^\vee$ is Cohen--Macaulay. 
\end{corollary}

\begin{proof}
By Theorem~\ref{thm:VD1}, $\Delta^\vee$ is vertex-decomposable, hence sequentially Cohen--Macaulay.
All minimal non-faces of $\Delta$ have the same dimension if and only if $\Delta^\vee$ is pure.
\end{proof}

\begin{remark}\label{remark:E8dual}
Theorem \ref{thm:VD1} is best possible, in the sense that the assumption ``skeleton-E-chordal''  cannot be weakened: For example, consider the skeleton-mid-chordal complex  
\[I = 123,124,134,135,145,234,235,245\]
of Theorem \ref{thm:weaklyVertexChordal}.
Among its minimal non-faces are $125$, $345$, and also $1234$; hence its Alexander dual is $I^\vee = 12, 34, 5$. The pure 1-skeleton then is $12, 34$, which is disconnected, and therefore not Cohen--Macaulay. By Duval's Theorem \ref{thm:duval},
$I^\vee$ is not sequentially Cohen--Macaulay. In particular, $I^\vee$ is
not vertex-decomposable. Note that $I$ is not W-chordal ($I/5$ does not have a simplicial vertex) and not ridge-chordal (see Theorem \ref{thm:weaklyVertexChordal}).
\end{remark}

Let us compare Theorem~\ref{thm:VD1} with the corresponding results of
Emtander, Woodroofe,
Bigdeli--Yazdan-Pour--Zaare-Nahandi, and
Bigdeli--Faridi. For simplicity, we rephrase them all in terms of the pure-skeleton of the Alexander dual of $\Delta$.

\begin{itemize}
\item
\textbf{Emtander.} \cite[Def.~4.3 \& Theorem 5.1]{Emt10}
If $\Delta$ is pure E-chordal, $\operatorname{pure-skel}_{n-d-2}(\Delta^\vee)$ is Cohen--Macaulay.

\item
\textbf{Woodroofe.} \cite[Theorem 6.9]{Woo11}
If $\Delta$ is W-chordal $d$-dimensional on $n$ vertices, $\operatorname{pure-skel}_{n-k-2}(\Delta^\vee)$ is vertex-decomposable, where $k$ is the smallest dimension of a facet of $\Delta$. In particular, if $\Delta$ is pure W-chordal, $\operatorname{pure-skel}_{n-d-2}(\Delta^\vee)$ is vertex-decomposable.

\item
\textbf{Bigdeli--Yazdan-Pour--Zaare-Nahandi.} \cite[Theorem 3.3]{BYZ17}
If $\Delta$ is pure ridge-chordal, $\operatorname{pure-skel}_{n-d-2}(\Delta^\vee)$ 
is Cohen--Macaulay.

\item
\textbf{Bigdeli--Faridi.} \cite[Corollary 4.8]{BF20}
If $\Delta$ is chordal in the sense of
\cite[Definition 2.6]{BF20} (that is, if
$\operatorname{pure-skel}_k(\Delta)$ is
ridge-chordal for every $k\le d$), then
$\Delta^\vee$ is sequentially Cohen--Macaulay, i.e. $\operatorname{pure-skel}_{t}(\Delta^\vee)$ is Cohen--Macaulay for all $t$.
\end{itemize}

Let us start by pointing out that Theorem~\ref{thm:VD1} applies also to nonpure complexes. On the common ground of pure complexes, its hypothesis is stronger than the hypotheses in the corresponding theorems of Emtander, of Bigdeli--Faridi, and of Bigdeli--Yazdan-Pour--Zaare-Nahandi, but its conclusion is much stronger. For example, all triangulations of the $d$-sphere are Cohen--Macaulay (and even Gorenstein), but when $d \ge 3$, many of them are not vertex-decomposable.

Comparing now Theorem \ref{thm:VD1} to Woodroofe's theorem: The respective assumptions are incomparable, cf.~Example  \ref{ex:WWeird} and Non-Example \ref{nex:WWeird1}. We suspect that Woodroofe's property occurs more frequently in randomly generated complexes. On the other hand, skeleton-E-chordality is easier to test. Moreover, our conclusion is stronger: Woodroofe's theorem gives vertex-decomposability of one relevant pure skeleton of the Alexander dual, whereas Theorem~\ref{thm:VD1} gives vertex-decomposability of the whole Alexander dual, which by Duval's criterion is stronger. In fact, even skeleton-W-chordality is not strong enough to imply $\Delta^\vee$ being vertex-decomposable:

\begin{proposition} \label{prop:DualsAndVd}
There are skeleton-W-chordal complexes $\Delta$ such that $\Delta^\vee$ is not even sequentially Cohen--Macaulay.
\end{proposition}

\begin{proof}
Consider the skeleton-W-chordal simplicial complex $T$ from Remark \ref{rem:DualsAndVd0} and Figure~\ref{fig:notVDnotShellNotCM}, namely
\[
T= 124,125,134,135,234,235,45.
\]
The minimal non-faces of $T$ are $123,145,245,345$. So $T^\vee = 45,23,13,12 = \operatorname{pure-skel}_1(T^\vee)$, which is disconnected. Hence $\operatorname{pure-skel}_1(T^\vee)$ is not
Cohen--Macaulay. By Duval's Theorem \ref{thm:duval}, $T^\vee$ is not sequentially Cohen--Macaulay (and in particular, not vertex-decomposable). 
\end{proof}

In view of Theorem \ref{thm:VD1}, it is natural to ask whether skeleton-E-chordal complexes are themselves vertex-decomposable. The answer, in general, is negative: Two-dimensional examples like $\Delta_1 = 123, 345$ or $\Delta_2 = 123, 456$ are not even Cohen--Macaulay. However, a strongly-connected assumption (cf.~Definition \ref{def:SConn})  suffices to prove vertex-decomposability:

\begin{theorem} \label{thm:VD0}
Let $\Delta$ be a strongly-connected simplicial complex. \\ If $\Delta$ is skeleton-E-chordal, it is vertex-decomposable.
\end{theorem}

\begin{proof} Let $d = \dim \Delta$. Let $k$ be the number of facets. 
Without loss, we may assume $d \ge 2$ and $k \ge 2$, because all $0$-dimensional simplicial complexes, all connected graphs, and all simplices, are vertex-decomposable. Let $v$ be the vertex labeled by $n$, in a labeling that proves $\Delta$ skeleton-E-chordal. Clearly $\operatorname{del}(v,\Delta)$ is skeleton-E-chordal. By Proposition \ref{prop:SimplicialVertices}, $v$ is skeleton-E-simplicial. We claim that 
$\operatorname{del}(v,\Delta)$ is strongly-connected. In particular, this implies that $v$ is shedding, since the latter property is  equivalent to claiming $\operatorname{del}(v,\Delta)$ pure. To prove the claim, there are two cases to consider:
\begin{compactitem}
\item If $v$ belongs to exactly one $d$-face, call it $v \ast f$. 
Since $\Delta$ has connected dual graph, $\Delta$ contains some $d$-face $G$ adjacent to $v \ast f$. Since $v$ belongs to one $d$-face only, $v\notin G$. Hence,  
$G \cap (v \ast f) = f$.
So $f$ is not a facet in $\operatorname{del}(v,\Delta)$, since $G$  strictly contains $f$. Thus $v$ is shedding and $\operatorname{del}(v,\Delta)$ is pure. By the same argument, all facets $G'$ of $\Delta$ that are adjacent to $v \ast f$, are disjoint from $v$; hence they intersect $v \ast f$ in $f$; hence they are all adjacent to one another. So the removal of the node corresponding to $v \ast f$ from the dual graph of $\Delta$ yields a connected graph. Hence $\operatorname{del}(v,\Delta)$ is strongly-connected.
\item If $v$ belongs to $m >1$ distinct $d$-faces $v*f_1, \ldots, v *f_m$, for each $i \ne j$ in $\{1, \ldots, m\}$ we can choose a vertex $x_{i,j}$ that is in $f_i$ but not in $f_j$, and a vertex $x_{j,i}$ in $f_j$ but not in $f_i$. By the skeleton-E-simpliciality assumption, both $x_{i,j}*f_j$ and $x_{j,i} * f_i$ are $d$-faces of $\Delta$. Since they do not contain $v$, they are $d$-faces of $\operatorname{del}(v,\Delta)$, proving that $f_j$ and $f_i$ are not facets of $\operatorname{del}(v,\Delta)$. Thus $v$ is shedding and $\operatorname{del}(v,\Delta)$ is pure. Now suppose $F$ and $G$ are two $d$-faces of $\operatorname{del}(v,\Delta)$ such that $F$ shares a ridge $f$ with a $d$-face $v \ast f$ of  $\operatorname{star}(v,\Delta)$, and $G$ shares a ridge $g$ with a $d$-face $v \ast g$ of  $\operatorname{star}(v,\Delta)$. We make the subclaim that in the dual graph of  $\operatorname{del}(v,\Delta)$, there is a path connecting (the node corresponding to) $F$ with (the node corresponding to) $G$. In fact, if $f=g$ then $F$ and $G$ are adjacent, and we are done. If $f \ne g$, the skeleton-E-simpliciality of $v$ implies that $\Delta$ contains all size-$(d+1)$ subsets of $f \cup g$. This allows us to form the desired (dual) path, since the $d$-skeleton of a simplex is strongly-connected. Thus the subclaim is proved.
Finally, take any two facets $A, B$ of $\operatorname{del}(v,\Delta)$. Since they are also facets in $\Delta$, which is strongly-connected, there is a walk connecting $A$ to $B$ in the dual graph of $\Delta$. By the subclaim we have just proved,  we may assume that such a walk consists entirely of $d$-faces  of $\operatorname{del}(v,\Delta)$.  Hence,  $\operatorname{del}(v,\Delta)$ is strongly-connected. 
\end{compactitem}
Thus the claim is settled. By inductive assumption, $\operatorname{del}(v,\Delta)$ is vertex-decomposable. To conclude that $\Delta$ is vertex-decomposable, it remains to show that $\link(v, \Delta)$ is vertex-decomposable. Since $\Delta$ is pure, $\link(v, \Delta)$ is pure.  We claim it is strongly-connected. In fact, if $f$ and $g$  are facets of $\link(v, \Delta)$, then $v*f$ and $v*g$ are facets of $\Delta$ containing $v$. By the skeleton-E-simpliciality of $v$,  
$\Delta$ contains  all size-$(d+1)$ subsets of $v \cup f \cup g$ containing $v$. So  $\link(v, \Delta)$ contains all size-$d$ subsets of $f \cup g$.  
This implies the claim. But by Proposition \ref{prop:SuperCLD},  $\link(v, \Delta)$ is skeleton-E-chordal. By induction, $\link(v, \Delta)$ is vertex-decomposable, as desired.
\end{proof}

\begin{corollary} For a pure skeleton-E-chordal simplicial complex $\Delta$, t.f.a.e.: 
\begin{compactenum}[\rm (a)]
\item $\Delta$ is strongly-connected;
\item $\Delta$ is Cohen--Macaulay over some field;
\item $\Delta$ is shellable;
\item $\Delta$ is vertex-decomposable.
\end{compactenum}
\end{corollary}

\begin{proof} The implications  (d) $\Rightarrow$ (c) $\Rightarrow$ (b) $\Rightarrow$ (a) are true for all pure complexes, regardless of the chordality assumption. Theorem \ref{thm:VD0} establishes that (a) implies (d).
\end{proof}

\begin{figure}
    \includegraphics[height=0.23\linewidth]{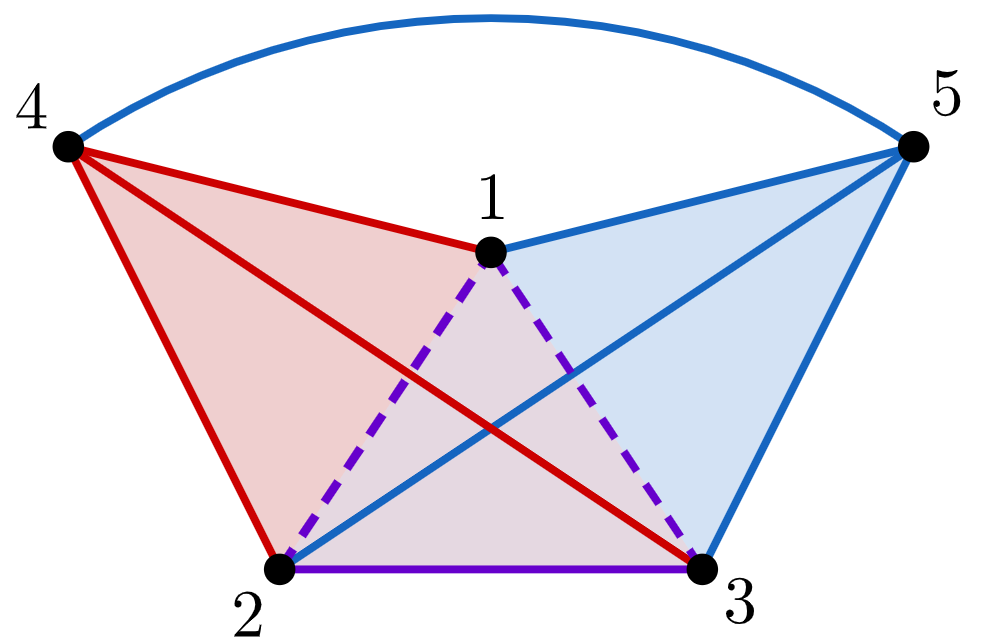}
    \hfill
    \includegraphics[height=0.23\linewidth]{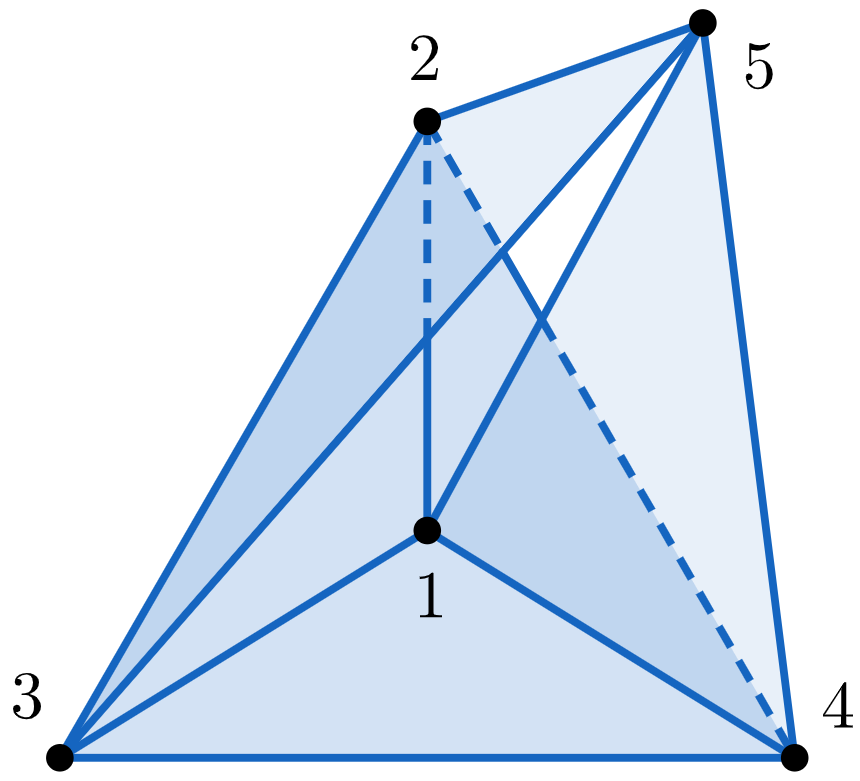}
    \caption{(left) The simplicial complex $T$ from Proposition \ref{prop:DualsAndVd}: the triangle 123 is missing, and every edge of it is coned over the vertices 4 and 5. (right) The simplicial complex $Z$ from Remark \ref{rem:notVDnotShellNotCM}.}
    \label{fig:notVDnotShellNotCM}
\end{figure}

\begin{remark}\label{rem:notVDnotShellNotCM}
    In Theorem \ref{thm:VD0}, the assumption of skeleton-E-chordality cannot be weakened. Figure \ref{fig:notVDnotShellNotCM} illustrates  a strongly-connected simplicial complex $Z$ that is skeleton-mid-chordal: 
    \[ 
    Z = 123,124,134,145,235. 
    \]
    However, this $Z$ is neither vertex-decomposable nor Cohen--Macaulay, since the link of vertex $5$ consists of two disjoint edges, $14$ and $23$.
\end{remark}

\subsection{From Cohen-Macaulayness to chordality} \label{sec:nik}
In this section we obtain an important partial converse statement to Theorem \ref{thm:VD1}, which turns out to be a full converse if we restrict ourselves to subflag complexes. We also review related converse statements by Nikseresht \cite{Nik19} and Bigdeli-Faridi \cite{BF20}.

Before we start, recall that in Theorem \ref{thm:VD1} we exhibited a simplicial complex that has vertex-decomposable Alexander dual, but is not very-weakly-chordal, and in particular, not skeleton-very-weakly-chordal. 

\begin{theorem}\label{thm:VD2}
Let $\Delta$ be a simplicial complex on the vertex set $[n]$. If the Alexander dual $\Delta^{\vee}$ is sequentially Cohen--Macaulay, then $\Delta$ is skeleton-clique-chordal.
\end{theorem}

\begin{proof}
By Remark \ref{rem:CharSkClCh}, it suffices to prove that the graph $G = \operatorname{skel}_1(\Delta)$ is chordal. 
By Duval's theorem \ref{thm:duval}, for each $t$, the pure-$t$-skeleton of $\Delta^{\vee}$ is Cohen--Macaulay. Now let
$I(\overline{G})=(x_i x_j : \{i,j\}\notin G)$ be the edge ideal of the complement graph of $G$. The Alexander
dual complex of this ideal is generated by the complements of the missing edges of $G$:
\[
\Bigl\langle
[n]\setminus\{i,j\} : \{i,j\}\notin G
\Bigr\rangle .
\]
Since $\{i,j\}\notin G$ is equivalent to $\{i,j\}\notin \Delta$, this complex is exactly
$\operatorname{pure-skel}_{n-3}(\Delta^{\vee}),
$
which is Cohen--Macaulay by what we said above. The Eagon--Reiner theorem then implies that $I(\overline{G})$ has a linear resolution, which by Fröberg's theorem is the same as saying that $G$ is chordal.
\end{proof}

\begin{corollary}\label{Cor:subflagSEchordalVD}
For any subflag simplicial complex $\Delta$,  t.f.a.e.: 
\begin{compactenum}[\rm (a)]
\item $\Delta$ is skeleton-clique-chordal;
\item $\Delta$ is skeleton-E-chordal;
\item $\Delta^\vee$ is vertex-decomposable;
\item $\Delta^\vee$ is shellable;
\item $\Delta^\vee$ is sequentially-Cohen--Macaulay over any field;
\item $\Delta^\vee$ is sequentially-Cohen--Macaulay over some field.
\end{compactenum}
\end{corollary}

\begin{proof}
The implication  (b) $\Rightarrow$ (c) is given by Theorem \ref{thm:VD1}; (c) $\Rightarrow$ (d) $\Rightarrow$ (e) $\Rightarrow$ (f) are trivial; the implication  (f) $\Rightarrow$ (a) is Theorem \ref{thm:VD2}. Note that all these implications are true without any assumption on $\Delta$. If we assume that $\Delta$ is subflag, then we also have (a) $\Rightarrow$ (b).
\end{proof}

\begin{corollary}\label{cor:graphDualCM}
For any graph $G$,
\[ G  \textrm{ is chordal } \Longleftrightarrow G^\vee \textrm{ is vertex-decomposable (or shellable)} \Longleftrightarrow G^\vee \textrm{ is seq.-Cohen-Macaulay}.\]
\end{corollary}

\begin{proof}
All graphs are subflag.
\end{proof}

\begin{remark}
Bayer et al.~\cite{B+25} recently defined the \emph{total $k$-cut complex $\Delta_k^t(G)$} of a graph $G$ as the simplicial complex whose  facets are the complements of the size-$k$ independent sets of~$G$. 
In \cite[Theorem 1.3]{B+25}, they proved that $G$ is chordal if and only if for all $k \ge 2$, $\Delta_k^t(G)$ is vertex decomposable.
Now, for $k=2$, one has
\[
\Delta_2^t(G)=\operatorname{Cl}(G)^\vee
=\operatorname{pure\text{-}skel}_{n-3}(G^\vee).
\]
 Thus the restriction of their result to $k=2$ concerns the pure top-dimensional skeleton of $G^\vee$. Corollary~\ref{cor:graphDualCM} is stronger than such restriction: It shows that, for any chordal graph $G$, the entire Alexander dual $G^\vee$ is vertex-decomposable, whether pure or not. ($G$ may have $3$-cycles, which yield additional facets of dimension $n-4$.) On the other hand,  for $k>2$, Bayer et al.'s  \cite[Theorem 1.3]{B+25} neither implies nor is implied by our Corollary~\ref{cor:graphDualCM}.  
\end{remark}

Theorem \ref{thm:VD2} should naturally be compared to the two  converse statements from 2019 and 2020, by Nikseresht and Bigdeli-Faridi, respectively:

\begin{itemize}
\item
\textbf{Nikseresht.} \cite[Theorem 3.10]{Nik19}
If $\Delta$ is pure $d$-dimensional, and $ \operatorname{pure-skel}_{n-d-2}(\Delta^\vee)$ is vertex-decomposable, then $\Delta$ is ridge-chordal.

\item \textbf{essentially Bigdeli--Faridi.} \cite[Theorem 5.4]{BF20} If $\Delta^\vee$ is vertex-decomposable, then $\Delta$ is skeleton-ridge-chordal.
\end{itemize}

The conclusion of Bigdeli--Faridi's theorem is stronger than ours, but also their vertex-decomposability assumption is stronger than our Cohen-Macaulayness one. For convenience, since the differences in the language made it hard for us to understand Nikseresht's argument, we provide a proof of the theorems above in our language:

\begin{proof}[\textbf{Proof of Nikseresht and Bigdeli--Faridi}]
Let $\Delta$ be a pure $k$-dimensional simplicial complex with vertex set contained in $[n]$. Set
\(
\Gamma_k(\Delta)=\pureskel_{n-k-2}(\Delta^\vee).
\)
Its facets are precisely the complements in $[n]$ of the missing $k$-faces of $\Delta$.
We prove Nikseresht's theorem by induction on $n$, simultaneously for all $k\ge1$.
If $\Gamma_k(\Delta)$ is empty or a simplex, then $\Delta$ is the $k$-skeleton of a simplex, possibly with one $k$-face missing. So $\Delta$ is ridge-chordal by Lemma~\ref{lem:CdnFaceChordal} and Theorem~\ref{thm:j-to-jplusone}. Otherwise, let $v$ be a vertex of $\Gamma=\Gamma_k(\Delta)$ such that both $\del(v,\Gamma)$ and $\link(v,\Gamma)$ are vertex-decomposable. Set $L=\link(v,\Delta)$. Let $D$ be the subcomplex generated by the $k$-faces of $\Delta$ avoiding $v$. Thus if $v$ belongs to $\Delta$, $D=\abdel(v,\Delta)$.
Since $\Gamma$ is pure and $v$ is shedding, $\del(v,\Gamma)$ is generated by the facets of $\Gamma$ avoiding $v$. Taking complements in $[n]-\{v\}$, we obtain
\[
\del(v,\Gamma)=\Gamma_{k-1}(L),\qquad
\link(v,\Gamma)=\Gamma_k(D).
\]
If $v$ is not a vertex of $\Delta$, then $D=\Delta$ and we conclude by induction. Otherwise, by induction $D$ is ridge-chordal whenever it has $k$-faces; for $k\ge2$, $L$ is ridge-chordal as well.
We claim that for any $(k+1)$-subset $A\subseteq[n]-\{v\}$,
\begin{equation}\label{eq:NiksereshtShedding}
A\notin\Delta\ \Longrightarrow\
(A-\{y\})\cup\{v\}\notin\Delta
\quad\text{for some }y\in A.
\end{equation}
In fact, $[n]-A$ is a facet of $\Gamma$ containing $v$. Since $v$ is shedding, $([n]-A)-\{v\}$ lies in a facet of $\Gamma$ avoiding $v$, obtained by adding some $y\in A$. Its complement is $(A-\{y\})\cup\{v\}$, which is therefore a missing $k$-face of $\Delta$.
If $k=1$, this shows that any two neighbors of $v$ are adjacent. Hence $v$ is simplicial, and deleting above $v$ leaves $D$, which is ridge-chordal by induction (or has no edges). So we may assume $k\ge2$.
Let $E_1,\ldots,E_t$ be a sequence of simplicial ridges proving $L$ ridge-chordal. Starting with $\Delta_0=\Delta$, successively delete above $R_i=E_i\cup\{v\}$, and let $L_i=\link(v,\Delta_i)$. By construction, $L_i=\abdel(E_i,L_{i-1})$, while the $k$-faces avoiding $v$ remain unchanged. We claim that each $R_i$ is a $k$-simplicial ridge of $\Delta_{i-1}$.
Indeed, the vertices in the star of $R_i$ in $\Delta_{i-1}$ are those in the star of $E_i$ in $L_{i-1}$, together with $v$. Let $A$ be any $(k+1)$-subset of these vertices. If $v\in A$, then $A-\{v\}$ is in $L_{i-1}$ by the simpliciality of $E_i$, and thus $A\in\Delta_{i-1}$. Otherwise, every $k$-subset of $A$ belongs to $L_{i-1}\subseteq L$. If $A\notin\Delta$, (\ref{eq:NiksereshtShedding}) contradicts this. Hence $A\in\Delta$; since $A$ avoids $v$, it also belongs to $\Delta_{i-1}$. This proves the claim. After these deletions, no $k$-face contains $v$, so $\Delta_t=D$. By induction, Nikseresht's theorem is proven.

Now suppose that $\Delta^\vee$ is vertex-decomposable, and fix $1\le k\le\dim\Delta$. If $\Delta$ has no missing $k$-face, its $k$-skeleton is complete, hence ridge-chordal. Otherwise, by Woodroofe's \cite[Lemma~3.10]{Woo11}  $\pureskel_{n-k-2}(\Delta^\vee)$ is vertex-decomposable. Since
$\pureskel_{n-k-2}(\Delta^\vee)
=\Gamma_k(\pureskel_k(\Delta))$, by
Nikseresht's theorem  $\pureskel_k(\Delta)$ is ridge-chordal. The two complexes have the same $k$-faces. Since each ridge used in the reduction belongs to a $k$-face, it cannot be contained in any additional lower-dimensional facet of $\skel_k(\Delta)$. Thus its star vertices are the same in both complexes, even after each deletion. Hence the same sequence proves $\skel_k(\Delta)$ ridge-chordal. As this holds for every $k$, $\Delta$ is skeleton-ridge-chordal. \qedhere
\end{proof}

\begin{remark}
The triangulation of the real projective plane from Proposition \ref{prop:FCtoGeo} is not ridge-chordal, but its Alexander dual is Cohen-Macaulay over any field of characteristic $\ne 2$, so the conclusion of Theorem \ref{thm:VD2} cannot be strengthened to ``$\Delta$ is skeleton-ridge-chordal''. A characteristic-free example  whose Alexander dual is even shellable is given in \cite{BB21}. Caveat: The notion of ``clique complex'' in  \cite{BB21} is not the same as the one we use here, and produces complexes that are typically not flag. In fact, putting together Bigdeli-Faridi's theorem with our Corollary \ref{Cor:subflagSEchordalVD}, one immediately obtains that for any subflag simplicial complex, $\Delta$ is skeleton-ridge-chordal if and only if $\Delta^\vee$ is vertex decomposable. In particular, any non-ridge-chordal complex whose Alexander dual is shellable cannot be (sub)flag.
\end{remark}

%%%%%%%%%%%%%%%% SUBSECTION 6.3 %%%%%%%%%%%%%%%%%%%%%%%
%\subsection{From Cohen--Macaulayness to geochordality}
%%%%%%%%%%%%%%%% SUBSECTION 6.3 %%%%%%%%%%%%%%%%%%%%%%%

\medskip
We now strengthen Theorem \ref{thm:VD2} by providing a simple direct proof of the fact that if $\Delta^\vee$ is sequentially Cohen--Macaulay, then  $\Delta$ is geochordal. This result, which also follows in the pure case from Connon--Faridi's work \cite[Theorem~6.1(1)]{CF13}, is stronger  than Theorem \ref{thm:VD2} because all geochordal complexes are skeleton-clique-chordal (Theorem \ref{thm:hierarchy2}). 
Our proof requires no purity assumption on $\Delta$, and does not use resolutions or derived functors.

\begin{lemma}\label{lem:pure-skeleton-dual-sphere}
Let $S$ be a simplicial complex homeomorphic to the $d$-sphere. Let $W$ be the vertex set of $S$. Assume $m=|W|>d+2$. 
Let $\Gamma=\pureskel_{m-d-2}(S^\vee)$, where the Alexander dual is taken inside the vertex set $W$. Then the Alexander dual $\Gamma^\vee$ (with respect to $W$) has no faces of dimension larger than $d$, and its $d$-faces are exactly the $d$-faces of $S$. Hence over any field $\kk$,
\[
\widetH_d(\Gamma^\vee;\kk) \: \cong \: \widetH_d(S;\kk)\:\cong  \: \kk.
\]
\end{lemma}

\begin{proof} Since $S$ is a manifold without boundary, every ridge of it is contained in exactly two $d$-faces. We claim that there is no $T\subseteq W$ of size $d+2$ such that every size-$(d+1)$ subset of $T$ is a face of $S$. Indeed, if such a $T$ existed, then $S$ would contain the boundary of the $(d+1)$-simplex $\Sigma_T$ with vertex set  $T$.  But the ridges of $\partial \Sigma_T$ inside $\partial \Sigma_T$ already have two $d$-faces containing them. This would force $S=\partial \Sigma_T$, contrary to the assumption. So the claim is proven.
Hence, every subset $A\subseteq W$ of size $|A|\ge d+2$ contains a $d$-dimensional non-face of $S$, i.e. a subset $F\subseteq A$ with $|F|=d+1$ and $F\notin S$.
By definition, the facets of $\Gamma$ are precisely the sets $W- F$, where $F$ is a $d$-dimensional non-face of $S$. So for any $A\subseteq W$,
\[
A\in \Gamma^\vee
\quad\Longleftrightarrow\quad
(W -  A) \; \notin \Gamma
\quad\Longleftrightarrow\quad
A \text{ contains no $d$-dimensional non-face of } S.
\]
By the previous paragraph, no subset of $W$ of size $\ge d+2$ can belong to $\Gamma^\vee$. Thus, $\Gamma^\vee$ has no faces of dimension greater
than $d$.
Now, let $A\subseteq W$ with $|A|=d+1$. Then $A\in\Gamma^\vee$ if and only if $A$ is not a $d$-dimensional non-face of $S$, which is equivalent to $A\in S$. Hence, the $d$-faces of $\Gamma^\vee$ and of $S$ are the same. 
Since $\Gamma^\vee$ has no faces above dimension $d$ and has exactly the same $d$-faces as $S$, 
\[
\widetH_d(\Gamma^\vee;\kk)\cong \widetH_d(S;\kk)\cong \kk. \qedhere
\]
\end{proof}

\begin{theorem} \label{thm:CMtoGeochordal}
Let $\Delta$ be a simplicial complex of dimension at least $d$ on the vertex set $[n]$.  \\
If $\pureskel_{n-d-2}(\Delta^\vee)$ is Cohen--Macaulay over some field $\kk$, then $\Delta$ is geometrically-$d$-chordal.
\end{theorem}

\begin{proof} 
Let $W\subseteq [n]$ be the vertex set of a subcomplex $S$  homeomorphic to the $d$-sphere.  Suppose $S$ has the property that every $d$-face of $\Delta$ with all vertices in $W$, is in $S$.
Set $m=|W|$. By contradiction, suppose also $m > d+2$. 
As in the proof of Lemma \ref{lem:pure-skeleton-dual-sphere}, $S$ has  a $d$-dimensional non-face
$F\subseteq W$. So $|F|=d+1$ and $F\notin S$. By the assumption on $S$, $F\notin \Delta$. So $[n] - F$ is a face of $\Delta^\vee$ of size $n-d-1$.
Now set $U=[n]- W$. (We are not excluding the possibility that $U$ may be empty.) Since $F \subseteq W$, and $F \notin \Delta$ also $W \notin \Delta$. Thus, $U=[n]-W$ is a face of $\Delta^{\vee}$. Moreover, 
\[
[n]-F=([n] - W) \cup (W - F) =  U \cup (W - F).
\]
Thus $U$ is a face of $M:=\pureskel_{n-d-2}(\Delta^\vee)$. We claim that
\[
\link(U,M)=\pureskel_{m-d-2}(S^\vee),
\]
where the Alexander dual is taken  with respect to $W$. In fact, the facets
of $\link(U,M)$ are the sets $A\subseteq W$ such that
$U\cup A$ is an $(n-d-2)$-face of $\Delta^\vee$. Such an $A$ has size
$m-d-1$, and
\[
U\cup A\in \Delta^\vee
\quad\Longleftrightarrow\quad
[n]-(U\cup A)=W- A\notin \Delta
\quad\Longleftrightarrow\quad
W-A\notin S
\quad\Longleftrightarrow\quad
A\in S^\vee.
\]
So the facets of $\link(U,M)$ are exactly the $(m-d-2)$-faces of $S^\vee$,
which proves the claim.
Set
\[
\Gamma:=\link(U,M)=\pureskel_{m-d-2}(S^\vee).
\]
By Lemma~\ref{lem:pure-skeleton-dual-sphere}, $ \widetH_d(\Gamma^\vee;\kk)\cong \kk.$
Since we work over the field $\kk$, the universal coefficient theorem for cohomology \cite[Theorem~3.2]{Hatcher} applies to every finite simplicial complex. Here it gives
\begin{equation} \label{eq:antibuchsbaum}
\widetilde H^d(\Gamma^\vee;\kk)
\cong
\operatorname{Hom}_{\kk}\bigl(\widetH_d(\Gamma^\vee;\kk),\kk\bigr)
\cong \kk.
\end{equation}
Applying combinatorial Alexander duality \cite[Theorem~1.1]{BT09} to the complex $\Gamma$, we conclude
\[
\widetH_{m-d-3}(\Gamma;\kk)
\cong
\widetilde H^d(\Gamma^\vee;\kk)
\cong \kk.
\]
But we still have to use the assumption that $M$ is Cohen--Macaulay over $\kk$. Since $S$ has a $d$-dimensional non-face,
$\Gamma$ has dimension $m-d-2$. By the Cohen--Macaulay condition, 
\[
\widetH_i(\Gamma;\kk)=0
\qquad \text{for all } i<m-d-2.
\]
In particular, $ \widetH_{m-d-3}(\Gamma;\kk)=0$:
A direct contradiction with Equation (\ref{eq:antibuchsbaum}) above.
Therefore the assumption $m>d+2$ is false. Thus $m=d+2$. Hence, $S$ must be combinatorially equivalent to the boundary of the $(d+1)$-simplex. So $\Delta$ is geometrically-$d$-chordal.\qedhere
\end{proof}

\begin{remark}
    The assumptions of Theorem \ref{thm:CMtoGeochordal} can be considerably weakened. For example, the same proof shows also the following statement: Let $\Delta$ be a simplicial complex of dimension at least $d$ on the vertex set $[n]$. Set $M=\operatorname{pure-skel}_{n-d-2}(\Delta^\vee)$. 
    If there is a field $\kk$ such that $M$ is Buchsbaum over $\kk$ and $\widetilde H_{n-d-3}(M;\kk)\not\cong \kk$, then $\Delta$ is geometrically-$d$-chordal.
\end{remark}

\begin{corollary}\label{cor:pure-skeleton-geochordality}
If $\Delta^\vee$ is sequentially-Cohen--Macaulay over some field $\kk$, then $\Delta$ is
geochordal. 
\end{corollary}

\begin{proof}
By Duval's theorem \ref{thm:duval}, $
\pureskel_{t}(\Delta^\vee)$ is Cohen--Macaulay over $\kk$ for all $t$. Fix $d \ge 1$. Set $P:=
\pureskel_{n-d-2}(\Delta^\vee)$. If $\dim P =n-d-2$, then by Theorem~\ref{thm:CMtoGeochordal}  
$\Delta$ is geometrically-$d$-chordal. If instead $P$ is empty,
then $\Delta^\vee$ has no faces of dimension $n-d-2$; that is,  $\Delta$ has no
$d$-dimensional non-faces. Hence $\Delta$ contains the complete $d$-skeleton on its vertex
set. So $\Delta$ is geometrically-$d$-chordal trivially.
Either way, $\Delta$ is geometrically-$d$-chordal for every $d$. 
\end{proof}

\begin{remark}
The sequential-Cohen--Macaulayness of $\Delta^\vee$ is implied by
skeleton-E-chordality; in turn, by Theorem \ref{thm:VD2} and Corollary
\ref{cor:pure-skeleton-geochordality}, it implies geochordality, which in turn implies skeleton-clique-chordality. Apart from this, the (sequential)-Cohen--Macaulayness of $\Delta^\vee$ neither implies, nor is implied by, any of the other properties mentioned in Theorem \ref{thm:hierarchy2}. In fact: \begin{compactitem} 
\item in Remark \ref{remark:E8dual} we saw a skeleton-mid-chordal simplicial complex $I$ with the property that  $\operatorname{pure-skel}_{n-d-2}(I^\vee)$ is not Cohen--Macaulay. 
\item The $5$-cycle is Cohen-Macaulay over any field, but its Alexander dual is the simplicial complex $(C_5)^\vee = 124,134,135,235,245$ that we encountered in the proof of Theorem \ref{thm:VD1}. Such simplicial complex is not skeleton-very-weakly-chordal.
\end{compactitem}
\end{remark}

\subsection*{Acknowledgments}
The first author is supported by a Simons grant MPS-TSM-00002873. The second author is supported by the VILLUM Foundation, grant no. 10059. The authors wish to thank Davide Bolognini, Anton Dochtermann, Bennet Goeckner, Lisa Seccia, and Matteo Varbaro, for helpful comments and conversations.

%%%%%%%%%%%%%%%%%%%%%%%%%%%%%%%%%%%%%%%%%%%%%%%%%%%%%%%%%%%%%%%%%%%%%%%%%

\newpage 

\end{document}